\documentclass[11pt]{article}

\usepackage{enumitem}

\usepackage[utf8]{inputenc}
\usepackage[T1]{fontenc}
\usepackage{xcolor,graphicx}
\usepackage{transparent}
\usepackage{eepic}
\usepackage{fouriernc}
\usepackage[english]{babel}
\usepackage{amsfonts, amssymb, amsmath, amsthm}

\usepackage{tocbibind}
\usepackage{mathrsfs}
\usepackage{cancel}

\usepackage{amsmath}
\numberwithin{equation}{section}

\usepackage{graphicx}
\usepackage{color}
\usepackage{graphicx}
\usepackage[colorinlistoftodos]{todonotes}
\usepackage[colorlinks=true, allcolors=blue]{hyperref}

\usepackage{array} % for the table at the end
\usepackage{amsfonts}
\usepackage{amsmath}
\usepackage{amssymb}
\usepackage{graphicx}
\usepackage{float} % to put figures at the right point option  [H]
\usepackage{color}
\usepackage{esint}    % for averages

\usepackage{enumitem} % for enumerate/itemize margins

\usepackage{multirow}

\newtheorem{thm}{Theorem}[section]
\newtheorem{cor}[thm]{Corollary}
\newtheorem{lem}[thm]{Lemma}
\newtheorem{prop}[thm]{Proposition}
\newtheorem{defn}[thm]{Definition}
\newtheorem{rem}[thm]{Remark}
\newtheorem{nota}[thm]{Notation}
\numberwithin{equation}{section}

\newcommand{\dx}{\,{\rm d}x}
\newcommand{\dy}{\,{\rm d}y}

\newcommand{\dt}{\,{\rm d}t}
\newcommand{\dr}{\,{\rm d}r}

\def\LL{\mathrm{L}} %per gli spazi L^p
\newcommand{\dtau}{\,{\rm d}\tau}

\newcommand{\A}{\mathcal{L}}
\newcommand{\AM}{\mathcal{L}^{\frac{1}{2}}}
\newcommand{\AI}{\mathcal{L}^{-1}}

\newcommand{\n}{F}
\newcommand{\nl}{{F^*}}

\newcommand{\ph}{{\Phi_1}}

\newcommand{\ud}{{u_\delta}}

\renewcommand{\k}{\kappa}
\newcommand{\K}{{\mathbb G}}

\newcommand{\solution}{{minimal weak dual solution}}
\newcommand{\solutions}{{minimal weak dual solutions}}

\newcommand{\RR}{\mathbb{R}}
\newcommand{\RN}{\mathbb{R}^N}

\newcommand{\B}{\mathcal{B}}

\def\dist{\mathrm{dist}} %per la distanza
\def\dom{\mathrm{dom}} %per il dominio

\def\qed{\,\unskip\kern 6pt \penalty 500
	\raise -2pt\hbox{\vrule \vbox to8pt{\hrule width 6pt
			\vfill\hrule}\vrule}\par}

\def\quotient#1#2{\raise1ex\hbox{$#1$}\Big/\lower1ex\hbox{$#2$}}

\definecolor{darkblue}{rgb}{0.05, .05, .65}
\definecolor{darkgreen}{rgb}{0.1, .65, .1}
\definecolor{darkred}{rgb}{0.8,0,0}

\begin{document}
	\title{\textbf{Sharp Boundary Estimates and Harnack Inequalities for Fractional Porous Medium type Equations}}
	
	\author{\Large Matteo Bonforte$^{a,b}$
		~and~ Carlos Fuertes Mor\'{a}n$^{a,c}$
	}
	\date{} %%  this cancels date in article format
	
	\maketitle
	
	%\nocite{*}
	
	\begin{abstract}
		This paper provides sharp quantitative and constructive estimates of nonnegative solutions $u(t,x)\geq 0$  to the nonlinear fractional diffusion equation, $$\partial_t u +\A F(u)=0,$$ also known as filtration equation, posed in a smooth bounded domain $x\in \Omega \subset \RN$ with suitable homogeneous Dirichlet boundary conditions. Both the operator $\A$ and the nonlinearity $F$ belong to a general class. The assumption on the operator $\A$ are set in terms of the kernel of $\A$ and $\AI$, depending on the result, and allow for operators with degenerate kernel at the boundary of $\Omega$. The main examples of $\A$ are the three different Dirichlet Fractional Laplacians on bounded domains, and the nonlinearity can be non-homogeneous, for instance, $F(u)=u^2+u^{10}$. Previous result were known in the porous medium case, i.e. $F(u)=|u|^{m-1} u$ with $m>1$, whose homogeneity allows to considerably simplify the proofs. We take the opportunity of exposing a complete basic theory of existence, uniqueness and boundedness for a quite general class of weak (dual) solutions, scattered in previous papers by Vazquez and the first author. Our aim here is to perform the next step: a delicate analysis of regularity through quantitative, constructive (all the constants can be explicitly computed) and sharp a priori estimates. Our first main results are global Harnack type inequalities of the form
		$$	H_0(t,u_0)\, {\rm dist}(x, \partial \Omega)^a\leq F(u(t,x))\leq H_1(t)\, {\rm dist}(x, \partial \Omega)^b\qquad\forall (t,x)\in (0,\infty)\times \overline{\Omega},$$
where the expressions of $H_0, H_1$ and $a,b$ are explicit and change according to the choice of operator $\A$ and nonlinearity $F$. The sharpness of such estimates is proven by means of examples and counterexamples: on the one hand, we can match the powers (i.e. $a=b$) when the operator has a non degenerate kernel. On the other hand,  when the operator $\A$ has a kernel that degenerates at the boundary $\partial\Omega$, there appear an intriguing anomalous boundary behaviour: the size of the initial data determines the sharp boundary behaviour of the solution, which can be different for ``small'' and ``large'' initial data. We conclude the paper with higher regularity results: solutions are always H\"older continuous in the interior, and even classical when the operator allows it.
	\end{abstract}
	
	%\vskip 1 cm
	
	%\vfill
\small
	
	\noindent {\sc Addresses:}\vspace{-3mm}\begin{itemize}[leftmargin=0pt]\itemsep0pt \parskip0pt \parsep0pt

        \item[(a)]Departamento de Matem\'{a}ticas, Universidad Aut\'{o}noma de Madrid,\\
		ICMAT - Instituto de Ciencias Matem\'{a}ticas, CSIC-UAM-UC3M-UCM, \\
		Campus de Cantoblanco, 28049 Madrid, Spain.
		\item[(b)]
		E-mail:\texttt{~matteo.bonforte@uam.es }
		Web-page:  \texttt{http://verso.mat.uam.es/\~{}matteo.bonforte}
		\item[(c)]
		E-mail:\texttt{~carlos.fuertesm@uam.es}		
	\end{itemize}

	

\smallskip\noindent {\sl\small\copyright~2025 by the authors. This paper may be reproduced, in its entirety, for non-commercial purposes.}

	\newpage
	
	\tableofcontents
	\normalsize
	
	\newpage
	%%%%%%%%%%%%%%%%%%%%%%%%%%%%%%%%%%%%%%%%%%%%%%%%%%%%%%%%%%%%%%%%%%
	
	\section{Introduction}\label{Sec1}
	
	In this paper we study fine regularity and pointwise properties of a suitable class of weak solutions to the nonlinear diffusion equation
	\begin{equation}\label{PM}
		\partial_t u + \A F(u)=0 \hspace{4mm}\mbox{ in }(0,\infty)\times \Omega\,,
	\end{equation}
posed on a  bounded domain $\Omega\subset \RN$ assumed to be of class $C^{1,1}$, and $N\ge 1$. We focus our attention on a number of quantitative a priori estimates of such solutions: local and global Harnack inequalities, (sharp) positivity and boundary behavior estimates, and interior regularity.

Here, $\A$ is a linear operator belonging to a quite large class, that we shall characterize by means of two parameters $s,\gamma\in (0,1]$, and that allows to study both local and nonlocal (fractional) equations. The nonlinearity $F$ is any non-decreasing monotone continuous function satisfying suitable assumptions, that allow for degeneracies, namely $F(0)=F'(0)=0$.

The prototype example, is given by the so-called (Fractional) Porous Medium equation, with the homogeneous nonlinearity $F(u)=|u|^{m-1}u$ with $m>1$,  which has been extensively studied in the last decade \cite{MB+AF+XR_positivity-and-regularity, MB+AF+JLV_sharp-global-estimates, MB+JLV+YS_exist-and-unique, DEPAB+QUI+ANA+JLV_fpme,DEPAB+QUI+ANA+JLV_general-fpme}. Here, we consider more general, possibly non-homogeneous nonlinearities, which introduce several new technical difficulties. These class of equations are often called Generalized Porous Medium equations or Filtration equations in literature, see \cite{ANDR+DIB_filtration, Daska+Kenig, AdP+FQ+AR_2016, AdP+FQ+JRC_2024, DIB+GIA+VES_harnineq, JLV_pmbook, AdP+FQ+AR+JLV_2017}. These equations allow to model a wide range of phenomena appearing in more applied sciences, like Physics, Biology, Engineering, Finance, etc. The use of nonlocal operators in diffusion equations reflects the need to model the presence of long-distance effects not included in evolution driven by the Laplace or any other local operator. When $\A$ is a local operator, the model has been extensively used since the mid 50s, see the monograph \cite{JLV_pmbook} and also \cite{JLV_dirichletproblem-pm}. On the other hand, the motivation and relevance of nonlinear diffusion models where $\A$ is a nonlocal operator, has been mentioned in many references; see for instance \cite{ATH+CAFF_2010, MB+JLV_quan-locandglo-aprioriestimates, MB+JLV_harnack-inequality, AdP+FQ+AR_2018, DEPAB+QUI+ANA+JLV_fpme, DEPAB+QUI+ANA+JLV_general-fpme, JLV_barenblatt} and the surveys \cite{JLV_survey-diff-interac, JLV_survey-fraclap}. Because in most of the applications $u$ usually represents a density, throughout the paper all data and solutions are supposed to be nonnegative.

	To better understand the effects of degeneration on the behaviour of the solutions, we impose homogeneous Dirichlet boundary conditions. As already mentioned, the prototype equation is the (possibly fractional) Porous Medium Equation
\[
u_t + (-\Delta)^s u^m=0\,,
\]
with $m>1$, $s\in (0,1]$, initial data $u_0\ge 0$ and $u=0$ on the lateral boundary\footnote{The boundary depends on the particular choice of operators, for instance if $\A=-\Delta$ we consider the topological boundary of $\Omega$, while for nonlocal operators, such as the Restricted Fractional Laplacian the natural lateral boundary is the complement of $\Omega$, namely $\RN\setminus \Omega$. Many more examples can be considered, see Section \ref{sec7} for more details.}. The nonlinearity $F$ need not to be homogeneous and it is allowed to be more general than a single power: our assumption will allow for non-decreasing $C^1(\RR\setminus\{0\})$ functions, trapped between two different powers both at zero and at infinity. This causes a number of extra difficulties in the proofs and also new behaviors (typically it produces different results for small and large times) that depart from the prototype case in a quantitative way. For technical reasons, we shall assume $N>2s$.

Note that, since we are in a bounded domain, the fractional Laplacian can take several non-equivalent forms, see Section \ref{sec7} and references therein, that affect also the type of lateral boundary conditions that we have to impose. The basic theory - existence, uniqueness and boundedness - of weak dual solutions for the equation \eqref{PM} have been studied in \cite{MB+JLV+YS_exist-and-unique, MB+JLV_harnack-inequality, MB+JLV_PM-with-F}. To the best of our knowledge, weak dual solutions represent the biggest class of nonnegative solutions known so far, allowing for ``big'' initial data, which may blow up at the boundary in a non integrable way. This complicates in a significant way the study of the boundary behaviour of these class of solutions. Our ambitious goal in this direction is to prove sharp pointwise boundary estimates both from above and from below. Such bounds are formulated in terms of $\Phi_1$, the first eigenfunction of the operator, which typically satisfies $\Phi_1 \asymp \dist(\cdot, \partial \Omega)^\gamma$, with $\gamma \in (0,1]$. The operator $\A$ is characterized in terms of the parameters $s$ and $\gamma$: It is clear now that the parameter $s$ correspond to the (possibly fractional) order of derivation in space, while $\gamma$ is related to the boundary behaviour. Notice that a large number of operators fall into this class, in particular, the three different possible Fractional Laplacians: Restricted, Spectral and Censored. We refer to \cite{MB+AF+JLV_sharp-global-estimates, MB+JLV_PM-with-F} and to Section \ref{sec7} for a list of examples.
	
\noindent\textit{Global pointwise inequalities and anomalous boundary behaviour. } The first set of main results, take the form of global Harnack-type inequalities, nowadays known as Global Harnack Principle (GHP), which can be roughly stated in the present setting, as
	\begin{equation}\label{ec}
		H_0(t,u_0) \Phi_1(x)^a \leq F(u(t,x))\leq H_1(t) \Phi_1(x)^b\qquad\mbox{for all $x\in \overline{\Omega}$ and all $t\ge t_*$},
	\end{equation}
where the expressions of $H_0, H_1$ and $a,b$ are explicit and change according to the choice of operator $\A$ and nonlinearity $F$, see Section \ref{sec2.4} for precise statements. The so-called waiting time $t_*(u_0)$ has an explicit expression, and in the local case ($s=1$) cannot be avoided, because of the possible finite speed of propagation\footnote{Meaning that compactly supported data produce compactly supported solutions for all times.}. An important role is played by the exponents $a$ and $b$: in some cases they can match, hence they are optimal and determine a precise boundary behaviour of (all nonnegative) solutions, and control  the sharp boundary regularity\footnote{Here we do not show boundary regularity for inhomogeneous $F$, but in \cite{MB+AF+JLV_sharp-global-estimates} sharp boundary regularity is established in the case $F(u)=|u|^{m-1}u$, with $m>1$.}. In some cases, the powers $a$ and $b$ do not match, and this is inevitable: we provide counterexamples that show that some initial data produce ``small'' solutions with boundary behaviour $\Phi_1^a$, while other data can produce ``larger'' solutions whose behaviour at the boundary is given by $\Phi_1^b$. Hence, also in this case, our global result turn out to be sharp. We refer to Remark \ref{rmk.sharpness} for a more detailed explanation. This anomalous boundary behaviour has been discovered in \cite{MB+AF+JLV_sharp-global-estimates}, for the case of homogeneous nonlinearities, where also numerical simulations confirm the theoretical results. It has to be mentioned that these kind of results are fundamental in the numerical implementation, indeed without the theoretical prediction, a priori one could think of a failure in the numerical method, see for instance the paper \cite{CUS+FDT+GER-GIOR+PAG_discre-sfl}, which was inspired by the results of  \cite{MB+AF+JLV_sharp-global-estimates}. Note that this anomalous boundary behaviour does not happen when the operator is local, nor in elliptic problems: it is a phenomenon typical of nonlinear, nonlocal and degenerate parabolic equations. Last but not least, the presence of a possibly non-homogeneous nonlinearity complicates the panorama, since different behaviours can appear for small and large times (which in turn depend on the size of the initial data).

\noindent\textit{Finite VS infinite speed of propagation. }A big difference between local and nonlocal diffusions can be appreciated in the validity of the above bounds: indeed, when we restrict ourselves to the nonlocal case, i.e. when $s<1$, the positivity estimate above holds for all $t>0$: this implies infinite speed of propagation, which is in clear contrast with the local case, for which finite speed of propagation happens, \cite{JLV_pmbook}. These phenomenon were observed for the first time in the case of the prototype nonlinearity $F(u)=u^m$, with $m>1$, in \cite{MB+AF+XR_positivity-and-regularity, MB+AF+JLV_sharp-global-estimates}.

\noindent\textit{Local estimates and regularity. }As a consequence of the above global Harnack estimates, we can derive more classical forms of local Harnack inequalities, which can be of backward/forward/elliptic type:
	$$\sup_{x\in B_R(x_0)}  F(u(t,x))  \leq H_2(t,u_0) \inf_{x\in B_R(x_0)} F(u(t\pm h,x)),$$
where again $H_2$ has an explicit form and $0<h\lesssim t\wedge t_*$. For linear local parabolic equations, only forward Harnack inequalities typically hold\footnote{It is well known in the case of the Cauchy problem on the whole space for the classical heat equation, it can be checked on the Gaussian. On the other hand, for the Dirichlet problem for the standard heat equation, backward Harnack inequalities were proved by Fabes, Garofalo and Salsa in \cite{FAB+GAR+SAL_harnineq-fatou}. In the case of the fractional heat equation on the whole space, backward/forward/elliptic Harnack have been proved in \cite{MB+JLV+YS_opt-exisanduniq-heatequ}.}and here, it is quite remarkable that the supremum and infimum can be taken at the same time or even at a previous one\footnote{In local nonlinear equations, this is a phenomenon typical of the fast diffusion, $m<1$ when the nonlinearity is inhomogeneous, but only forward Harnack inequalities hold in general for the degenerate case under consideration \cite{DIB_deg-para-equa, DIB+GIA+VES_harnineq}. Here, we are considering the Dirichlet problem, and we show how also in the local case, after some waiting time, backward/forward/elliptic Harnack inequalities hold. }.
\normalcolor
The above local and global Harnack estimates are the key to show regularity of solutions. Indeed, once solutions are bounded and positive, it is possible to show that they gain interior regularity. We focus here on the nonlocal case\footnote{The results in the local case follows in analogous way and holds true after the waiting time, but since the local result are nowadasy ``classical'', we decide to concentrate the efforts in the nonlocal part.} $s<1$: first we show that solutions are always H\"older continuous in the interior, and even classical, whenever the operator allows for it, see Section \ref{sec6} for more details.

\noindent\textit{Similarities and differences with the case of power $F(u)=u |u|^{m-1}$. }This paper is the natural continuation of \cite{MB+JLV_PM-with-F}, as the latter examines the same equation addressed in the present work, also considering a general nonlinearity $F$. In particular, the existence and uniqueness of a class of solutions—namely, the minimal weak dual solutions—are established, along with a number of quantitative a priori estimates, among which there are absolute upper bounds, pointwise estimates, smoothing effects, and weighted $\LL^1$-estimates, that we shall use systematically throughout the paper. For the reader's convenience, whenever necessary, we will state these results alongside some key observations to facilitate a better understanding.

One of the main goal of this paper is to determine the behavior near the boundary of solutions to \eqref{PM}, with homogeneous Dirichlet boundary conditions. The answer to this question can be expressed in the form \eqref{ec}, and the boundary behaviour is quantified in terms of $F^{-1}(\Phi_1^c)$, where $\Phi_1$ denotes the first eigenfunction of $\A$ and $c\in (0,1]$. We adapt the ideas of \cite{MB+AF+JLV_sharp-global-estimates} to the present case, but the lack of homogeneity in the nonlinearity $F$ significantly complicates the situation: many technical issues arise and must be solved through new alternative methods, we shall try to give a rough idea below.

The absence of an explicit solution, such as the separation of variables in the power case, prevents us from showing the long-term behavior\footnote{For general $F$ under our assumption, the asymptotic behavior is still an open problem even in the case of the Dirichlet problem for the classical GPME: $u_t=\Delta F(u)$.}, because of the lack of improved lower bounds obtained by comparison with separate  variables solutions. Moreover, the explicit expression of the waiting time $t_*$, may vary depending on the size of the initial data, measured in terms of  a suitable $L^1$-weighted norm. As a consequence, when determining the various lower quantitative estimates for positivity, the panorama gets more involved and we obtain different bounds in each case. Although global Harnack-type estimates of type \eqref{ec} can be obtained with matching powers, the lack of homogeneity of $F$ does not allow us to obtain sharp regularity up to the boundary, since rescaling techniques fail in this case. New techniques and ideas are needed, and we leave this issue as an open problem. However, our estimates give indications of the maximal boundary regularity in space, as in the case of homogeneous nonlinearities. For analogous reasons, it turns out to be delicate to analyze the regularity in time, which again we leave as an open problem.

\noindent\textbf{Plan of the paper.} Section 2 is dedicated to formulating the Cauchy-Dirichlet problem. We impose homogeneous boundary conditions along with the appropriate space of non-negative initial data. Additionally, we present the basic properties required for the operators $\A$ and the nonlinearities $F$ to which the techniques employed can be applied. We also recall and review how the (unique) weak dual solutions are constructed. Finally, we present the precise statements of our main results: global Harnack-type inequalities and interior regularity estimates.
	
	We proceed to demonstrate our main results throughout Sections \ref{sec3}, \ref{sec4}, and \ref{sec5}. In Section \ref{sec3}, we begin by recalling a series of technical results of \cite{MB+AF+JLV_boundary-estimates-elliptic, MB+JLV_PM-with-F, CRA+PIE} that we will use in the proofs.  In Section \ref{sec5}, we will prove the different versions of the lower bounds of the GHP. To this end, we need weighted $\LL^1$ estimates proven in Section \ref{sec4}, and a constructive proof by contradiction that generalizes the one introduced for the first time in  \cite{MB+AF+XR_positivity-and-regularity} and \cite{MB+AF+JLV_sharp-global-estimates}. Section \ref{sec6} contains the proof of the interior regularity results.  Finally, in Section \ref{sec7}, we present some examples of different operators to which all the techniques used throughout the work can be applied.

	\begin{nota}\label{notacion}
		The symbol $\infty$ means $+\infty$, when we write $\Omega$ it is always a bounded domain with boundary smooth enough, at least of class $C^{1,1}$. The distance to the boundary will be denoted by $d(x):=\dist(x,\partial\Omega)|_\Omega$, for $x\in \Omega$. Two quantities will be comparable, $a\asymp b$, iff there exists constants $c_0, c_1>0$ such that $c_0 a \leq b \leq c_1 a$. We represent the maximum or the minimum as follows, $a\land b = \min\left\{ a,b \right\}$ and $a\lor b = \max\left\{ a,b \right\}$. Finally, let's define, for $\gamma\in (0,1]$, the following parameters that will appear in the estimates
		$$m_i=\frac{1}{1-\mu_i},\mbox{ with }\mu_i\in (0,1)\qquad\mbox{and}\qquad\sigma_i=\left( 1 \land \frac{2s m_i}{\gamma (m_i-1)} \right).$$
	\end{nota}
	
	%%%%%%%%%%%%%%%%%%%%%%%%%%%%%%%%%%%%%%%%%%%%%%%%%%%%%%%%%%%%%%%%%%%%%
	%%%%%%%%%%%%%%%%%%%%%%%%%%%%%%%%%%%%%%%%%%%%%%%%%%%%%%%%%%%%%%%%%%%%%
	%%%%%%%%%%%%%%%%%%%%%%%%%%%%%%%%%%%%%%%%%%%%%%%%%%%%%%%%%%%%%%%%%%%%%
	%%%%%%%%%%%%%%%%%%%%%%%%%%%%%%%%%%%%%%%%%%%%%%%%%%%%%%%%%%%%%%%%%%%%%
	%%%%%%%%%%%%%%%%%%%%%%%%%%%%%%%%%%%%%%%%%%%%%%%%%%%%%%%%%%%%%%%%%%%%%
	
	\section{Statement of the Problem}\label{sec2}
	We will consider the homogeneous Dirichlet problem
	\begin{equation}\label{CDP}\tag{CDP}
		\left\{\begin{array}{cll}\partial_t u +\A \n(u) =0, &\mbox{in } (0,\infty)\times \Omega,\\ u=0, & \mbox{on the lateral boundary},\\ u(0,x)=u_0(x), &\mbox{in } x\in \Omega,\end{array}\right.
	\end{equation}
where the lateral boundary depends on the operator $\A$, for example, for the Restricted Fractional Laplacian it is $(0,\infty)\times (\RN\setminus \Omega)$. We recall that $\Omega$ is a bounded domain with smooth boundary, at least $C^{1,1}$.
The initial data $u_0$ is nonnegative and belongs to the following space of measurable functions
$$\LL^1_{\Phi_1}(\Omega)=\left\{f:\Omega \rightarrow \RR \mbox{ such that }\int_{\Omega} |f(x)| \Phi_1(x) \dx <\infty \right\}.$$
	
	Now, let's fix some properties of both the operator $\A$ and the nonlinearity $\n$, following the notation of \cite{MB+JLV_PM-with-F}. We shall define the precise concept of weak solution at the end of this section.

	%%%%%%%%%%%%%%%%%%%%%%%%%%%%%%%%%%%%%%%%%%%%%%%%%%%%%%%%%
	%%%%%%%%%%%%%%%%%%%%%%%%%%%%%%%%%%%%%%%%%%%%%%%%%%%%%%%%%
	
	\subsection{General Operator and their Kernels}\label{sec2.1}
	The techniques we use can be applied to a wide class of operators under certain assumptions that we are going to present. The key is that we usually work with the inverse operator $\AI$ and its kernel $\K$, the "Green function," which is not necessarily explicit, but it typically satisfies some key estimates. Furthermore, we shall need some assumptions about the kernel of $\A$, which will be useful to prove the positivity of solutions.
	
	\begin{itemize}[leftmargin=*]
		\item \textbf{Basic Assumptions on $\A$}: The operator $\A: \dom(\A)\subseteq \LL^1(\Omega)\rightarrow \LL^1(\Omega)$ is assumed to be densely defined and  sub - Markovian, i.e. it satisfies:
		\begin{equation}\label{A1}\tag{A1}
			\mathcal{L} \mbox{ is $m$-acretive in }\LL^1(\Omega).
		\end{equation}
	
		\begin{equation}\label{A2}\tag{A2}
			\mbox{If }0\leq f\leq 1 \mbox{ then }0\leq e^{-t\A}f \leq 1.
		\end{equation}

		\item \textbf{Assumptions on the kernels}: Whenever $\A$ is well defined in terms of a (hyper)singular kernel $K(x,y)\ge 0$, i.e.
		$$\A f(x)= P.V \int_{\RN} [f(x)-f(y)]K(x,y) \dy,$$
		we will suppose that
		\begin{equation}\label{L1}\tag{L1}
			\inf_{x,y\in \Omega} K(x,y)\geq k_\Omega>0.
		\end{equation}
		
		Moreover, if $\A$ is defined by a kernel and a zero order term, i.e.,
		$$\A f(x)= P.V  \int_{\RN} [f(x)-f(y)]K(x,y) \dy + B(x)f(x),$$
		then, we shall assume
		\begin{equation}\label{L2}\tag{L2}
			K(x,y)\geq c_0 d^\gamma(x)  d^\gamma(y), \hspace{3mm}\mbox{and}\hspace{3mm} B(x)\geq 0.
		\end{equation}
		Where $d(x)=\dist(x,\partial\Omega)|_\Omega$ is the (inner) distance to the boundary and $\gamma\in(0,1]$ is a parameter that depends on the operator. See the end of the subsection for its meaning explained throughout a number of relevant examples, collected in Section \ref{sec7}.
		\item \textbf{Assumptions on $\AI$}: In order to prove our  quantitative estimates, we need to be more precise about $\A$. We suppose it has a left inverse $\AI:\LL^1(\Omega)\rightarrow \LL^1(\Omega)$ which can be defined through a kernel $\K$ as follows,
		$$\AI f(x)=\int_{\Omega} f(y)\K(x,y) \dy.$$
		Note that, in this definition, the homogeneous Dirichlet boundary conditions are automatically included thanks to the function $\K$, hence we can integrate only over $\Omega$.

we are considering the homogeneous Dirichlet conditions posed outside of $\Omega$, because we are only integrating over $\Omega$. The estimates we need to prove our results are the following:
		
		There exists a constant $c_1>0$ which depends on $\Omega$ such that for a.e. $x,y\in \Omega$ we have
		\begin{equation}\label{K1}\tag{K1}
			0\leq \K(x,y)\leq \frac{c_1}{|x-y|^{N-2s}}.
		\end{equation}
		
		There exists a parameter $\gamma\in(0,1]$ and constants $c_0, c_1>0$ which depend on $\A$ and $\Omega$ respectively, such that for almost every point $x,y\in \Omega$ we have
		\begin{equation}\label{K2}\tag{K2}
			c_0 d(x)^\gamma d(y)^\gamma \leq \K(x,y)\leq \frac{c_1}{|x-y|^{N-2s}}\left( \frac{d^\gamma(x)}{|x-y|^\gamma} \land 1 \right) \left( \frac{d^\gamma(y)}{|x-y|^\gamma} \land 1 \right).
		\end{equation}
		
		The lower bound of the above inequality is sometimes weaker than the next well know bound for the Green function of the Fractional Laplacian
		\begin{equation}\label{K4}\tag{K4}
			\K(x,y)\asymp \frac{1}{|x-y|^{N-2s}} \left( \frac{d^\gamma(x)}{|x-y|^\gamma} \land 1 \right) \left( \frac{d^\gamma(y)}{|x-y|^\gamma} \land 1 \right).
		\end{equation}
	\end{itemize}
	In the classical case, for $\A=(-\Delta)$, the Green function satisfies \eqref{K4} when $N\geq 3$. In the fracctional case, the formulas also change when $N=1$ and $s\in (0,1/2)$. This is the reason why we choose $N>2s$.
	
	\underline{The first eigenfunction}: If \eqref{K1} holds it is known that $\AI$ has a nonnegative bounded first eigenfunction $0\leq \ph \in \LL^\infty(\Omega)$ which satisfies the following assertion, there exists $\lambda_1>0$ such that $\A \ph =\lambda_1 \ph$. If \eqref{K2} holds it has been shown, see \cite{MB+AF+JLV_boundary-estimates-elliptic}, that for the parameter $\gamma$ of this assertion we have
\begin{equation}\label{phi1.gamma.xx}
\ph(x)\asymp d^\gamma(x),\hspace{4mm} \forall x\in \overline{\Omega}.
\end{equation}
	Here, we can observe that the parameter $\gamma$ encodes the boundary behaviour of solutions to the problem \eqref{CDP} that depends on the operator $\A$. Thanks to this last result, we can rewrite the hyphothesis \eqref{K2} and \eqref{K4} as follows:
	
	There exists a parameter $\gamma\in(0,1]$ such that \eqref{phi1.gamma.xx} holds and there exist constants $c_0, c_1>0$ which depend on $\A$ and $\Omega$ respectively, such that for almost every point $x,y\in \Omega$ we have
	\begin{equation}\label{K3}\tag{K3}
		c_0\ph(x)\ph(y)\leq \K(x,y)\leq \frac{c_1}{|x-y|^{N-2s}} \left( \frac{\ph(x)}{|x-y|^{\gamma}} \land 1 \right) \left( \frac{\ph(y)}{|x-y|^\gamma} \land 1 \right).
	\end{equation}
	
	\begin{equation}\label{K5}\tag{K5}
		\K(x,y)\asymp \frac{1}{|x-y|^{N-2s}} \left( \frac{\ph(x)}{|x-y|^{\gamma}} \land 1 \right) \left( \frac{\ph(y)}{|x-y|^\gamma} \land 1 \right).
	\end{equation}
	
	We will provide examples at the end of the paper, in Section \ref{sec7}, of operators for which these hypotheses hold. The primary examples include the three non-equivalent definitions of the fractional Laplacian on a bounded domain: the Restricted Fractional Laplacian (RFL), the Censored Fractional Laplacian (CFL), and the Spectral Fractional Laplacian (SFL).
	
	%%%%%%%%%%%%%%%%%%%%%%%%%%%%%%%%%%%%%%%%%%%%%%%%%%%%%%%%%
	%%%%%%%%%%%%%%%%%%%%%%%%%%%%%%%%%%%%%%%%%%%%%%%%%%%%%%%%%
	
	\subsection{The nonlinearity $F$}\label{sec2.2}
	One of the principal points of interest of this paper is the general nonlinearity $\n : \RR \longrightarrow \RR$. In the rest of this work, we always assume that $F$ is a continous function, non decreasing, with the normalizacion $\n(0)=0$ and the following hyphotesis:

\begin{equation}\label{N1}\tag{N1}\begin{split}
	\n \in C^1(\RR \setminus \{0\}), \frac{\n}{F'}\in \mbox{Lip}(\RR)\hspace{2mm}&\mbox{and there exists }0<\mu_0\leq \mu_1< 1 \mbox{ such that:}\\
	1-\mu_1&\leq \left( \frac{F}{F'} \right)'\leq 1 -\mu_0 \hspace{4mm}\mbox{in }\RR.
\end{split}
\end{equation}
	Where $F/F'$ is understood to vanish if $F(r)=F'(r)=0$ or $r=0$. We can alternatively express this property as follows.
	\begin{equation}\label{N2}\tag{N2}\begin{split}
		\n \in C^1(\RR \setminus \{0\}), F' \in \mbox{Lip}_{\mbox{\rm loc}}(\RR\setminus \{0\}) \hspace{2mm}&\mbox{and there exists } 0<\mu_0\leq \mu_1<1 \mbox{ such that:}\\
		\mu_0&\leq \frac{F F''}{(F')^2}\leq \mu_1 \hspace{4mm}\mbox{in }\RR.
	\end{split}
	\end{equation}
	Some consequences of these properties \eqref{N1} and \eqref{N2} can be found in Section \ref{sec3.1}. The main example of nonlinearity is $\n(u) = u |u|^{m-1}$ with $m > 1$, where $\mu_0 = \mu_1 = \frac{m-1}{m}$. This is why we denote $m_i = \frac{1}{1-\mu_i}$ or equivalently $\mu_i = \frac{m_i - 1}{m_i}$ for $i=0,1$. Another variation involves combining two powers, for example $\n(u)=u^4 + 2u^2$, one controls the behavior near $u=0$ and the other near $u=\infty$.
	
	%%%%%%%%%%%%%%%%%%%%%%%%%%%%%%%%%%%%%%%%%%%%%%%%%%%%%%%%%%%
	%%%%%%%%%%%%%%%%%%%%%%%%%%%%%%%%%%%%%%%%%%%%%%%%%%%%%%%%%%%
	
	\subsection{Minimal Weak Dual Solutions}\label{sec2.3}
	We are going to study properties of \solution, introduced by the first author and V\'azquez in \cite{MB+JLV_PM-with-F}. These solutions are obtained by approximation by monotone limits from below in terms of semigroup (or mild) solutions.

	\begin{defn}[\textbf{Mild Solutions}]\label{defmild}
		 A mild or semigroup solution of the problem \eqref{CDP} is a function $u\in C([0,\infty): \LL^1(\Omega))$, with $u(0,x)=u_0(x)$ that is obtained by Crandall-Liggett's method.
	\end{defn}
	Now, we describe briefly this method. Let the interval $[0,T]$ and $n\in \mathbb{N}$ large, we write for every $0\leq k\leq n$ the partition on time $t_k=\frac{k}{n}T$ and the distance between two consecutive times $h=t_{k+1}-t_k=\frac{T}{n}$. For all time $t\in [0,T]$ the mild solution $u(t,\cdot)$ is obtained as a $\LL^1(\Omega)$ limit of solutions $u_{k+1}= u(t_{k+1},\cdot)$ of the following elliptic equations
	$$\frac{T}{n}\A [F(u_{k+1})] + u_{k+1}= u_k,\hspace{4mm}\mbox{or}\hspace{4mm} \frac{u_{k+1}-u_k}{h}= -\A[F(u_{k+1})],$$
	where the data $u_k$ is known from the previous step. See the fully detailed proof in \cite{CRA+LIG}.
	\begin{thm}[\textbf{Crandall - Pierre \cite{CRA+PIE}}]\label{exismild}
		 Let $\A$ satisfy \eqref{A1} and \eqref{A2} and let $F$ satisfy \eqref{N1}. Then, for all initial datum $u_0\in \LL^1(\Omega)$ there exists a unique mild solution to the problem \eqref{CDP}. The semigroup is contractive in $\LL^1(\Omega)$, and futhermore, if $u$ the unique mild solution of \eqref{CDP} with initial datum $u_0\in \LL^p(\Omega)\subset \LL^1(\Omega)$, for $p\geq 1$. Then, $u(t)\in \LL^p(\Omega)$ for every $t>0$, more precisely,
		$$\left\|u(t)\right\|_{\LL^p(\Omega)}\leq \left\|u_0\right\|_{\LL^p(\Omega)}.$$
	\end{thm}
	
	Next, we recall the definition of weak dual solution used in \cite{MB+JLV_harnack-inequality, MB+JLV_PM-with-F}. This class of solutions  are expressed in terms of the inverse operator $\AI$ and encodes the Dirichlet boundary condition.

	\begin{defn}[\textbf{Weak Dual Solutions}]\label{defweakdual}
		 We say that a function $u$ is a weak dual solution of the Cauchy - Dirichlet problem \eqref{CDP} if:
		\begin{itemize}
			\item  $u\in C([0,\infty) : \LL^1_{\Phi_1}(\Omega))$ and $F(u)\in \LL^1 ((0,\infty) : \LL^1_{\Phi_1}(\Omega))$. Moreover, $u(0,x)= u_0\in \LL^1_{\Phi_1}(\Omega)$.
			\item For every test function $\psi$ such that $\psi/\Phi_1\in C^1_c ((0,\infty) : \LL^\infty (\Omega))$, the following identity holds
			\begin{equation}\label{weakdualfor}
				\int_0^\infty \int_{\Omega} \AI u \partial_t \psi = \int_0^\infty \int_{\Omega} F(u) \psi.
			\end{equation}
		\end{itemize}
	\end{defn}
	
	\begin{rem}\label{weakdualremark}\rm
		\renewcommand{\labelenumi}{\rm (\alph{enumi})}
		\begin{enumerate}
			\item We are considering the weak solution of the dual equation $\partial_t U = F(u)$, with $U=\AI u$. The equation is satisfied on $\Omega$ with homogeneous Dirichlet type boundary conditions, which are encoded in the inverse operator $\AI$. Note that the space $\LL^1_{\Phi_1}(\Omega)$ is bigger than $\LL^1(\Omega)$ and allows for non-integrable functions at the boundary of $\Omega$.
			\item Condition $\psi/\Phi_1\in C^1_c ((0,\infty) : \LL^\infty(\Omega))$ implies $\left\| \frac{\psi(t,\cdot)}{\Phi_1} \right\|_{\LL^\infty(\Omega)}<\infty$ and $\left\| \frac{\partial_t \psi(t,\cdot)}{\Phi_1} \right\|_{\LL^\infty(\Omega)}<\infty$ for all $t>0$. Moreover, these two functions have compact support in time, which lead us to conclude that they are in $\LL^1(0,\infty)$.
			\item Existence of weak dual solutions for \eqref{CDP} and ``uniqueness'' of the minimal ones (in the sense explained above) has been proven in \cite{MB+JLV_PM-with-F}. However, a uniqueness result is still missing. For the Cauchy problem on the whole space, the pioneering work \cite{PIE_uniq-sol} shows uniqueness for general measure data, when $\A=-\Delta$. In the case of RFL with general $F$ and measure data, uniquenes of distributional solutions has been proven \cite{GMP-Filtration}, and in \cite{DTEJ-uniq} for bounded integrable distributional solutions and a wide class of Levy operators. Uniqueness for Dirichlet-type problem remains a difficult open problem.
			\item If $u$ is the unique mild solution of the problem \eqref{CDP} with initial datum $u_0\in \LL^1(\Omega)$. Then, this $u$ is weak dual solution of \eqref{CDP} in the sense of Definition \ref{defweakdual}. This fact allow us to see the weak dual formulation \eqref{weakdualfor} as a property of mild solutions. The proof can be found in \cite[Proposition 7.2]{MB+JLV_PM-with-F}.
		\end{enumerate}
	\end{rem}
	In the rest of this section, we will show how to construct the \solution\, as the monotone limit from below of mild solutions. Such solutions satisfy Definition \ref{defweakdual}, and are unique, as we will see below.
	\begin{defn}[\textbf{Minimal Weak Dual Solutions}]\label{minimalweakdual}
		Let $0\leq u_0\in \LL^1_{\Phi_1}(\Omega)$ and choose any increasing sequence of bounded functions $\left\{ u_{0,n} \right\}_{n=0}^\infty$ that converge to $u_0$ in the topology of $\LL^1_{\Phi_1}(\Omega)$. We denote $u_n$ as the unique mild solution of the problem \eqref{CDP} with initial datum $u_{0,n}$ for all n, (Theorem \ref{exismild}). Then, we say that a \solution\,$u$ of \eqref{CDP} with initial datum $u_0$ is defined as the following monotone limit from below for $(t,x)\in ((0,\infty)\times \Omega)$
		$$u(t,x):= \lim\limits_{n\to \infty} u_n(t,x).$$
	\end{defn}
 Note that mild solutions are ordered by comparison $u_n\leq u_{n+1}$ in $(0,\infty)\times \Omega$. Hence, we have $\left\{ u_n(t,x) \right\}_{n=1}^\infty$ is a monotone increasing sequence for all $(t,x)\in (0,\infty)\times\Omega$. Therefore, there exists a pointwise limit $u(t,x)$, whose value, a priori, can be infinite. Nevertheless, it has been shown in \cite{MB+JLV_PM-with-F} that this limit converges uniformly. A standard procedure for choosing the increasing sequence is as follows: $\forall n\in \mathbb{N}$ define $u_{0,n}:= (u_0\land n)\in \LL^\infty(\Omega)$ and let $u_n$ the unique mild solution of problem \eqref{CDP} with initial datum $u_{0,n}$.

	\begin{thm}[\textbf{Existence Uniqueness and Properties of Minimal Weak Dual solutions \cite[Theorem 4.4 and Theorem 4.5]{MB+JLV_PM-with-F}}]\label{exisweakdual}
		Let $\A$ satisfy \eqref{A1}, \eqref{A2}, \eqref{K3} and let $F$ satisfy \eqref{N1}. Then, for every $0\leq u_0\in \LL^1_{\Phi_1}(\Omega)$ the following assertions are true:
		\renewcommand{\labelenumi}{\rm (\roman{enumi})}
		\begin{enumerate}
			\item There exists a unique $u$ \solution\,of \eqref{CDP} as we constructed in Definition \ref{minimalweakdual}. If we choose two differents increasing sequences of bounded functions $\left\{ u_{0,n} \right\}$, $\left\{ v_{0,n} \right\}$ and we denote $u_n$ and $v_n$ as the mild solutions of \eqref{CDP} with initial datum $u_{0,n}$ and $v_{0,n}$ respectively, then they converge to the same limit, i.e. $\lim\limits_{n\to \infty}u_n(t,x)=\lim\limits_{n\to \infty}v_n(t,x)$.
			\item For all $\tau>0$ we have that $u_n\to u$ as $n\to \infty$ in $\LL^\infty((\tau,\infty)\times \Omega)$.
			\item Minimal weak dual solution $u$ is a weak dual solution in the sense of Definition \ref{defweakdual}. In fact, $u\in C([0,\infty) : \LL^1_{\Phi_1}(\Omega))$ and identity \eqref{weakdualfor} holds. Moreover, the standard comparison result can be applied to \solution.
		\end{enumerate}
	\end{thm}

The purpose of \solutions\,is to provide mild solutions the weak dual formulation \eqref{weakdualfor} as a property, because they are effective for proving existence and uniqueness but often challenging to obtain a priori estimates. This is why we introduce the weak dual formulation, which significantly aids in establishing these estimates. As we have seen, mild solutions are weak dual solutions, ensuring that the estimates for weak dual solutions hold for mild solutions too. Essentially, \solutions\,can be seen as an extension or a limit of mild solutions. These solutions exist, are unique for all non-negative initial data in $\LL^1_{\Phi_1}$ and satisfy the following properties:
	\begin{itemize}
		\item $u\in C( [0,\infty) : \LL^1_{\Phi_1}(\Omega))$
		\item For all test function $\psi$ such that $\psi/\Phi_1\in C^1_c ((0,\infty) : \LL^\infty(\Omega))$ we have
		$$\int_0^\infty \int_{\Omega} \AI u \partial_t \psi - \int_0^\infty \int_{\Omega} F(u) \psi =0.$$
		\item $u(t)\in \LL^p(\Omega)$ for all $t>0$ for some $p>\frac{N}{2s}$, in fact, they are in $\LL^\infty(\Omega)$.
	\end{itemize}
	
	%%%%%%%%%%%%%%%%%%%%%%%%%%%%%%%%%%%%%%%%%%%%%%%%%
	%%%%%%%%%%%%%%%%%%%%%%%%%%%%%%%%%%%%%%%%%%%%%%%%%
	
	\subsection{Main results}\label{sec2.4}
	
	Here, we present a summary of the most important estimates that we obtain along this paper for \solutions\,of \eqref{CDP}, these results are various forms of upper and lower bounds which we call Global Harnack Principle, (GHP) for short, namely Theorems \ref{GHPI}, \ref{GHPII} and \ref{GHPIII}. As a consequence, we can prove interior regularity estimates, Theorem \ref{interiorregularity}.
	
	Our initial result provides a quantitative estimate for $F(u)$ that is valid for a wide range of operators, including the classical Laplacian, for which, due to the finite speed of propagation, the lower bound can only be true after a waiting time $t_*$. For this reason the validity of such general estimates is restricted to ``large times''. We shall see that if we restrict to purely nonlocal operators, we can extend the estimates to all $t>0$.
	\begin{thm}[\textbf{GHP I}]\label{GHPI}
		Let $\A$ satisfy \eqref{A1}, \eqref{A2} and \eqref{K2}, with $\sigma_1=\left(1\land \frac{2sm_1}{\gamma (m_1 -1)}\right)=1$, let $F$ satisfy \eqref{N1} and let $u$ be the \solution\,of \eqref{CDP} with initial datum $0\leq u_0\in \LL^1_{\Phi_1}(\Omega)$. Then, for all $t>0$ large enough, i.e., $t\geq t_*= c_*\left( \|u_0\|_{\LL^1_{\Phi_1}(\Omega)}^{-(m_1-1)} \lor \|u_0\|_{\LL^1_{\Phi_1}(\Omega)}^{-(m_0-1)} \right)$, there exists constants $\underline{\kappa_1}$ and $\overline{\kappa_2}$ such that
		
		\begin{equation}\label{GHPineq1}
			\underline{\kappa_1} \frac{\Phi_1(x)}{t^{\frac{m_0}{m_0-1}}}\leq F(u(t,x))\leq \overline{\kappa_2} \frac{\Phi_1(x)}{t^{\frac{m_1}{m_1-1}}}.
		\end{equation}
		Constants $\underline{\kappa_1}$, $\overline{\kappa_2}$ and $c_*$ only depends on N, s, $\gamma$, $m_i$, F, $\Omega$ and $\lambda_1$.
	\end{thm}
	
	We now present several results that provides us with upper and lower bounds valid for all positive times: the class of operators to which these result apply is ``purely'' non-local, i.e., local operators do not belong to this class. More precisely we shall assume \eqref{L1} in Theorem \ref{GHPII} or \eqref{L2} in Theorem \ref{GHPIII}. A relevant difference between Theorem \ref{GHPII} and Theorem \ref{GHPIII} is that in the first theorem the powers that describe the boundary behaviour (dictated by the first eigenfunction $\Phi_1$) are matching, while in the latter they are not, we shall devote further comments to this issue below.
	\begin{thm}[\textbf{GHP II}]\label{GHPII}
		Let $\A$ satisfy \eqref{A1}, \eqref{A2}, \eqref{K2} and \eqref{L1}, let F satisfy \eqref{N1} and let $u$ be the \solution\,of \eqref{CDP} with initial datum $0\leq u_0\in \LL^1_\ph(\Omega)$. Assume either $\sigma_1=1$ or $\sigma_1<1$, $K(x,y)<c_1|x-y|^{-(N+2s)}$ and $\ph\in C^\gamma(\overline{\Omega})$, then there exists positive constants $\underline{\kappa_3}$ and $\overline{\kappa_4}$ which only depends on N,s,$m_i$, F, $\gamma$, $\Omega$ and $\lambda_1$ such that:
		\renewcommand{\labelenumi}{\rm (\roman{enumi})}
		\begin{enumerate}
			\item If $\|u_0\|_{\LL^1_{\Phi_1}(\Omega)}\leq 1$, or equivalently $t_*=c_*\|u_0\|_{\LL^1_{\Phi_1}(\Omega)}^{-(m_1-1)}$, then for any $t>0$ and a.e. $x\in \Omega$ we have
			\begin{equation}\label{GHPineq2}
				\underline{\kappa_3}\left( 1 \land \frac{t}{t_*} \right)^{\frac{m_1^2}{m_1-1}} \frac{\ph^{\sigma_1}(x)}{t^{\frac{m_j}{m_j-1}}}\leq F(u(t,x)) \leq \overline{\kappa_4} \frac{\Phi_1(x)^{\sigma_1}}{t^{\frac{m_i}{m_i-1}}}.
			\end{equation}
			Where $i=1$ $j=0$ if $t\geq t_*$ and $i=0$ $j=1$ if $t\leq t_*$.
			\item If $\|u_0\|_{\LL^1_{\Phi_1}(\Omega)}>1$, or equivalently $t_*=c_*\|u_0\|_{\LL^1_{\Phi_1}(\Omega)}^{-(m_0-1)}$, then for $t>0$ and $x\in \Omega$
			\begin{equation}\label{GHPineq3}
				\underline{\kappa_3} \left( \left[ \frac{t}{t_*} \right]^{m_1} \land \left[ \frac{t_*}{t} \right]^{\frac{m_0}{m_0-1}} \right)  \Phi_1(x)^{\sigma_1}\leq F(u(t,x)) \leq \overline{\kappa_4} \frac{\Phi_1(x)^{\sigma_1}}{t^{\frac{m_i}{m_i-1}}}.
			\end{equation}
			Where $i=1$ if $t\geq t_*$ and $i=0$ if $t\leq t_*$.
		\end{enumerate}
		
	\end{thm}

\medskip

In the case when we allow $\A$ to have a degenerate kernel at the boundary $\partial\Omega$, i.e. assumption \eqref{L2}, our sharp result -- with non matching power in general-- reads as follows.

\medskip

	\begin{thm}[\textbf{GHP III}]\label{GHPIII}
		Let $\A$ satisfy \eqref{A1}, \eqref{A2}, \eqref{K2} and \eqref{L2}, let F satisfy \eqref{N1} and let $u$ the \solution\,of \eqref{CDP} with initial datum $0\leq u_0\in \LL^1_{\ph(\Omega)}$. Then, there exists positive constants $\underline{\kappa_5}$ and $\overline{\kappa_6}$ which only depends on N,s,$m_i$, F, $\gamma$, $\Omega$ and $\lambda_1$ such that:
		\renewcommand{\labelenumi}{\rm (\roman{enumi})}
		\begin{enumerate} 	
			\item If $\|u_0\|_{\LL^1_{\Phi_1}(\Omega)}\leq 1$, or equivalently $t_*=c_*\|u_0\|_{\LL^1_{\Phi_1}(\Omega)}^{-(m_1-1)}$, then for any $t>0$ and a.e. $x\in \Omega$ we have
			\begin{equation}\label{GHPineq4}
				\underline{\kappa_5} \left( 1 \land \frac{t}{t_*} \right)^{\frac{m_1^2}{m_1-1}} \frac{\ph(x)^{m_1}}{t^{\frac{m_j}{m_j-1}}}\leq F(u(t,x))\leq \overline{\kappa_6} \frac{\Phi_1(x)^{\sigma_1}}{t^{\frac{m_i}{m_i-1}}}.
			\end{equation}
			Where $i=1$ $j=0$ if $t\geq t_*$ and $i=0$ $j=1$ if $t\leq t_*$.
			\item If $\|u_0\|_{\LL^1_{\Phi_1}(\Omega)}> 1$, or equivalently $t_*=c_*\|u_0\|_{\LL^1_{\Phi_1}(\Omega)}^{-(m_0-1)}$, then for $t>0$ and a.e. $x\in \Omega$ we have
			\begin{equation}\label{GHPineq5}
				\underline{\kappa_5} \left( \left[ \frac{t}{t_*} \right]^{m_1} \land \left[ \frac{t_*}{t} \right]^{\frac{m_0}{m_0-1}} \right)  \Phi_1(x)^{m_1}\leq F(u(t,x))\leq \overline{\kappa_6} \frac{\Phi_1(x)^{\sigma_1}}{t^{\frac{m_i}{m_i-1}}}.
			\end{equation}
			Where $i=0$ if $t\leq t_*$ and $i=1$ if $t\geq t_*$.
		\end{enumerate}
	\end{thm}

The proof of the above theorems is obtained by combining upper and lower bounds in Section \ref{sec3} and Section \ref{sec5}.

\begin{rem}[Sharpness of the main results]\label{rmk.sharpness}~\\[-8mm] \rm
\begin{enumerate}[leftmargin=*]
\item \textit{Theorem \ref{GHPI}: a general result. }Observe that Theorem \ref{GHPI} is only valid for large times,  i.e.,  after a waiting time $t\geq t_*$. However, this result has wider applicability since it only requires assumption \eqref{K2} on the Green function: indeed, it holds for both local and nonlocal operators, providing a new/alternative proof in the local case, where finite speed of propagation prevents from having the bounds for arbitrarily small times. Remember that if $\A=(-\Delta)$ the classical Laplacian, the Porous Medium equation exhibits finite speed of propagation, cf \cite{ARO+PEL}. Additionally, the bound is valid for a more general class of solutions, see Definition \ref{S}.  As a consequence,  in general it is unavoidable to have the waiting time $t_*$: roughly speaking, $t_*$ estimates the time required to ``fill the hole'' left by the initial data $u_0$, which may be concentrated near the boundary of $\Omega$, leaving a region in the domain's interior where it vanishes. Consequently, when $\|u_0\|_{\LL^1_{\Phi_1}(\Omega)}$ is small, then $t_*$ has to be large due to the prolonged time needed to fill this hole. This result is therefore sharp since it is the only possible if we want it to hold for the complete class of operators $\A$ that we consider.
\item
\textit{The sharp nonlocal results of Theorems \ref{GHPII} and \ref{GHPIII} and infinite speed of propagation. }These results establish the positivity of \solutions\,for all positive times, exploiting the non-local property of the operator. A quantitative lower bound is shown, indicating infinite speed of propagation for these equations, as it was proved in \cite{MB+AF+XR_positivity-and-regularity} for the first time, in the case when $F(u)=u^m$ and $\A$ is the RFL, see also \cite{MB+AF+JLV_sharp-global-estimates} for the case of more general $\A$. Note that, in Theorems \ref{GHPII} and \ref{GHPIII} the dependence of the initial datum $u_0$ may vanish in the lower bound for large times, hence the upper and lower estimates do not depend on $u_0$, as it happens in Theorem \ref{GHPI}.

In Theorem \ref{GHPII}, the powers that characterize the upper and lower boundary behaviour are matching (i.e. the same power of $\Phi_1\asymp \dist(\cdot, \partial\Omega)^\gamma$ appear in the upper and lower bounds of inequalities \eqref{GHPineq2} and \eqref{GHPineq3}), which will allow us to obtain the regularity results of Theorem \ref{interiorregularity}, following the ideas of \cite{MB+AF+JLV_sharp-global-estimates}. These bounds are indeed sharp, because in the particular case $F(u)=u^m$, there is a separate-variables solution whose boundary behaviour precisely conforms to inequalities \eqref{GHPineq2} and \eqref{GHPineq3}, namely
\[
U_T(t,x)^m=\frac{S(x)^m}{(T+t)^{\frac{m}{m-1}}}\qquad\mbox{where}\qquad S(x)^m\asymp \Phi_1(x)^\sigma\,,
\]
see \cite{MB+AF+JLV_boundary-estimates-elliptic, MB+AF+JLV_sharp-global-estimates}. This Theorem can be applied to a wide class of nonlocal operators: as main examples we mention the RFL and CFL.
The key property here is that the kernel of the operators is non-degenerate at the boundary, namely is strictly positive on $\overline{\Omega}$, which is assumption \eqref{L1}. See Section \ref{sec7} for more examples of operators.

In Theorem \ref{GHPIII} the powers of the first eigenfunction in the upper and lower bounds do not match, thus, a priori we cannot expect the bounds to be sharp. However, a posteriori these bounds turn out to be sharp, as we shall explain following the leading example provided by the Spectral Fractional Laplacian (SFL). The source of the problem in this case is that the operator is allowed to have a kernel which can be degenerate at the boundary: in the case of SFL, the kernel vanishes as $\Phi_1$ at $\partial\Omega$, and this may affect the boundary behaviour of solutions. The discriminant factor is the ``size'' of the initial datum: as first observed in \cite{MB+AF+JLV_sharp-global-estimates} for the case $F(u)=u^m$, when $\sigma<1$, and for small times and small initial data, there exists a solution whose upper bound match the lower bound of \eqref{GHPineq4}. We extend such result to the case of general nonlinearities in Corollary \ref{smalldata2} where we provide an upper bound that holds for small data (and small times) and matches the lower bounds of \eqref{GHPineq4} (when $m_0=m_1$) as follows:  for a.e. $x\in \Omega$, and all $0<t\leq t_*$ we have
\begin{equation}\label{matching.III}
\underline{\kappa_5} \left(\frac{t}{t_*} \right)^{\frac{m_1^2}{m_1-1}} \frac{\ph^{m_1}(x)}{t^{\frac{m_1}{m_1-1}}} \le F(u(t,x)) \leq \frac{\Phi_1^{m_0}(x) }{\left[ A^{1-m_1} - \tilde{C} t \right]^{\frac{m_1}{m_1-1}}}.
\end{equation}
where the positive constants $\underline{\kappa_5}, A, \tilde{C}$ have explicit expressions, see Theorem  \ref{GHPIII} and Corollary \ref{smalldata2}. It is remarkable that in this case also the upper and lower behaviour in time matches, hence it is sharp for small times.

\item\textit{On the form of the sharp bounds. }Finally, we would like to mention that for the boundary estimates,  we bound $F(u)$ from above and from below rather than the solution $u$. This is because, when estimating the general nonlinearity $F$ by the powers $m_0$ and $m_1$, as stated in the Lemma \ref{Nonlin} below, some information is lost. The sharp form of such estimates can only be stated in terms of $F(u)$, which is somehow more implicit: indeed, if we want to get more explicit bounds, the powers of the first eigenfunction that we get, could be different, i.e., $m_0$ and $m_1$ respectively, as we have seen in \eqref{matching.III}, and can be easily understood with a simple example such as $F(u)= u^2+u^{10}$. Let us make an example: fix a positive time $0<t\leq t_*$ in inequality \eqref{GHPineq2} of Theorem \ref{GHPII} (matching powers), and we obtain
$$\underline{\kappa_3}\left(  \frac{t}{t_*} \right)^{\frac{m_1^2}{m_1-1}} \frac{\ph^{\sigma_1}(x)}{t^{\frac{m_1}{m_1-1}}}\leq F(u(t,x)) \leq u(t,x)^{m_0} \hspace{5mm}\mbox{and}\hspace{5mm} u(t,x)^{m_1}\leq F(u(t,x)) \leq \overline{\kappa_4} \frac{\Phi_1(x)^{\sigma_1}}{t^{\frac{m_0}{m_0-1}}},$$
that combined give us the more explicit bound at the price of having non-matching powers:
\begin{equation}\label{a}
	\underline{\kappa_3}^{1/m_0}\left(  \frac{t}{t_*} \right)^{\frac{m_1^2}{(m_1-1)m_0}} \frac{\ph(x)^{\sigma_1/m_0}}{t^{\frac{m_1}{(m_1-1)m_0}}}\leq u(t,x)\leq \overline{\kappa_4}^{1/m_1} \frac{\Phi_1(x)^{\sigma_1/m_1}}{t^{\frac{m_0}{(m_0-1)m_1}}}.
\end{equation}
This may seem not optimal, but keep in mind that estimating $F(u)$ as above, some information is lost. This additional difficulty disappears when $m_1=m_0$, i.e., we are in the particular case of a single power, and we recover homogeneity, which in general we do not have. In the more general case, this phenomenon cannot be avoided, as far as we know, because the nonlinearity $F$ may oscillate between the two previously mentioned powers near zero. Therefore, it is natural for the lower bound behavior to be determined by the larger power $m_1$ and the upper bound by the smaller power $m_0$. Summing up, if we wish to express the sharp boundary behavior of $u$ in a more precise way than \eqref{a}, we must describe it in terms of $F^{-1}(\ph^{\sigma_1})$ as follows. Starting from Theorems \ref{GHPI}, \ref{GHPII} and \ref{GHPIII}, and, possibly modifying the constants, we can apply the inverse $F^{-1}$, because $F$ is non decreasing, to get
\begin{equation}\label{b}
	\underline{\kappa_3}F^{-1}\left( \left( \frac{t}{t_*} \right)^{\frac{m_1^2}{m_1-1}} \frac{\ph^{\sigma_1}(x)}{t^{\frac{m_1}{m_1-1}}} \right)\leq u(t,x) \leq \overline{\kappa_4} F^{-1}\left( \frac{\Phi_1(x)^{\sigma_1}}{t^{\frac{m_0}{m_0-1}}} \right).
\end{equation}
 \normalcolor
\end{enumerate}
\end{rem}	

\noindent\textit{Interior regularity. }Once solutions are positive and bounded, it often happens that they are (H\"older) continuous in the interior of the domain and even classical, when the operators allows it. We shall state our regularity results first for weak solutions (positive and bounded by general constants), and we shall explore their extension to minimal weak dual solutions in Section \ref{sec6}, where we also provide a complete proof of the following theorem.

	  \begin{thm}[\textbf{Interior Regularity}]\label{interiorregularity}
	  	Assume that the operator $\A$ is defined by
	  	\begin{equation*}\label{intregec1}
	  		\A f(x)= P.V \int_{\RN} (f(x)-f(y))K(x,y) \dy +B(x)f(x).
	  	\end{equation*}
	  	Let $r>0$ and $x_0\in \Omega$ such that $B_{2r}(x_0)\subset \Omega$, suppose the kernel satisfies the following properties
	  	$$K(x,y) \asymp \frac{1}{|x-y|^{N+2s}}\hspace{3mm}\mbox{in }B_{2r}(x_0)\hspace{4mm}\mbox{and}\hspace{4mm}K(x,y) \lesssim \frac{1}{|x-y|^{N+2s}}\hspace{3mm}\mbox{in }\RN \setminus B_{2r}(x_0).$$
	  	Assume $B$ is locally bounded on $\Omega$. Let u be a weak solution of the problem \eqref{CDP} in $(T_0,T_1)\times \Omega$ such that
	  	$$0<\delta\leq u(t,x)\hspace{3mm}\forall (t,x)\in (T_0,T_1)\times B_{2r}(x_0)\hspace{4mm}\mbox{and}\hspace{4mm} 0\leq u(t,x)\leq M \hspace{3mm}\forall (t,x)\in (T_0,T_1)\times \Omega.$$
	  	\renewcommand{\theenumi}{\roman{enumi})}
	  	\begin{enumerate}
	  		\item Then $u$ is H\"older continuous in the interior, that is, there exists an $\alpha\in (0,1]$ such that for all $0<T_0<T_2<T_1$ we have
	  		\begin{equation}\label{intregec2}
	  			\|u\|_{C_{x,t}^{\alpha, \alpha/2s}((T_2,T_1)\times B_r(x_0))}\leq C.
	  		\end{equation}
	  		\item If there exists a $\beta\in (0,1\land 2s)$ with $\beta+2s\not \in \mathbb{Z}$ such that $|K(x,y)-K(x',y)|\leq C |x-x'|^\beta |y|^{-(N+2s)}$. Then the solution $u$ is classical in the interior, more precisely, we have the following estimate for all $0<T_0<T_2<T_1$
	  		\begin{equation}\label{intregec3}
	  			\|u\|_{C_{x,t}^{2s+\beta, 1+\beta/2s}((T_2,T_1)\times B_r(x_0))}\leq C.
	  		\end{equation}
	  	\end{enumerate}
	  \end{thm}
	
\medskip

\subsection{The main results in the form of tables}
	We end this section by collecting all the sharp boundary estimates obtained in this work in tables, namely the results contained the previous section,
	 which describe when we have a Global Harnack Principle, and Theorems \ref{smalldata1}, \ref{smalldata2}, \ref{smalldata3}, \ref{smalldata4} which reflect the anomalous boundary behaviour.
The aim is to clearly present the best estimates that determine the precise boundary behavior of the solutions to the problem \eqref{CDP}, separating all the cases, depending on the size of the initial datum or of the time regime. We separate large and small data, depending on the size of $\|u_0\|_{\LL^1_{\Phi_1}}$, and also consider the case of ``very small data'', i.e., when $u_0\lesssim \ph^\beta$ for some $\beta>0$. We shall also separate the small and large times regimes, meaning times smaller or bigger than $t_*$. In this way, both the expression of $F$ and the exponents can be estimated  ``more explicitly'', see for instance Lemma \ref{Nonlin} below. We will now fix a positive time $t>0$, and see how the estimate changes in the various cases, focussing on three main examples: the Classical Laplacian, Theorem \ref{GHPI},  for which $\gamma=s=1$, the Restricted Fractional Laplacian, Theorem \ref{GHPII},  for which $\gamma=s\in (0,1)$ and $\sigma_1=1$. Observe that the Censorel Fractional Laplacian also belongs to this class of operators, with $s\in (1/2,1)$, $\gamma=2s-1$ and $\sigma_1=1$ . And finally, the Spectral Fractional Laplacian, Theorem \ref{GHPIII},  for which $s\in (0,1)$,  $\gamma=1$ and $\sigma_1=\left( 1 \land \frac{2s m_1}{m_1-1} \right)$.  % and since the first eigenfunction is

\begin{center}
	\begin{tabular}{| c | c | c | }
		\hline  \multicolumn{3}{ |c| }{} \\
		\multicolumn{3}{ |c| }{\textbf{\Large The Classical Laplacian}} \\
		\multicolumn{3}{ |c| }{} \\ \hline
		& Small times $t\leq t_*$
		
		& Large times $t\geq t_*$
		\\
		\hline & &  \\
		$ u_0\in \LL^1_{\Phi_1}(\Omega)$ &
		$F(u(t,x))\lesssim \frac{\Phi_1(x)}{t^{\frac{m_1}{m_1-1}}}$
		
		& $\frac{\Phi_1(x)}{t^{\frac{m_0}{m_0-1}}} \lesssim F(u(t,x))\lesssim \frac{\Phi_1(x)}{t^{\frac{m_1}{m_1-1}}}$ \\
		& &  \\ \hline

	\end{tabular}
\end{center}

\begin{center}
	\begin{tabular}{| c | c | c | }
		\hline  \multicolumn{3}{ |c| }{} \\
		\multicolumn{3}{ |c| }{\textbf{\Large Restricted Fractional Laplacian}} \\
		\multicolumn{3}{ |c| }{} \\ \hline
		& Small times $t\leq t_*$
		& Large times $t\geq t_*$
		\\
		 \hline &  & \\
		$ \|u_0\|_{\LL^1_{\Phi_1}}>1$ &
		$\left( \frac{t}{t_*} \right)^{m_1} \Phi_1(x) \lesssim F(u(t,x))\lesssim \frac{\Phi_1(x)}{t^{\frac{m_0}{m_0-1}}}$
		
		& $t_*^{\frac{m_0}{m_0-1}}\frac{\Phi_1(x)}{t^{\frac{m_0}{m_0-1}}} \lesssim F(u(t,x))\lesssim \frac{\Phi_1(x)}{t^{\frac{m_1}{m_1-1}}}$ \\
		& &  \\  \hline & &  \\
		$ \|u_0\|_{\LL^1_{\Phi_1}}\leq 1$ &
		$\left( \frac{t}{t_*} \right)^{\frac{m_1^2}{m_1-1}}\frac{\Phi_1(x)}{t^{\frac{m_1}{m_1-1}}}\lesssim F(u(t,x))\lesssim \frac{\Phi_1(x)}{t^{\frac{m_0}{m_0-1}}}$
		& $\frac{\Phi_1(x)}{t^{\frac{m_0}{m_0-1}}}\lesssim F(u(t,x))\lesssim \frac{\Phi_1(x)}{t^{\frac{m_1}{m_1-1}}}$ \\
		& &  \\  \hline
	\end{tabular}
\end{center}

\begin{center}
	\begin{tabular}{| c | c | c | }
		\hline  \multicolumn{3}{ |c| }{} \\
		\multicolumn{3}{ |c| }{\textbf{\Large Spectral Fractional Laplacian}} \\
		\multicolumn{3}{ |c| }{} \\ \hline
		& Small times $t\leq t_*$
		& Large times $t\geq t_*$
		\\
		\hline & &  \\
		$ \|u_0\|_{\LL^1_{\Phi_1}}>1$ &
		$\left( \frac{t}{t_*} \right)^{m_1}\Phi_1(x)^{m_1}\lesssim F(u(t,x))\lesssim \frac{\Phi_1(x)^{\sigma_1}}{t^{\frac{m_0}{m_0-1}}}$
		& $t_*^{\frac{m_0}{m_0-1}}\frac{\Phi_1(x)^{m_1}}{t^{\frac{m_0}{m_0-1}}}\lesssim F(u(t,x))\lesssim \frac{\Phi_1(x)^{\sigma_1}}{t^{\frac{m_1}{m_1-1}}}$ \\
		& & \\  \hline & & \\
		$ \|u_0\|_{\LL^1_{\Phi_1}}\leq 1$ &
		$\left( \frac{t}{t_*} \right)^{\frac{m_1^2}{m_1-1}} \frac{\Phi_1(x)^{m_1}}{t^{\frac{m_1}{m_1-1}}}\lesssim F(u(t,x))\lesssim \frac{\Phi_1(x)^{\sigma_1}}{t^{\frac{m_0}{m_0-1}}}$
		& $\frac{\Phi_1(x)^{m_1}}{t^{\frac{m_0}{m_0-1}}}\lesssim F(u(t,x))\lesssim \frac{\Phi_1(x)^{\sigma_1}}{t^{\frac{m_1}{m_1-1}}}$ \\
		& &  \\  \hline & &  \\
		$ u_0\leq A \Phi_1^{1-2s/\gamma}$ &
		$\left( \frac{t}{t_*} \right)^{\frac{m_1^2}{m_1-1}}\frac{\Phi_1(x)^{m_1}}{t^{\frac{m_1}{m_1-1}}}\lesssim F(u(t,x))\lesssim \frac{\Phi_1(x)^{m_0\frac{1-2s}{\gamma}}}{\left[ A^{1-m_1}-\tilde{C}t \right]^{\frac{m_1}{m_1-1}}}$
		&  \\
		$t<(t_* \land T_A)$& &  \\ \hline & &  \\
		$ u_0\leq A \Phi_1$ &
		$\left( \frac{t}{t_*} \right)^{\frac{m_1^2}{m_1-1}}\frac{\Phi_1(x)^{m_1}}{t^{\frac{m_1}{m_1-1}}}\lesssim F(u(t,x))\lesssim \frac{\Phi_1(x)^{m_0}}{\left[ A^{1-m_1}-\tilde{C}t \right]^{\frac{m_1}{m_1-1}}}$
		&  \\ and $B\equiv 0$& &  \\
		$t<(t_* \land T_A)$& &  \\ \hline
	\end{tabular}
\end{center}
	
	%%%%%%%%%%%%%%%%%%%%%%%%%%%%%%%%%%%%%%%%%%%%%%%%%%%%%%%%%%%%%%%%%%%%%%
	%%%%%%%%%%%%%%%%%%%%%%%%%%%%%%%%%%%%%%%%%%%%%%%%%%%%%%%%%%%%%%%%%%%%%%
	%%%%%%%%%%%%%%%%%%%%%%%%%%%%%%%%%%%%%%%%%%%%%%%%%%%%%%%%%%%%%%%%%%%%%%
	%%%%%%%%%%%%%%%%%%%%%%%%%%%%%%%%%%%%%%%%%%%%%%%%%%%%%%%%%%%%%%%%%%%%%%
	
	\section{Upper Bounds}\label{sec3}
	This section is devoted to proving the upper bound, i.e. the upper part of the Generalized Harnack Principle (GHP) of Subsection \ref{sec2.4}. We will divide the section into three steps: Firstly, we present some known results from \cite{ MB+AF+JLV_boundary-estimates-elliptic, MB+JLV_PM-with-F, CRA+PIE}. Then, we proceed with the proof of the quantitative upper estimate. Finally, we discuss the smoothing effects and elucidate the distinctions from previous bounds. These smoothing effects have already been established in \cite{MB+JLV_PM-with-F}; however, we restate them here as they are essential for the lower bounds.
	
	Before delving into the estimates, let us first define a more general class of weak dual solutions.
	\begin{defn}\label{S}
		We denote by $S$ the class of all non-negative weak dual solutions of the problem \eqref{CDP} with initial data $0\leq u_0\in \LL^1_{\Phi_1}(\Omega)$, as defined in Definition \ref{defweakdual}. These solutions $u\in S$ also satisfy the following conditions (i) the map $u_0 \to u(t)$ is order preserving in $\LL^1_{\Phi_1}(\Omega)$ (ii) for all $t>0$ we have $u(t)\in \LL^p(\Omega)$ for some $p>N/2s$.
	\end{defn}
	
	\begin{rem}\label{class} \rm
		Note that $u$, the unique \solution\,of \eqref{CDP} possesses all the properties to belong to the class $S$ because it satisfies Definition \ref{defweakdual}. Moreover, mild solutions of the sequence $\left\{ u_n \right\}_{n\in \mathbb{N}}$  are bounded. Therefore, for any $t>0$ we have $u_n(t)\in \LL^p(\Omega)$ for all $p>1$. In particular, $u_n\in S$ for all $n$, allowing us to apply the absolute upper bounds from Theorem \ref{AbsoluteUpperBounds}, which are independent of the initial datum.
		
		 Now, we can guarantee that for any fixed $0<\tau<t$, there exists the monotone limit from below $u(t,x)=\lim\limits_{n\to \infty} u_n(t,x)$ in $\LL^\infty((\tau,\infty)\times \Omega)$, and that such limit satisfies the same upper estimates due to the lower semicontinuity of the $\LL^\infty$ norm; that is $u(t,\cdot)$ is bounded on $\Omega$.
	\end{rem}
The upper bound and the lower bound for large times $t>t_*$ are true for any $u\in S$.
	
	%%%%%%%%%%%%%%%%%%%%%%%%%%%%%%%%%%%%%%%%%%%%%%%%%%%%%%%%%%%%%%%%%%%%%%
	%%%%%%%%%%%%%%%%%%%%%%%%%%%%%%%%%%%%%%%%%%%%%%%%%%%%%%%%%%%%%%%%%%%%%%
	
	\subsection{First set of estimates}\label{sec3.1}
	Now we collect some results of  \cite{MB+AF+JLV_boundary-estimates-elliptic, MB+JLV_PM-with-F,CRA+PIE} which we will use to prove the upper bound of GHP. We remember that hypothesis \eqref{N1} provides us a way to compare $F(r)$ with the powers $r^{m_0}$ and $r^{m_1}$. Furthermore, the monotonicity estimates of Bénilan and Crandall, first proven in the case of homogeneous $F$ in \cite{Benilan+Crandall}, are valid for this class of nonlinearities, as shown by Crandall and Pierre in \cite{CRA+PIE}.
	\begin{lem}[\textbf{Consequences of the hypothesis \eqref{N1} \cite[Lemma 10.1]{MB+JLV_PM-with-F}}]\label{Nonlin}
		Let $F:\RR \rightarrow \RR$ a nonlinearity, as in Subsection \ref{sec2.2}, which satisfies \eqref{N1}, or equivalently \eqref{N2}, $$\mu_0\leq \frac{F(r)F''(r)}{(F'(r)^2)}\leq \mu_1.$$
		Then, the following estimates are true:
		\begin{equation}\label{lem2.2_1}
			\left( \frac{r}{r_0} \right)^{m_0}\leq \frac{F(r)}{F(r_0)}\leq \left( \frac{r}{r_0} \right)^{m_1}\hspace{4mm}\mbox{for }0\leq r_0\leq r.
		\end{equation}
		\begin{equation}\label{lem2.2_2}
			\underline{k}\left( \frac{r}{r_0} \right)^{m_1}\leq \frac{F(r)}{F(r_0)}\leq \overline{k} \left( \frac{r}{r_0} \right)^{m_0}\hspace{4mm}\mbox{for }0\leq r\leq r_0.
		\end{equation}
		where we have denoted $\underline{k}= \left( \frac{m_0}{m_1} \right)^{m_1}<1$ and $\overline{k}= \left( \frac{m_1}{m_0} \right)^{m_0}>1$.

	\end{lem}
	
	\begin{lem}[\textbf{Benilan-Crandall estimates \cite{CRA+PIE}}]\label{benilancrandall}
		Let $\A$ satisfy \eqref{A1}, \eqref{A2}, \eqref{K3} and let $F$ satisfy \eqref{N1}. If $u$ is the unique \solution\,for the initial datum $0\leq u_0 \in \LL^1_{\Phi_1}(\Omega)$, then, the following inequality holds (in the sense of distributions):
		\begin{equation*}\label{BC1}
			\partial_t u\geq - \frac{1}{\mu_0 t}\frac{F(u)}{F'(u)},
		\end{equation*}
		where the quotient $F/F'$ vanish when $u=0$ or $F(u)=F'(u)=0$.\\
		As a consequence, for a.e. point $x\in \Omega$ and $0<\tau\leq t$, we have
		\begin{equation}\label{BC2}
			t^{\frac{m_0 }{m_0-1}} F(u(t,x))\geq \tau^{\frac{m_0}{m_0-1}} F(u(\tau,x))
		\end{equation}
		Moreover, If we take into account $F(u)/F'(u)\leq (1-\mu_0) u$, we can reformulate the previous results as follows
		\begin{equation*}\label{BC3}
			\partial_t u \geq - \frac{1-\mu_0}{\mu_0} \frac{u}{t},
		\end{equation*}
		for a.e. $x\in \Omega$ and $0<\tau\leq t$.
		\begin{equation}\label{BC4}
			t ^{\frac{1}{m_0-1}} u(t,x)\geq \tau^{\frac{1}{m_0-1}} u(\tau,x).
		\end{equation}
	\end{lem}
	
	Here is where the assumptions \eqref{K1}, \eqref{K2} and \eqref{K3} play an important role in computing the Green estimates. The key element is the pointwise estimate for weak dual solutions $u\in S$, Proposition \ref{pointwiseestimates}, that allows us to obtain both the upper bounds (absolute and upper behaviour) and the smoothing effects.
	\begin{lem}[\textbf{Green function estimates I \cite[Lemma 4.2]{MB+AF+JLV_boundary-estimates-elliptic}}]\label{GestimatesI}
		Let $\A$ satisfy \eqref{A1}, \eqref{A2}, \eqref{K1} and $0<q<\frac{N}{N-2s}$. Then, there exists a constant $c_2>0$ such that:
		\begin{equation*}\label{Gestimates1}
			\sup_{x_0\in \Omega} \int_{\Omega} \K(x,x_0)^q \dx \leq c_2.
		\end{equation*}
		For each parameter $q$ we define
		
		$$	B_q(\Phi_1(x_0))=\left\{\begin{array}{lll}\Phi_1(x_0) &0<q<\frac{N}{N-2s+\gamma},\\ \Phi_1(x_0) (1+ |\log \Phi_1(x_0)|^{\frac{1}{q}}) &q=\frac{N}{N-2s+\gamma},\\ \Phi_1(x_0)^{\frac{N-q(N-2s)}{\gamma q}} &\frac{N}{N-2s+\gamma}<q<\frac{N}{N-2s}.\end{array}\right.$$
		
		Then, if $\A$ satisfy \eqref{K2}, there exists constants $c_3,c_4>0$ such that
		\begin{equation*}\label{Gestimates2}
			c_3\Phi_1(x_0)\leq \left( \int_{\Omega} \K(x,x_0)^q \dx \right)^{\frac{1}{q}}\leq c_4 B_q(\Phi_1(x_0)).
		\end{equation*}
		Constants $c_j$, $j=2,3,4$ only depends on N, s, $\gamma$, q, $\Omega$ and they have an explicit expression.
	\end{lem}
	Note that in the especial case $q=1$ we have
	\begin{equation*}\label{B1}
		B_1(\Phi_1(x_0))= \left\{\begin{array}{lll}\Phi_1(x_0) &\gamma<2s\\ \Phi_1(x_0) (1+ |\log \Phi_1(x_0)|) &\gamma=2s\\ \Phi_1(x_0)^{\frac{2s}{\gamma }} &\gamma >2s,\end{array}\right.
	\end{equation*}
	which implies that for $\epsilon$ small enough
	$$B_1(\Phi_1(x_0))\leq C \left\{\begin{array}{lll}\Phi_1(x_0) &\gamma<2s\\ \frac{1}{\epsilon}\Phi_1(x_0)^{1-\epsilon} &\gamma=2s\\ \Phi_1(x_0)^{\frac{2s}{\gamma }} &\gamma >2s.\end{array}\right.$$
	
	\begin{prop}[\textbf{\cite[Proposition 5.1]{MB+JLV_PM-with-F}}]\label{pointwiseestimates}
		Let $\A$ satisfy \eqref{A1}, \eqref{A2}, \eqref{K1} and let $F$ satisfy \eqref{N1}. If $u \in S$ is weak dual solution of \eqref{CDP} with initial datum $0\leq u_0\in \LL^1_{\Phi_1}(\Omega)$. Then, for a.e. $x_0\in \Omega$ and $t>0$ we obtain
		\begin{equation*}\label{pe1}
			\int_{\Omega} u(t,x) \K(x,x_0) \dx\leq \int_{\Omega} u_0(x) \K (x,x_0) \dx.
		\end{equation*}
		Besides, for every $0<t_0\leq t_1\leq t$ and for a.e. $x_0\in \Omega$, the following inequality is true
		\begin{align}\label{pe2}
			\left( \frac{t_0}{t_1} \right)^{\frac{m_0}{m_0-1}} (t_1-t_0) F(u(t_0,x_0))&\leq \int_{\Omega} [u(t_0,x) - u(t_1,x)] \K(x,x_0)\dx\nonumber\\
			&\leq (m_0-1)\left( \frac{t}{t_0^{\frac{1}{m_0}}} \right)^{\frac{m_0}{m_0-1}} F(u(t,x_0)).
		\end{align}
	\end{prop}
	
	Observe that solutions $u\in S$ are bounded. The reason is that, for all $t>0$ we know $u(t)\in \LL^p(\Omega)$ for some $p>\frac{N}{2s}$. Consequently, its conjugate $p'=\frac{p}{p-1}< \frac{N}{N-2s}$. Therefore, by Hölder's inequality
	$$\int_{\Omega} u(t,x) \K(x,x_0)\dx<\infty.$$
	In this way, the first part of inequality \eqref{pe2} guarantees that $u(t)\in \LL^\infty(\Omega)$ $\forall t>0$. However, for $t=0$ this claim is not always true because $u_0\in \LL^1_{\Phi_1}(\Omega)$, so the integral above could be infinite. The key to Proposition \ref{pointwiseestimates} is the middle term of \eqref{pe2}, as it encodes the upper bounds from the first part and the lower bounds for large times from the second.
	
	\underline{Fundamental Upper Bound:}\\
	Starting with the lower bound of \eqref{pe2}, we select $t_1=2t_0>t_0$. By eliminating the non-positive term $-\int_{\Omega} u(t_1,x)\K(x,x_0)\dx$ we get
	\begin{equation}\label{FUB}\tag{FUB}
		F(u(t_0,x_0))\leq \frac{2^{\frac{m_0}{m_0-1}}}{t_0} \int_{\Omega} u(t_0,x) \K(x,x_0)\dx.
	\end{equation}
	This last bound encodes both the smoothing effect and the absolute upper bound, and it is sharp for both large and small times, (see \cite{MB+JLV_harnack-inequality} for more details). Summing up, the absolute upper bounds of Theorem \ref{AbsoluteUpperBounds} are sharp for large times. However, to be more precise for $t$ close to zero, a dependency on the initial datum appears, as is typical in smoothing effects (see Theorems \ref{Smoothing} and \ref{SmoothingExplicit}).
	
	\begin{lem}[\textbf{Green function estimates II \cite[Proposition 6.5]{MB+AF+JLV_boundary-estimates-elliptic}}]\label{GestimatesII}
		Let $\A$ satisfy \eqref{A1}, \eqref{A2}, \eqref{K2}, let $H$ a convex positive function, normalized at zero $H(0)=0$ and $v:\Omega\rightarrow [0,\infty)$ a measurable bounded function such that for all $x_0\in \Omega$
		\begin{equation}\label{Gestimates3}
			H(v(x_0))\leq k_0 \int_{\Omega} v(x)\K(x,x_0)\dx.
		\end{equation}
		Then, the following assertions are true:
			
		\renewcommand{\labelenumi}{\rm (\roman{enumi})}
		\begin{enumerate}
			\item There exists a positive constant $c_5>0$ such that
			\begin{equation*}\label{Gestimates4}
				H(v(x_0))\leq k_0 \int_{\Omega} v(x)\K(x,x_0) \dx \leq c_5 k_0 F^{-1}(\nl (2c_5 k_0)) B_1(\Phi_1(x_0)).
			\end{equation*}
where $\nl$ is the Legendre transform of the convex function $F$, whose definition we recall below in Remark \ref{remupper}\,(d).
			\item Moreover, if there exists $\underline{c}$ and $m_i>1$ such that $H(a)\geq \underline{c}a^{m_i}$ for all $a\in [0,1]$. Then, we define
			$$\sigma_i=\left(1\land \frac{2sm_i}{\gamma (m_i -1)}\right).$$
			\begin{enumerate}
				\item If $\gamma< \frac{2sm_i}{m_i-1}$, there exists $c_6>0$ such that
				\begin{equation}\label{Gestimates5}
					\underline{c} v(x_0)^{m_i}\leq H(v(x_0))\leq k_0 \int_{\Omega} v(x) \K(x,x_0) \dx \leq c_6 k_0^{\frac{m_i}{m_i-1}} \Phi_1(x_0).
				\end{equation}
				\item If $1\geq \gamma \geq \frac{2sm_i}{m_i-1}$, there exists $c_7>0$ such that
				\begin{align}\label{Gestimates6}
					\underline{c} v(x_0)^{m_i}\leq H(v(x_0))&\leq k_0 \int_{\Omega} v(x)\K(x,x_0) \dx\nonumber\\
					 &\leq c_7 k_0^{\frac{m_i}{m_i-1}} \left\{\begin{array}{ll}\Phi_1(x_0) (1+ \log|\Phi_1(x_0)|) &\gamma=\frac{2sm_i}{m_i-1}\\ \Phi_1(x_0)^{\frac{2sm_i}{\gamma (m_i-1)}} &\gamma> \frac{2sm_i}{m_i-1}.\end{array}\right.
				\end{align}
			\end{enumerate}
		\end{enumerate}
		These onstants $c_5,c_6,c_7$ only depend on N, s, $\gamma$, $m_i$, $\Omega$ and all the estimates are sharp.
	\end{lem}
	Note that Lemma \ref{GestimatesII} is valid for any H, v, $m_i$ which satisfy the hypotheses. In particular, it applies to the nonlinearity $F$ with \eqref{N1}, a weak dual solution of \eqref{CDP} $u\in S$ and the parameters $m_i=\frac{1}{1-\mu_i}$, $i=0,1$. Thus, thanks to \eqref{Gestimates5} and \eqref{Gestimates6}, in our case we obtain
	$$u(t,x_0)^{m_i}\leq F(u(t,x_0))\leq c_6 k_0^{\frac{m_i}{m_i-1}} \Phi_1(x_0)^{\sigma_i}.$$
	
	The following theorem tell us how the the solution $u(t)$ is bounded in terms of t.
	\begin{thm}[\textbf{Absolute Upper Bound \cite[Theorem 5.2]{MB+JLV_PM-with-F}}]\label{AbsoluteUpperBounds}
		Let $\A$ satisfy \eqref{A1}, \eqref{A2} and \eqref{K1}, let $F$ satisfy \eqref{N1} and $u\in S$ weak dual solution of \eqref{CDP}. Then, if we write the Legendre Transform of $F$ as $F^*$, there exists constants $k_1$, $k_2$ and $\overline{k_2}$ such that for all $t>0$.
		\begin{equation}\label{AUB1}
			F(\|u(t)\|_{\LL^\infty(\Omega)})\leq \nl \left( \frac{\overline{k_2}}{t} \right).
		\end{equation}
		Futhermore, there exists one nonnegative time $0\leq\tau_1(u_0)\leq k_1$, which depends on the initial datum, such that $||u(t)||_{\LL^\infty(\Omega)}\leq 1$ $\forall t \geq \tau_1(u_0)$ and we also have the following bounds
		\begin{equation}\label{AUB3}
			F(\|u(t)\|_{\LL^\infty(\Omega)})\leq \frac{k_2}{t^{\frac{m_i}{m_i-1}}},
		\end{equation}
		\begin{equation}\label{AUB2}
			\|u(t)\|_{\LL^\infty(\Omega)}\leq \frac{k_2}{t^{\frac{1}{m_i-1}}},
		\end{equation}
		
		where $i=0$ if $t\leq k_1$ and $i=1$ if $t\geq k_1$.
	\end{thm}
	\begin{rem}\rm \label{remupper}
		\renewcommand{\labelenumi}{\rm (\alph{enumi})}
		\begin{enumerate}
			\item The upper bound proves a strong regularization that is independent of the initial datum. Inequality \eqref{AUB1} is intrinsic and more precise, while inequalities \eqref{AUB2} and \eqref{AUB3} are more useful for achieving explicit estimates. The constants $k_1$, $k_2$ and $\overline{k_2}$ are positive, depend only on N, s, $m_0$, $m_1$ and $\Omega$, and have an explicit form given in \cite{MB+JLV_PM-with-F}.
			\item $\tau_1(u_0)$ is the first time for which $||u(t_0)||_{\LL^\infty(\Omega)}\leq 1$ for all $t_0\geq \tau_1(u_0)$ and $k_1$ is an upper bound for $\tau_1(u_0)$ that does not depend on $u_0$. This is important because the behavior of $u\to F(u)$ changes when $u$ is less than or bigger than 1, as noted in Lemma \ref{Nonlin}, where the powers in these estimates change.
			\item Inequality \eqref{AUB2} is valid for all times; however, it is not sharp for small times. The precise upper bound for small times involves the initial datum $u_0$. We will address this topic after discussing the smoothing effects.
			\item If $F$ is a nonlinearity as in Section \ref{sec2.2}, we define the Legendre Transform of $F$ as follows $$\nl (z)=\sup_{r\in \RR}(r z - F(r)) = z (F')^{-1}(z) - F((F')^{-1}z) = r F'(r) - F(r).$$
			Where, $r=(F')^{-1}(z)$ and $z=F'(r)$.
		\end{enumerate}
	\end{rem}

	%%%%%%%%%%%%%%%%%%%%%%%%%%%%%%%%%%%%%%%%%%%%%%%%%%%%%%%%%%%%%%%%
	%%%%%%%%%%%%%%%%%%%%%%%%%%%%%%%%%%%%%%%%%%%%%%%%%%%%%%%%%%%%%%%%
	
	\subsection{Proof of the Upper Bound of the GHP}\label{sec3.2}
	This section is devoted to the proof of the Upper bound for all different versions of GHP that we have presented above, in Section \ref{sec2.4}.
The estimate will be sharp for all nonnegative, nontrivial solutions in the case of (RFL) and (CFL). However, it can be improved in the case of (SFL) when data are small, as  in Theorem \ref{smalldata3}.
	\begin{thm}[\textbf{Upper Boundary Behaviour of Solutions}]\label{UpperBehaviour}
		Let $\A$ satisfy \eqref{A1} \eqref{A2} and \eqref{K2}, let $F$ satisfy \eqref{N1} and $u\in S$ a weak dual solution of \eqref{CDP} with initial datum $0\leq u_0\in \LL^1_{\Phi_1}(\Omega)$. Then, there exists positive constants  $k_3$ and $k_4$ which depends only on N, s, $\gamma$, $m_i$, F and $\Omega$, such that for any $t>0$ and $x\in \Omega$ we have
		\begin{equation}\label{UpperBehaviour1}
			F(u(t,x))\leq k_3 \frac{\Phi_1(x)^{\sigma_1}}{t^{\frac{m_i}{m_i-1}}}.
		\end{equation}
		\begin{equation}\label{UpperBehaviour2}
			u(t,x)\leq k_3 \frac{\Phi_1(x)^{\frac{\sigma_1}{m_1}}}{t^{\frac{1}{m_i-1}}}.
		\end{equation}
		Where $i=0$ if $t\leq k_1$ and $i=1$ if $t\geq k_1$. Moreover, we have an intrinsic inequality
		\begin{equation}\label{UpperBehaviour3}
			F(u(t,x))\leq \frac{k_4}{t} F^{-1}\left( \nl \left( \frac{k_4}{t} \right) \right).
		\end{equation}
	\end{thm}
	\begin{proof}
		We begin by recalling the Fundamental Upper Bound \eqref{FUB}
		$$F(u(t_0,x_0))\leq \frac{2^{\frac{m_0}{m_0-1}}}{t_0} \int_{\Omega} u(t_0,x) \K(x,x_0)\dx.$$
		Additionally, we know $u(t,\cdot)\in \LL^\infty(\Omega)$ $\forall t>0$ by Theorem \ref{AbsoluteUpperBounds} and the nonlinearity $F$ satisfy \eqref{N1}, so  all the hypotheses of Lemma \ref{GestimatesII} are satisfied with $k_0= 2^{\frac{m_0}{m_0-1}}t_0^{-1}$. Hence, for all $t_0>0$ and $x_0\in \Omega$
		$$F(u(t_0,x_0))\leq c_5 \frac{2^{\frac{m_0}{m_0-1}}}{t_0} F^{-1}\left( \nl \left( 2 c_5 \frac{2^{\frac{m_0}{m_0-1}}}{t_0} \right) \right),$$
		which implies \eqref{UpperBehaviour3}.
		
		Now, let's assume first that $t_0\geq k_1$, so $\|u(t_0)\|_{\LL^\infty(\Omega)}\leq 1$, remember (b) of Remark \ref{remupper}. For $r\leq 1$ we have $\underline{k}F(1)r^{m_1}\leq F(r)$ by Lemma \ref{Nonlin}. Hence, applying Lemma \ref{GestimatesII} again, we obtain
		\begin{align*}
			\underline{k}F(1)u(t_0,x_0)^{m_1} &\leq F(u(t_0,x_0))
			\leq k_0 \int_{\Omega} u(t_0,x) \K(x,x_0) \dx
			\leq c_6 k_0^{\frac{m_1}{m_1-1}} \Phi_1(x_0)^{\sigma_1},
		\end{align*}
		therefore, we have shown
		\begin{equation}\label{UpperBehaviour4}
			u(t_0,x_0)\leq k_3 \frac{\Phi_1(x_0)^{\frac{\sigma_1}{m_1}}}{t_0^{\frac{1}{m_1-1}}} \hspace{5mm}\mbox{and} \hspace{5mm} F(u(t_0,x_0))\leq k_3 \frac{\Phi_1(x_0)^{\sigma_1}}{t_0^{\frac{m_1}{m_1-1}}}.
		\end{equation}

		Finally, for $t_0\leq k_1$, we use Benilan-Crandall monotonicity (Lemma \ref{benilancrandall}) and after that, we apply inequalities \eqref{UpperBehaviour4}
		 for $t_0=k_1$. On the one hand,
		\begin{align*}
			t_0^{\frac{m_0}{m_0-1}} F(u(t_0,x_0))&\leq k_1^{\frac{m_0}{m_0-1}} F(u(k_1,x_0))
			\leq k_3 k_1^{\frac{m_0}{m_0-1}-\frac{m_1}{m_1-1}} \Phi_1(x_0)^{\sigma_1},
		\end{align*}
		which finishes the proof of \eqref{UpperBehaviour1}. On the other hand,
		
		\begin{align*}
			t_0^{\frac{1}{m_0-1}}u(t_0,x_0)&\leq k_1^{\frac{1}{m_0-1}} u(k_1,x_0)
			\leq k_3 k_1^{\frac{1}{m_0-1}-\frac{1}{m_1-1}} \Phi_1(x_0)^{\frac{\sigma_1}{m_1}},
		\end{align*}
		which completes the proof of \eqref{UpperBehaviour2}.
	\end{proof}
	Inequality \eqref{UpperBehaviour1} allows us to describe the upper behavior of \solutions\,near the boundary, indeed, it establishes the upper bound for $F(u)$ of the GHP of Theorems \ref{GHPI}, \ref{GHPII}, and \ref{GHPIII}.
	
	\begin{rem}\rm \label{remark}A careful inspection of the above proof reveals we have the following estimate for large times, namely
there exists a constant $\overline{\k_3}$ such that for all $t\geq k_1$
		$$\int_{\Omega} u(t,x) \K(x,x_0) \dx \leq \overline{k_3} \frac{\Phi_1(x_0)^{\sigma_1}}{t^{\frac{1}{m_1-1}}}.$$
		However, we can extend it to small times, thanks to Lemma \ref{benilancrandall}: for $t<k_1$ we have
		\begin{align*}
			t^{\frac{1}{m_0-1}}\int_{\Omega} u(t,x) \K(x,x_0) \dx \leq k_1^{\frac{1}{m_0-1}} \int_{\Omega} u(k_1,x) \K(x,x_0) \dx
			\leq k_1^{\frac{1}{m_0-1}} k_3 \frac{\Phi_1^{\sigma_1}(x)}{k_1^{\frac{1}{m_1-1}}},
		\end{align*}
		hence, we get
		\begin{equation}\label{ecremark}
			\int_{\Omega} u(t,x) \K(x,x_0)\dx \leq \overline{k_3} \frac{\Phi_1^{\sigma_1}(x)}{t^{\frac{1}{m_i-1}}}.
		\end{equation}
		Where, $i=0$ if $t<k_1$ and $i=1$ if $t\geq k_1$. Notice that $\overline{\k_3}$ depends only on N,s, $\gamma$, $m_i$, F and $\Omega$.
We have therefore sharp upper estimates for all times. This will be a essential tool in proving the lower bounds of the next section.
	\end{rem}

%%%%%%%%%%%%%%%%%%%%%%%%%%%%%%%%%%%%%%%%%%%%%%%%%%%%%%%%%%%%%%%%%%%%
%%%%%%%%%%%%%%%%%%%%%%%%%%%%%%%%%%%%%%%%%%%%%%%%%%%%%%%%%%%%%%%%%%%%

\subsection{Upper bounds for small data and small times}

Consider now the SFL with $\sigma_1\leq\sigma_0<1$. In this case, we can improve the previous upper bound, cf. Theorem \ref{UpperBehaviour}, because it is possible to find initial data for which inequality \eqref{UpperBehaviour1} is not sharp. This means that depending on the initial data, there can appear different boundary behaviours, i.e. different powers of $\ph$. More precisely, when $u_0\leq A \Phi_1^{1-2s/\gamma}$ then:
$$u(t)\leq g(t)\Phi_1^{1-2s/\gamma}\ll F^{-1}(\Phi_1^{\sigma_1})\lesssim \Phi_1^{\sigma_1/m_1}.$$
This happens because, when $x$ is close enough to the boundary, $\Phi_1(x)<1$ and moreover, the inequality $\sigma_1<\sigma_0<1$ implies $1-2s/\gamma>\sigma_1/m_1$. This has been observed in \cite{MB+AF+JLV_sharp-global-estimates} for the homogeneous case $F(u)=u^m$, here, due to the lack of such homogeneity, things become technically more complicated, this motivates our extra -yet not so restrictive- assumption $F(ab)\leq C F(a) F(b)$. This is a reasonable hypothesis because our non-linearities are allowed to be  sum of powers, for instance, $F(x)=x^2 + x^4$, which satisfies the extra assumption.

\begin{thm}\label{smalldata1}
	Let $\A$ satisfy \eqref{A1}, \eqref{A2} and \eqref{L2}. Suppose also that $\A$ has a first eigenfunction $\Phi_1 \asymp \dist^\gamma$ with $\sigma_0<1$ and let $F$ be the non-linearity which satisfies \eqref{N1} and $F(ab)\leq C F(a)F(b)$. Assume that for all $x,y\in \Omega$ we have
	\begin{equation}\label{ineq2}
		K(x,y)\leq \frac{c_1}{|x-y|^{N+2s}}\left( \frac{\Phi_1(x)}{|x-y|^\gamma} \land 1 \right) \left( \frac{\Phi_1(y)}{|x-y|^\gamma} \land 1 \right)\hspace{3mm}\mbox{and}\hspace{3mm}B(x)\leq c_1\Phi_1(x)^{-2s/\gamma}.
	\end{equation}
	
	Let $u$ be a minimal weak dual solution to the problem \eqref{CDP} with initial data $0\leq u_0\leq A \Phi_1^{1-2s/\gamma}$, for some $A>0$. Then, we have
	$$u(t,x)\leq \frac{\Phi_1^{1-2s/\gamma}(x)}{\left[ A^{1-m_1} - \tilde{C}t \right]^{\frac{1}{m_1-1}}}\hspace{4mm}\mbox{on }\left[0,T_A\right], \hspace{4mm}\mbox{where }T_A=\frac{1}{\tilde{C}A^{m_1-1}}.$$
\end{thm}

\begin{proof}
	It is enough to prove that $\overline{u}(t,x):= \frac{\Phi_1(x)^{1-2s/\gamma}}{\left[ A^{1-m_1}-\tilde{C} t \right]^{\frac{1}{m_1-1}}}$ is a supersolution, i.e. $\partial_t \overline{u}\geq -\A F(\overline{u})$, because the initial data satisfies $u_0\leq A \Phi_1^{1-2s/\gamma}$ and we have comparison.
	
	We will use the following inequality: for any $\eta>1$ and $M>0$, then letting $\tilde{\eta}=\eta \land 2$, there exists $\tilde{b}(M)>0$ such that
	\begin{equation}\label{ineq}
		a^\eta - b^\eta \leq \eta b^{\eta-1} (a-b)+ \tilde{b} |a-b|^{\tilde{\eta}}\hspace{4mm}\mbox{for all }0\leq a,b\leq M.
	\end{equation}
	Now, we apply inequality \eqref{ineq} to $a=\Phi_1(y)$ and $b=\Phi_1(x)$, $\eta=m_0(1-2s/\gamma)$, recalling $\eta>1$ if and only if $\sigma_0<1$ and both $\Phi_1$ and $\left[ A^{1-m_1}-\tilde{C} t \right]^{\frac{1}{m_1-1}}$ are bounded. If we denote $g(t)=\frac{1}{\left[ A^{1-m_1}-\tilde{C} t \right]^{\frac{1}{m_1-1}}}$, thus, we can also use Lemma \ref{Nonlin} and $F(ab)\leq CF(a)F(b)$ to obtain
	\begin{align*}
		F(\overline{u}(t,y))-F(\overline{u}(t,x))&\leq C \left(F(\Phi_1(y)^{1-2s/\gamma})-F(\Phi_1(x)^{1-2s/\gamma})\right)F(g(t))\\
		&\leq C \left(\Phi_1(y)^\eta - \Phi_1(x)^\eta\right) g(t)^{m_1}\\
		&\leq C\eta \Phi_1(x)^{\eta-1}\left( \Phi_1(y)-\Phi_1(x) \right)g(t)^{m_1} + C\tilde{b} \left( \Phi_1(y)-\Phi_1(x) \right)^{\tilde{\eta}} g(t)^{m_1}\\
		&\leq C\eta \Phi_1(x)^{\eta-1}\left( \Phi_1(y)-\Phi_1(x) \right)g(t)^{m_1} + C\tilde{b} C_\gamma^{\tilde{\eta}} |x-y|^{\tilde{\eta}\gamma} g(t)^{m_1},
	\end{align*}
	where we have used $|\Phi_1(x)-\Phi_1(y)|\leq c_\gamma |x-y|^\gamma$. Using \eqref{ineq2}, we get
	$$\int_{\RN}[\Phi_1(y)-\Phi_1(x)]K(x,y)\dy = -\A\Phi_1(x)+B(x)\Phi_1(x)\leq \lambda_1 \Phi_1(x)+c_1 \Phi_1(x)^{1-2s/\gamma}.$$
	Thus, recalling that $\eta, \tilde{\eta}>1$ and that $\Phi_1$ is bounded, it follows
	\begin{align*}
		-\A F(\overline{u}(t,x))&= \int_{\RN}[F(\overline{u}(t,y))-F(\overline{u}(t,x))]K(x,y)\dy - B(x)F(\overline{u}(t,x))\\
		&\leq C\eta \Phi_1(x)^{\eta-1} g(t)^{m_1} \int_{\RN}\left( \Phi_1(y)-\Phi_1(x) \right)K(x,y)\dy + Cg(t)^{m_1} \int_{\RN} |x-y|^{\tilde{\eta}\gamma} K(x,y)\dy\\
		&+ c_1 \Phi_1(x)^{-2s/\gamma}\Phi_1(x)^\eta g(t)^{m_1}\\
		&\leq\overline{K}\frac{1}{\left[ A^{1-m_1}-\tilde{C} t \right]^{\frac{m_1}{m_1-1}}}\left( \Phi_1(x)^{\eta-2s/\gamma} + \int_{\RN} |x-y|^{\tilde{\eta}\gamma} K(x,y)\dy \right).
	\end{align*}
	  We claim that
	 $$	g'(t)= \frac{\tilde{C}}{m_1-1}g(t)^{m_1}\hspace{3mm}\mbox{and}\hspace{3mm}\int_{\RN} |x-y|^{\tilde{\eta}\gamma} K(x,y)\dy\leq c_4 \Phi_1(x)^{1-2s/\gamma}.$$
	 Assume these last inequalities, therefore, if we choose $\tilde{C}\geq \overline{K} (m_1-1)$ we have finished the proof because $\overline{u}$ is a supersolution
	 \begin{equation*}
	 	\partial_t \overline{u}\geq \overline{K} \frac{\Phi_1(x)^{1-2s/\gamma}}{\left[ A^{1-m_1}-\tilde{C}t \right]^{\frac{m_1}{m_1-1}}}\geq -\A F(\overline{u}).
	 \end{equation*}
	 Finally, we prove the claim. For this, using hypothesis \eqref{ineq2} and choosing $r=\Phi_1(x)^{1/\gamma}$ we have
	 \begin{align*}
	 	\int_{\RN} |x-y|^{\tilde{\eta}\gamma} K(x,y)\dy&\leq c_1 \int_{\B_r(x)} \frac{1}{|x-y|^{N+2s-\tilde{\eta}\gamma}}\dy + c_1\Phi_1(x)\int_{\RN\setminus B_r(x)} \frac{1}{|x-y|^{N+2s+\gamma-\tilde{\eta}\gamma}}\dy\\
	 	&\leq c_2 r^{\tilde{\eta}\gamma-2s}+c_1\frac{\Phi_1(x)}{r^{2s}}\int_{\RN\setminus B_r(x)} \frac{1}{|x-y|^{N+\gamma-\tilde{\eta}\gamma}}\dy\\
	 	&\leq c_4 \Phi_1(x)^{1-2s/\gamma}.
	 \end{align*}
 Where we use that $\tilde{\eta}\gamma-2s>0$ and $\tilde{\eta}>1$.
\end{proof}

This Theorem can be applied to the Spectral Fractional Laplacian, here, the general upper bound of \ref{GHPIII} is improved for small data. For \eqref{L2} type operators with $B(x)\equiv 0$, we can actually prove a better upper bound for smaller data, namely:
\begin{cor}\label{smalldata2}
	In the hipotheses of Theorem \ref{smalldata1}, assume also that $B(x)\equiv 0$ and $u_0\leq A \Phi_1$ for some $A>0$. Then, for a.e. $x\in\Omega$ we have
	$$u(t,x)\leq \frac{\Phi_1(x)}{\left[ A^{1-m_1} - \tilde{C} t \right]^{\frac{1}{m_1-1}}}\hspace{4mm}\mbox{on }\left[0,T_A\right],\hspace{4mm}\mbox{where }T_A=\frac{1}{\tilde{C}A^{m_1-1}}.$$
\end{cor}

Here, we essentialy repeat the proof of Theorem \ref{smalldata1} replacing $1-2s/\gamma$ by $1$. The only fact which we must take into account is that $B\equiv 0$ and $u_0\leq A \Phi_1$.
	
	%%%%%%%%%%%%%%%%%%%%%%%%%%%%%%%%%%%%%%%%%%%%%%%%%%%%%%%%%%%%%%
	%%%%%%%%%%%%%%%%%%%%%%%%%%%%%%%%%%%%%%%%%%%%%%%%%%%%%%%%%%%%%%
	
	\subsection{Smoothing Effects}\label{sec3.3}
	Before starting with the lower bounds, we need to prove some weighted $\LL^1$ estimates. Some fundamental facts to achieve this goal are the next smoothing effects results; see \cite{MB+JLV_PM-with-F} for more details. In the classical case, $\A=-\Delta$, see the monograph \cite{JLV_smooth-decay-estimates}. Here, we recall these bounds of the $\LL^{\infty}$ norm  of solutions for $t>0$, in terms of the weighted $\LL^1$ norm of solutions at time $t$, or even at a previous time $0\leq t_0 \leq t$. For each parameter $\gamma\in (0,1]$ of the hipothesis \eqref{K2}, we define the exponent
	$$\theta_{i,\gamma}=\left( 2s + (N+\gamma) (m_i-1) \right)^{-1}.$$
	
	\begin{thm}[\textbf{Smoothing Effects I \cite[Theorem 6.1]{MB+JLV_PM-with-F}}]\label{Smoothing}
		Let $\A$ satisfy \eqref{A1}, \eqref{A2} and \eqref{K2}, and let $F$ satisfy \eqref{N1}. Suppose $u\in S$ is a weak dual solution of \eqref{CDP} with initial datum $0\leq u_0\in \LL^1_{\Phi_1}(\Omega)$. Then, there exists a constant $k_5>0$ which depends only on $N$,$s$, $m_i$, $\gamma$, $F$ and $\Omega$ such that for all $0\leq t_0\leq t$
		\begin{equation}\label{Smoothing1}
			F(\|u(t)\|_{\LL^\infty(\Omega)})\leq \left\{\begin{array}{ll}k_5 \frac{\|u(t_0)\|_{\LL^1_{\Phi_1}(\Omega)}^{2s m_1 \theta_{1,\gamma}}}{t^{m_1(N+\gamma)\theta_{1,\gamma}}} \hspace{8mm}\mbox{if }t\geq \|u(t_0)\|_{\LL^1_{\Phi_1}(\Omega)}^{\frac{2s}{N+\gamma}}\\ k_5 \frac{\|u(t_0)\|_{\LL^1_{\Phi_1}(\Omega)}^{2s m_0 \theta_{0,\gamma}}}{t^{m_0(N+\gamma)\theta_{0,\gamma}}} \hspace{8mm}\mbox{if }t\leq \|u(t_0)\|_{\LL^1_{\Phi_1}(\Omega)}^{\frac{2s}{N+\gamma}}.\end{array}\right.
		\end{equation}
	\end{thm}
		Note that the powers of the smoothing change for large or small times, depending on the norm $\|u(t_0)\|_{\LL^1_{\Phi_1}(\Omega)}^{\frac{2s}{N+\gamma}}$. We can choose $t_0=t$, so the smoothing happen at the same time, or we can choose $t_0=0$, in which case the bound depends on the initial data $u_0$. Note that, if we compare the upper bounds \eqref{AUB3} and \eqref{Smoothing1}, the power of the time $\frac{m_i}{m_i-1}$ is greater than $m_i(N+\gamma)\theta_{i,\gamma}$. Thus, for small times, $t<\left( k_1 \land \|u(t_0)\|_{\LL^1_{\Phi_1}(\Omega)}^{\frac{2s}{N+\gamma}} \land 1 \right)$, \eqref{Smoothing1} is more precise, and for large time $t> \left( k_1 \lor \|u(t_0)\|_{\LL^1_{\Phi_1}(\Omega)}^{\frac{2s}{N+\gamma}} \lor 1 \right)$, \eqref{AUB3} is better.

	In the case $F(u)=u^m$ with $m>1$, we have $m_0=m_1$, so there is only one case in inequality \eqref{Smoothing1}. Theorem \ref{Smoothing} presents an intrinsic form of the smoothing, which can be made explicit as follows.
	
	\begin{cor}[\textbf{Smoothing Effects II \cite[Corollary 6.3]{MB+JLV_PM-with-F}}]\label{SmoothingExplicit}
		With the same hypotheses as in Theorem \ref{Smoothing}, there exists an explicit constant $k_6$ which depends only on $N$,$s$, $m_i$, $\gamma$, $F$ and $\Omega$ such that for all times $0\leq t_0\leq t$, we have
		
		\begin{equation}\label{Smoothing3}
			\|u(t)\|_{\LL^\infty(\Omega)}\leq k_6 \frac{\|u(t_0)\|_{\LL^1_{\Phi_1}(\Omega)}^{2s\theta_{1,\gamma}}}{t^{(N+\gamma) \theta_{1,\gamma}}}, \hspace{6mm}\mbox{if }t\geq \|u(t_0)\|_{\LL^1_{\Phi_1}(\Omega)}^{\frac{2s}{N+\gamma}}.
		\end{equation}
		
		\begin{equation}\label{Smoothing2}
			\|u(t)\|_{\LL^\infty(\Omega)}\leq k_6 \frac{\|u(t_0)\|_{\LL^1_{\Phi_1}(\Omega)}^{2s\theta_{0,\gamma}}}{t^{(N+\gamma) \theta_{0,\gamma}}},\hspace{6mm}\mbox{if }t\leq \|u(t_0)\|_{\LL^1_{\Phi_1}(\Omega)}^{\frac{2s}{N+\gamma}}.
		\end{equation}
	\end{cor}
	We provide here below a proof of \eqref{Smoothing3}, as a correction of the original proof   in \cite[Corollary 6.3]{MB+JLV_PM-with-F}, since it contained a fixable error.
	\begin{proof}
		We first prove the instantaneous smoothing effects, namely $t_0=t$. The case $t_0<t$ follows from monotonicity of the norms $\LL^1_{\Phi_1}(\Omega)$, see Theorem \ref{WeightEstimates} below.\\	
		\underline{Bound for small times} $t_0\leq \|u(t_0)\|_{\LL^1_{\Phi_1}(\Omega)}^\frac{2s}{N+\gamma}$:
		
		Here, we are in the case of inequality \eqref{Smoothing2}, therefore we have $i=0$ and define the quantity
		$$U_{0,\gamma}(t_0)=\frac{\|u(t_0)\|_{\LL^1_{\Phi_1}(\Omega)}^{2s\theta_{0,\gamma}}}{t_0^{(N+\gamma)\theta_{0,\gamma}}}\geq 1.$$
		If $\|u(t_0)\|_{\LL^\infty(\Omega)}\leq U_{0,\gamma}(t_0)$ there is nothing to prove. Thus, we assume $\|u(t_0)\|_{\LL^\infty(\Omega)}\geq U_{0,\gamma}(t_0)\geq 1$. Applying Lemma \ref{Nonlin} with $r\geq 1$ we get
		\begin{align*}
			F(1)\|u(t_0)\|_{\LL^\infty(\Omega)}^{m_0} &\leq F\left( \|u(t_0)\|_{\LL^\infty(\Omega)} \right)\\
			&\leq k_5 \frac{\|u(t_0)\|_{\LL^1_{\Phi_1}(\Omega)}^{2sm_0\theta_{0,\gamma}}}{t_0^{(N+\gamma) \theta_{0,\gamma} m_0}},
		\end{align*}
		therefore, we have already proved \eqref{Smoothing2}.\\
		\underline{Bound for large times} $t_0\geq \|u(t_0)\|_{\LL^1_{\Phi_1}(\Omega)}^\frac{2s}{N+\gamma}$:
		
		In this case, we have to demostrate inequality \eqref{Smoothing3}. Thus for $i=1$, we similarly define
		$$U_{1,\gamma}(t_0)=\frac{\|u(t_0)\|_{\LL^1_{\Phi_1}(\Omega)}^{2s\theta_{1,\gamma}}}{t_0^{(N+\gamma)\theta_{1,\gamma}}}\leq 1.$$
		By Theorem \ref{Smoothing} we get
		\begin{equation*}
			F\left( \|u(t_0)\|_{\LL^\infty(\Omega)} \right) \leq k_5 \frac{\|u(t_0)\|_{\LL^1_{\Phi_1}(\Omega)}^{2sm_1\theta_{1,\gamma}}}{t_0^{(N+\gamma)m_1 \theta_{1,\gamma}}}= k_5 U_{1,\gamma}(t_0)^{(N+\gamma)m_1 \theta_{1,\gamma}}\leq k_5.
		\end{equation*}
		Therefore, we have $\|u(t_0)\|_{\LL^\infty(\Omega)}\leq F^{-1}(k_5)$ because $F$ is non decreasing. So, we apply Lemma \ref{Nonlin} with $F^{-1}(k_5)=r_0$ to obtain
		\begin{align*}
			F(\|u(t_0)\|_{\LL^\infty(\Omega)})\geq \underline{k} \frac{F(r_0)}{r_0^{m_1}} \|u(t_0)\|_{\LL^\infty(\Omega)}^{m_1}
			= \frac{\underline{k}}{(F^{-1}(k_5))^{m_1}}k_5 ||u(t_0)||_{\LL^\infty(\Omega)}^{m_1},
		\end{align*}
		and using Theorem \ref{Smoothing} again in the inequality above
		\begin{align*}
			\|u(t_0)\|_{\LL^\infty(\Omega)}^{m_1}&\leq \left( \frac{\underline{k}}{(F^{-1}(k_5))^{m_1}}k_5 \right)^{-1} F(\|u(t_0)\|_{\LL^\infty(\Omega)})\\
			&\leq  \frac{(F^{-1}(k_5))^{m_1}}{\underline{k}} \frac{\|u(t_0)\|_{\LL^1_{\Phi_1}(\Omega)}^{2sm_1 \theta_{1,\gamma}}}{t_0^{(N+\gamma) m_1 \theta_{1,\gamma}}}.
		\end{align*}
		Choosing $k_6$ as the maximum of all the constant above, we conclude the proof.
	\end{proof}

	%%%%%%%%%%%%%%%%%%%%%%%%%%%%%%%%%%%%%%%%%%%%%%%%%%%%%%%%%%%%
	%%%%%%%%%%%%%%%%%%%%%%%%%%%%%%%%%%%%%%%%%%%%%%%%%%%%%%%%%%%%
	%%%%%%%%%%%%%%%%%%%%%%%%%%%%%%%%%%%%%%%%%%%%%%%%%%%%%%%%%%%%
	%%%%%%%%%%%%%%%%%%%%%%%%%%%%%%%%%%%%%%%%%%%%%%%%%%%%%%%%%%%%

	\section{Weighted $\LL^1$ - Estimates}\label{sec4}
	
	One of the most important applications of smoothing effects is the Weighted $\LL^1$-estimates, which enable us to determine the lower boundary behavior for solutions of \eqref{CDP}.
	
	\begin{thm}\label{WeightEstimates}
		Let $\A$ satisfy \eqref{A1}, \eqref{A2} and \eqref{K2}, and let $F$ satisfy \eqref{N1}. Suppose $u\in S$ is a weak dual solution of \eqref{CDP} with initial datum $0\leq u_0\in \LL^1_{\Phi_1}(\Omega)$. Then, for any $0\leq \tau \leq t$ we have
		\begin{equation}\label{WeightEstimates1}
			\int_{\Omega} u(t,x) \Phi_1(x)\dx \leq \int_{\Omega} u(\tau,x) \Phi_1(x)\dx.
		\end{equation}
		Besides, for all $0\leq \tau_0\leq \tau \leq t<\infty$, we get
		\begin{equation}\label{WeightEstimates2}
			\int_{\Omega} u(\tau,x)\Phi_1(x)\dx\leq \int_{\Omega} u(t,x)\Phi_1(x)\dx + k_7 |t-\tau|^{2s\theta_{i,\gamma}} \|u(\tau_0)\|_{\LL^1_{\Phi_1}(\Omega)}^{2s(m_i-1)\theta_{i,\gamma} +1},
		\end{equation}
		where $i=0$ if $t\leq \|u(\tau_0)\|_{\LL^1_{\Phi_1}(\Omega)}^{\frac{2s}{N+\gamma}}$ and $i=1$ if $\tau\geq \|u(\tau_0)\|_{\LL^1_{\Phi_1}(\Omega)}^{\frac{2s}{N+\gamma}}$. Constant $k_7$ depends only on N, s, $m_i$, $\gamma$, F, $\Omega$ and $\lambda_1$.
	\end{thm}
	\begin{proof}
		We split the proof in several steps.
		
		\textbf{Step 1}: Monotonicity in time of the $\LL^1_{\Phi_1}$ norm.\\
		Let's differentiate the weighted $\LL^1$ norm at time $t$
		\begin{align*}
			\frac{d}{dt} \int_{\Omega} u(t,x) \Phi_1(x)\dx &= - \int_{\Omega} \A F(u(t,x))\Phi_1(x)\dx
			= -\lambda_1 \int_{\Omega} F(u(t,x))\Phi_1(x) \dx
			\leq 0.
		\end{align*}
		This last derivative is non-positive because $u$ is non negative and the non linearity $F$ is non-decreasing and normalized at zero. Note that, this derivative is devoutly justified by the argument presented in \cite[Proposition 5.1]{MB+JLV_PM-with-F}. Integrating  from $\tau$ to $t$ yields \eqref{WeightEstimates1} that is
		\begin{equation}\label{We3}
			\int_{\Omega} u(t,x) \Phi_1(x) \dx - \int_{\Omega} u(\tau,x)\Phi_1(x) \dx = -\lambda_1 \int_\tau^t \int_{\Omega} F(u(r,x))\Phi_1(x) \dx \dr.
		\end{equation}
		
		\textbf{Step 2}: For every $t\geq \tau_0$ there exists a constant $k_6'$ such that
		$$F(u(t,x))\leq k_6'\frac{\|u(\tau_0)\|_{\LL^1_{\Phi_1}(\Omega)}^{2s(m_i-1)\theta_{i,\gamma}}}{t^{(N+\gamma)(m_i-1)\theta_{i,\gamma}}} u(t,x),$$
		where $i=0$ if $t\leq \|u(\tau_0)\|_{\LL^1_{\Phi_1}(\Omega)}^{\frac{2s}{N+\gamma}}$ and $i=1$ if $t\geq \|u(\tau_0)\|_{\LL^1_{\Phi_1}(\Omega)}^{\frac{2s}{N+\gamma}}$.
		
		$F$ is convex as a consequence of \eqref{N1} thus, for any $0\leq u \leq V$ we have
		\begin{equation}\label{We4}
			F(u)\leq u F'(u)\leq u F'(V)\leq m_1 u \frac{F(V)}{V}.
		\end{equation}
		Moreover, by Lemma \ref{Nonlin} for every $r_0\geq 0$, we have
		\begin{equation}\label{We5}
			F(V)\leq (1 \lor \overline{k}) F(r_0) \left( \frac{V}{r_0} \right)^{m_i},
		\end{equation}
		where $i=0$ if $V\in [0,r_0]$ and $i=1$ if $V\geq r_0$. Combining inequalities \eqref{We4} and \eqref{We5} we obtain
		\begin{equation}\label{We6}
			F(u)\leq m_1 u V^{m_i-1} \frac{(1 \lor \overline{k}) F(r_0)}{r_0^{m_i}}.
		\end{equation}
		 Now, we denote
		$$U_{j,\gamma}= k_6 \frac{\|u(\tau_0)\|_{\LL^1_{\Phi_1}(\Omega)}^{2s\theta_{j,\gamma}}}{t^{(N+\gamma)\theta_{j,\gamma}}},$$
		with $j=0$ if $U_{j,\gamma} \geq k_6$ and $j=1$ if $U_{j,\gamma}\leq k_6$. Then, if we choose $r_0=k_6$ fixed and $U_{j,\gamma}=V \geq u$, the index $j$ in the inequality of Collolary \ref{SmoothingExplicit}, $\|u(t)\|_{\LL^\infty(\Omega)}\leq k_6 \frac{\|u(\tau_0)\|_{\LL^1_{\Phi_1}(\Omega)}^{2s\theta_{j,\gamma}}}{t^{(N+\gamma) \theta_{j,\gamma}}}$, will be the same as $i$ in formula \eqref{We6}. Therefore,
		\begin{align*}
			F(u(t,x))&\leq m_1 \frac{(1 \lor \overline{k}) F(k_6)}{k_6^{m_i}} \left( k_6 \frac{\|u(\tau_0)\|_{\LL^1_{\Phi_1}(\Omega)}^{2s\theta_{i,\gamma}}}{t^{(N+\gamma) \theta_{i,\gamma}}} \right)^{m_i-1} u(t,x)\\ &:= k_6' \frac{\|u(\tau_0)\|_{\LL^1_{\Phi_1}(\Omega)}^{2s\theta_{i,\gamma}(m_i-1)}}{t^{(N+\gamma)\theta_{i,\gamma}(m_i-1)}} u(t,x).
		\end{align*}
		With $i=0$ if $t\leq \|u(\tau_0)\|_{\LL^1_{\Phi_1}(\Omega)}^{\frac{2s}{N+\gamma}}$ and $i=1$ if $t\geq \|u(\tau_0)\|_{\LL^1_{\Phi_1}(\Omega)}^{\frac{2s}{N+\gamma}}$ and $\displaystyle k_6'= m_1 (1 \lor \overline{k}) \frac{F(k_6)}{k_6}$.
		
		\textbf{Step 3}: Bound of $\displaystyle \lambda_1 \int_\tau^t \int_{\Omega} F(u(r,x)) \Phi_1(x) \dx \dr$.\\
		The key is applying Step 2 to $F(u(r,x))$ inside the integral because $r\geq \tau\geq \tau_0$. We shall adopt the following notation until the end of the proof: let  $0\leq \tau_0\leq\tau\leq t<\infty$, and the index $i$ will denote $i=0$ if $t\leq \|u(\tau_0)\|_{\LL^1_{\Phi_1}(\Omega)}^{\frac{2s}{N+\gamma}}$ and $i=1$ if $\tau \geq \|u(\tau_0)\|_{\LL^1_{\Phi_1}(\Omega)}^{\frac{2s}{N+\gamma}}$. Thus,
		\begin{align}\label{We9}
			\lambda_1 \int_\tau^t \int_{\Omega} F(u(r,x)) \Phi_1(x) \dx \dr &\leq \lambda_1 k_6' \int_\tau^t \int_{\Omega} \frac{\|u(\tau_0)\|_{\LL^1_{\Phi_1}(\Omega)}^{2s(m_i-1)\theta_{i,\gamma}}}{(r-\tau_0)^{(N+\gamma)(m_i-1)\theta_{i,\gamma}}} u(r,x) \Phi_1(x) \dx \dr\nonumber\\
			&\leq \lambda_1 k_6' \|u(\tau_0)\|_{\LL^1_{\Phi_1}(\Omega)}^{2s (m_i-1)\theta_{i,\gamma} +1}\int_\tau^t \frac{\dr}{(r-\tau_0)^{(N+\gamma)(m_i-1)\theta_{i,\gamma}}}\nonumber\\
			&\leq \frac{\lambda_1 k_6' \|u(\tau_0)\|_{\LL^1_{\Phi_1}(\Omega)}^{2s(m_i-1)\theta_{i,\gamma} +1}}{2s\theta_{i,\gamma}} \left[ (t-\tau_0)^{2s\theta_{i,\gamma}} - (\tau-\tau_0)^{2s\theta_{i,\gamma}} \right]\nonumber\\
			&\leq \frac{\lambda_1 k_6'}{2s\theta_{i,\gamma}} |t-\tau|^{2s\theta_{i,\gamma}} \|u(\tau_0)\|_{\LL^1_{\Phi_1}(\Omega)}^{2s(m_i-1)\theta_{i,\gamma} +1}.
		\end{align}
		Where we have used the following claims:
		\begin{itemize}
			\item When we integrate, we rewrite the power as follows $$1-(N+\gamma)(m_i-1)\theta_{i,\gamma} = \frac{2s+(N+\gamma)(m_i-1) - (N+\gamma)(m_i-1)}{2s+(N+\gamma)(m_i-1)} = 2s\theta_{i,\gamma}<1.$$
			\item The numerical inequality $(t-\tau_0)^\beta - (\tau-\tau_0)^\beta\leq (t-\tau)^\beta, \hspace{4mm}\mbox{for any }\beta\in (0,1).$
		\end{itemize}
		Finally, joinning \eqref{We9} and \eqref{We3} we obtain \eqref{WeightEstimates2} with the explicit constant is $k_7= \frac{\lambda_1 k_6'}{2s \theta_{1,\gamma}}$.
	\end{proof}
	
	Note that in the intermediate case, i.e. $\tau \leq \|u(\tau_0)\|_{\LL^1_{\Phi_1}(\Omega)}^{\frac{2s}{N+\gamma}} \leq t$, we can also derive a similar inequality of the form
	$$\int_{\Omega} u(\tau,x)\Phi_1(x)\dx\leq \int_{\Omega} u(t,x)\Phi_1(x)\dx + 2k_7 \max_{i=0,1} \left\{ |t-\tau|^{2s\theta_{i,\gamma}} \|u(\tau_0)\|_{\LL^1_{\Phi_1}(\Omega)}^{2s(m_i-1)\theta_{i,\gamma} +1} \right\}.$$
	However, it is no needed in this work.
	
	\begin{cor}\label{corWE}
		Let $\A$ satisfy \eqref{A1}, \eqref{A2} and \eqref{K2}, and let F satisfy \eqref{N1}. Suppose $u\in S$ is a weak dual solution of \eqref{CDP} with initial datum $0\leq u_0\in \LL^1_{\Phi_1}(\Omega)$. Then, for any positive $\tau_0$ and $t$ such that:
		\begin{equation}\label{corWE1}
			0\leq \tau_0\leq t\leq \tau_0 + \left( (2k_7)^{\frac{1}{2s\theta_{i,\gamma}}} \|u(\tau_0)\|_{\LL^1_{\Phi_1}(\Omega)}^{m_i-1} \right)^{-1},
		\end{equation}
		with, $i=0$ if $t\leq \|u(\tau_0)\|_{\LL^1_{\Phi_1}(\Omega)}^{\frac{2s}{N+\gamma}}$ and $i=1$ if $t\geq \|u(\tau_0)\|_{\LL^1_{\Phi_1}(\Omega)}^{\frac{2s}{N+\gamma}}$, the following inequality holds
		\begin{equation}\label{corWE2}
			\frac{1}{2} \int_{\Omega} u(\tau_0,x)\Phi_1(x) \dx \leq \int_{\Omega} u(t,x) \Phi_1(x)\dx.
		\end{equation}
	\end{cor}
	\begin{proof}
		Taking $\tau=\tau_0$ in Theorem \ref{WeightEstimates} and applying condition \eqref{corWE1} to the term $|t-\tau|^{2s\theta_{i,\gamma}}$ in \eqref{WeightEstimates2} directly implies \eqref{corWE2}.
	\end{proof}
	The way we are going to use this result for the lower boundary behaviour of \solutions\,is taking $\tau_0=0$. Thus, Corollary \ref{corWE} tells us that for sufficiently small times i.e. $\displaystyle t\leq \left( (2k_7)^{\frac{1}{2s\theta_{i,\gamma}}} \|u_0\|_{\LL^1_{\Phi_1}(\Omega)}^{m_i-1} \right)^{-1}$, where $i=0$ if $t\leq \|u_0\|_{\LL^1_{\Phi_1}(\Omega)}^{\frac{2s}{N+\gamma}}$ and $i=1$ if $t\geq \|u_0\|_{\LL^1_{\Phi_1}(\Omega)}^{\frac{2s}{N+\gamma}}$, we obtain
		\begin{equation}\label{corWE3}
			\frac{1}{2} \int_{\Omega} u_0(x)\Phi_1(x)\dx \leq \int_{\Omega} u(t,x) \Phi_1(x) \dx.
		\end{equation}

	%%%%%%%%%%%%%%%%%%%%%%%%%%%%%%%%%%%%%%%%%%%%%%%%%%%%%%%%%%%%%%%%
	%%%%%%%%%%%%%%%%%%%%%%%%%%%%%%%%%%%%%%%%%%%%%%%%%%%%%%%%%%%%%%%%
	%%%%%%%%%%%%%%%%%%%%%%%%%%%%%%%%%%%%%%%%%%%%%%%%%%%%%%%%%%%%%%%%
	%%%%%%%%%%%%%%%%%%%%%%%%%%%%%%%%%%%%%%%%%%%%%%%%%%%%%%%%%%%%%%%%
	
	\section{Lower Bounds}\label{sec5}
	In this section we will prove the lower bounds of the Global Harnack Principle (GHP), which will determine the lower boundary behaviour of \solutions\,to \eqref{CDP}. Depending on the operator $\A$ we will observe tree types of lower bounds. The first one, presented in Theorem \ref{GHPI}, applied after a waiting time $t_*$ and allows for a very general kind of operator, namely \eqref{A1}, \eqref{A2} and \eqref{K2}. It is important to recall that the classical Laplacian belong to this case. In fact, the well-known finite speed of propagation implies that a waiting time is necessary, so this result is consistent.
	
	 The two next lower bounds, presented in Theorems \ref{GHPII} and \ref{GHPIII}, ensure the positivity of \solutions, i.e., they are valid for any time $t>0$. The trade-off is that we need to be more exlicit about the operator, specifically \eqref{L1} and \eqref{L2}. Here, the non-local nature of the operator plays a fundamental role in achieving the positivity, which implies infinite speed of propagation.
	
	%%%%%%%%%%%%%%%%%%%%%%%%%%%%%%%%%%%%%%%%%%%%%%%%%%%%%%%%%%%%%%%
	%%%%%%%%%%%%%%%%%%%%%%%%%%%%%%%%%%%%%%%%%%%%%%%%%%%%%%%%%%%%%%%
	
	\subsection{Lower bounds for large time}\label{sec5.1}
	
	The main tools to achieve the lower bound are the pointwise estimates of Proposition \ref{pointwiseestimates} and the weighted-$\LL^1$ estimates of Corolary \ref{corWE}, specifically in the form of inequality \eqref{corWE3}.
	
	\begin{thm}\label{LowerBoundI}
		Let $\A$ satisfy \eqref{A1}, \eqref{A2} and \eqref{K2}, with $\sigma_1=1$, and let $F$ satisfy \eqref{N1}. Suppose $u\in S$ is a weak dual solution of \eqref{CDP} with initial datum $0\leq u_0\in \LL^1_{\Phi_1}(\Omega)$. Then, there exists a waiting time
		\begin{equation}\label{waitingtime}
			t_*= c_*\left( \frac{1}{\|u_0\|_{\LL^1_{\Phi_1}(\Omega)}^{m_1-1}} \lor \frac{1}{\|u_0\|_{\LL^1_{\Phi_1}(\Omega)}^{m_0-1}} \right),
		\end{equation}
		where $c_*$ is given in the proof, such that for all $ t\geq t_*$ and a.e. $x\in \Omega$ we have
		\begin{equation}\label{lowerbound1}
			F(u(t,x))\geq k_8 \frac{\Phi_1(x)}{t^{\frac{m_0}{m_0-1}}}.
		\end{equation}
		For a positive constant $k_8$ which depends on $N$,$s$, $\gamma$, $m_i$, $F$, $\Omega$ and $\lambda_1$.
	\end{thm}
	\begin{proof}
		We split the proof in several steps.
		
		\textbf{Step 1}: Fundamental Lower Bounds.\\
		Starting from \eqref{pe2}, we have for almost every point $x_0\in \Omega$ and any $0<t_0\leq t_1\leq t$
		\begin{equation}\label{LB1}
			\int_{\Omega} u(t_0,x) \K(x,x_0)\dx - \int_{\Omega} u(t_1,x)\K(x,x_0) \dx \leq (m_0-1) \left( \frac{t^{m_0}}{t_0} \right)^{\frac{1}{m_0-1}} F(u(t,x_0)).
		\end{equation}
		Moreover, we have seen in Remark \ref{remark} the following inequality holds for any $\tau>0$
		\begin{equation}\label{LB2}
			\int_{\Omega} u(\tau,x) \K(x,x_0) \dx \leq \overline{k_3}\frac{\Phi_1(x_0)}{\tau^{\frac{1}{m_i-1}}},
		\end{equation}
		where $i=0$ if $\tau<k_1$ and $i=1$ if $\tau\geq k_1$. Then, for $t_1$ large enough i.e.
		\begin{equation}\label{LB3}
			t_1\geq \max_{i=0,1} \left( \frac{2 \overline{k_3} \Phi_1(x_0)}{\int_{\Omega} u(t_0,x) \K(x,x_0) \dx} \right)^{m_i-1},
		\end{equation}
		we get by combining \eqref{LB2} and \eqref{LB3}
		$$\int_{\Omega} u(t_1,x) \K(x,x_0)\dx \leq \frac{1}{2} \int_{\Omega} u(t_0,x)\K (x,x_0) \dx.$$
		Therefore, using this last inequality together with \eqref{LB1} we obtain the Fundamental Lower Bound
		\begin{equation}\label{FLB}\tag{FLB}
			t^{\frac{m_0}{m_0-1}} F(u(t,x_0))\geq \frac{t_0^{\frac{1}{m_0-1}}}{2(m_0-1)} \int_{\Omega} u(t_0,x) \K(x,x_0)\dx,
		\end{equation}
		which is valid whenever $t_0\leq t_1\leq t$. This fact always happen because thanks to \eqref{LB2} we have
		$$t_0\leq \max_{i=0,1}\left( \frac{\overline{k_3} \Phi_1(x_0)}{\int_{\Omega} u(t_0,x)\K(x,x_0)\dx} \right)^{m_i-1} \leq \max_{i=0,1}\left( \frac{2\overline{k_3} \Phi_1(x_0)}{\int_{\Omega} u(t_0,x)\K(x,x_0)\dx} \right)^{m_i-1}\leq t_1.$$
		
		\textbf{Step 2}: Quantitative Lower Bounds.\\
		Note that $t_0$ and $t_1$ have not been fixed yet, we have only restriction \eqref{LB3}. Now, let's fix $t_0$ small enough
		$$t_0=\frac{C}{(2k_8)^{\frac{1}{2s\theta_{0,\gamma}}} \|u_0\|_{\LL^1_{\Phi_1}(\Omega)}^{m_0-1}},$$
		where $C$ is a constant which allow us to use Corollary \ref{corWE} with $\tau_0=0$, we will explain how to choose it later, see Remarck \ref{rem2} below.
		Therefore, we have
		\begin{equation}\label{LB5}
			\frac{1}{2} \int_{\Omega} u_0(x) \Phi_1(x)\dx \leq \int_{\Omega} u(t_0,x)\Phi_1(x)\dx.
		\end{equation}
		Starting from \eqref{FLB} and take into account both \eqref{LB5} and the lower bound of the Green function in \eqref{K2}, i.e., $\K(x,y)\geq c_0 \Phi_1(x)\Phi_1(y)$, we obtain the following quantitative lower bound
		\begin{align*}
			t^{\frac{m_0}{m_0-1}} F(u(t,x_0)) &\geq \frac{t_0^{\frac{1}{m_0-1}}}{2(m_0-1)}\int_{\Omega} u(t_0,x)\K(x,x_0)\dx\\
			&\geq \frac{c_0 \Phi_1(x_0)}{2(m_0-1)} t^{\frac{1}{m_0-1}} \int_{\Omega} u(t_0,x)\Phi_1(x)\dx\\
			&\geq \frac{c_0 \Phi_1(x_0)}{4(m_0-1)} t_0^{\frac{1}{m_0-1}} \int_{\Omega} u_0(x)\Phi_1(x)\dx\\
			&= \frac{c_0}{4(m_0-1) (2k_7)^{\frac{1}{2s(m_0-1)\theta_{0,\gamma}}}} \Phi_1(x_0)\\
			&= k_8 \Phi_1(x_0).
		\end{align*}

		Summing up, we have already shown \eqref{lowerbound1} holds whenever $t\geq t_1$, now, $t_1$ is fixed and it is already given by \eqref{LB3}. The purpose of the waiting time \eqref{waitingtime} is to find an explicit upper bound for $t_1$ which does not depend on $t_0$, $x_0$ and $u$.
		
		\textbf{Step 3:} Critical time $t_*$.\\
		Let's assume for a moment that $t_1= \left( \frac{2 \overline{k_3} \Phi_1(x_0)}{\int_{\Omega} u(t_0,x)\K(x,x_0)\dx} \right)^{m_1-1}$. We use Green estimates \eqref{K2} and inequality \eqref{corWE3} for $t_0$ to obtain an explicit bound
		\begin{align*}
			t_1^{\frac{1}{m_1-1}}&= \frac{2\overline{k_3}\Phi_1(x_0)}{\int_{\Omega} u(t_0,x)\K(x,x_0)\dx}
			\leq \frac{2\overline{k_3}}{c_0 \int_{\Omega} u(t_0,x)\Phi_1(x)\dx}\\
			&\leq \frac{4 \overline{k_3}}{c_0 \|u_0\|_{\LL^1_{\Phi_1}(\Omega)}}
			= \frac{c_*^{\frac{1}{m_1-1}}}{\|u_0\|_{\LL^1_{\Phi_1}(\Omega)}}.
		\end{align*}
		If the maximun of the $t_1$ formula is the other case, we can apply the same argument. Thus, we achieve the expresion \eqref{waitingtime} for the waiting time.
	\end{proof}
	This lower bound for large times, together with the upper bound of Theorem \ref{UpperBehaviour}, proves our first result of the Global Harnack Principle, Theorem \ref{GHPI} of Section \ref{sec2.4}. Remember it is the most general result because it could be applied to both local and non-local operators.
	
	\begin{rem}\label{rem2}\rm
		Let's name the following constants in order to choose $t_0$ properly in the proof above. $A:= (2k_7)^{\left( \frac{1}{2s\theta_{0,\gamma}} - \frac{1}{2s \theta_{1,\gamma}} \right)\frac{1}{m_1-m_0} }$ and $B:= (2k_7)^{- \frac{1}{2s\theta_{0,\gamma}} \frac{N+\gamma}{(N+\gamma)(m_0-1)+2s}}$. Note that we have the next assertions:
		\begin{align*}
			\|u_0\|_{\LL^1_{\Phi_1}(\Omega)}&\leq A, \hspace{3mm} \mbox{ if and only if } \hspace{3mm} \frac{1}{(2k_7)^{\frac{1}{2s\theta_{0,\gamma}}} \|u_0\|_{\LL^1_{\Phi_1}(\Omega)}^{m_0-1}}\leq \frac{1}{(2k_7)^{\frac{1}{2s\theta_{1,\gamma}}} \|u_0\|_{\LL^1_{\Phi_1}(\Omega)}^{m_1-1}}.\\
			\|u_0\|_{\LL^1_{\Phi_1}(\Omega)}&\geq B \hspace{3mm} \mbox{ if and only if } \hspace{3mm} \frac{1}{(2k_7)^{\frac{1}{2s\theta_{0,\gamma}}} \|u_0\|_{\LL^1_{\Phi_1}(\Omega)}^{m_0-1}} \leq \|u_0\|_{\LL^1_{\Phi_1}(\Omega)}^{\frac{2s}{N+\gamma}}.
		\end{align*}
		Therefore, according to the statements of Corollary \ref{corWE}, if we choose $t_0=((2k_7)^{\frac{1}{2s\theta_{0,\gamma}}} \|u_0\|_{\LL^1_{\Phi_1}(\Omega)}^{m_0-1})^{-1}$ in any of these cases above, we can directly apply inequality \eqref{corWE3}. In the intermediate case $A\leq ||u_0||_{\LL^1_{\Phi_1}(\Omega)}\leq B$, we have the following computation
		\begin{align*}
			\frac{1}{(2k_7)^{\frac{1}{2s\theta_{1,\gamma}}} \|u_0\|_{\LL^1_{\Phi_1}(\Omega)}^{m_1-1}} &= \frac{1}{(2k_7)^{\frac{1}{2s\theta_{0,\gamma}}} \|u_0\|_{\LL^1_{\Phi_1}(\Omega)}^{m_0-1}} \frac{(2k_7)^{\frac{1}{2s\theta_{0,\gamma}}-\frac{1}{2s\theta_{1,\gamma}}}}{\|u_0\|_{\LL^1_{\Phi_1}(\Omega)}^{m_1-m_0}}\\
			&\geq \frac{1}{(2k_7)^{\frac{1}{2s\theta_{0,\gamma}}} \|u_0\|_{\LL^1_{\Phi_1}(\Omega)}^{m_0-1}} (2k_7)^{\frac{1}{2s\theta_{0,\gamma}}-\frac{1}{2s\theta_{1,\gamma}}+ \frac{m_1-m_0}{2s\theta_{0,\gamma}}\frac{N+\gamma}{(N+\gamma)(m_0-1)+2s}}.
		\end{align*}
		
		Thus, we denote $$C=\frac{1}{2}\left( 1 \land (2k_7)^{\frac{1}{2s\theta_{0,\gamma}}-\frac{1}{2s\theta_{1,\gamma}}+ \frac{m_1-m_0}{2s\theta_{0,\gamma}}\frac{N+\gamma}{(N+\gamma)(m_0-1)+2s}} \right),$$
		and if we choose $t_0=C\left( (2k_7)^{\frac{1}{2s\theta_{0,\gamma}}} \|u_0\|_{\LL^1_{\Phi_1}(\Omega)}^{m_0-1} \right)^{-1}$, all the hipotheses of Coralary \ref{corWE} are always satisfied.
		
	\end{rem}
	
	%%%%%%%%%%%%%%%%%%%%%%%%%%%%%%%%%%%%%%%%%%%%%%%%%%%%%%%%%%%%%%%%%%
	%%%%%%%%%%%%%%%%%%%%%%%%%%%%%%%%%%%%%%%%%%%%%%%%%%%%%%%%%%%%%%%%%%
	
	\subsection{Approximate solutions}\label{sec5.2}
	Some new tools are required to establish the lower bounds for all positive times, we collect them below; further details can be found in \cite{MB+AF+XR_positivity-and-regularity}. We begin by defining the following class of approximate solutions, that we will use throughout the proof. Let $\delta > 0$ be fixed, and we address the larger approximate problem.
	\renewcommand{\theequation}{AP1}
	\begin{equation}\label{AP1}
		\left\{\begin{array}{cll}\partial_t u_\delta +\A \n(u_\delta) =0, & (0,\infty)\times \Omega.\\ u_\delta(t,x)=\delta, & (0,\infty)\times \RN \setminus\Omega.\\ u_\delta(0,x)=u_0(x) + \delta, & x\in \Omega.\end{array}\right.
	\end{equation}
	Or equivalently, if we write $u_\delta=v_\delta+\delta$
	\renewcommand{\theequation}{AP2}
	\begin{equation}\label{AP2}
		\left\{\begin{array}{cll}\partial_t v_\delta +\A (\n(v_\delta +\delta) - \n(\delta)) =0, & (0,\infty)\times \Omega.\\ v_\delta=0, & (0,\infty)\times \RN \setminus\Omega.\\ v_\delta(0,x)=u_0(x), & x\in \Omega.\end{array}\right.
	\end{equation}
	
	\renewcommand{\theequation}{\thesection.\arabic{equation}}
	\setcounter{equation}{7}
	
	The idea is to utilize problem \eqref{AP1} for a priori estimates and \eqref{AP2} for establishing the existence of solutions. Problem \eqref{AP2} represents an approximate Porous Medium type equation with a nonlinearity $H_\delta(v)= F(v+\delta)-F(\delta)$ satisfying $H_\delta(0)=0$ and $H'_\delta(0)= F'(\delta)>0$. Thus, it admits a solution, implying the existence of solutions for problem \eqref{AP1}. We adopt the solution definition as detailed in Section \ref{sec2.3}, ensuring its uniqueness and existence.
	
	 Some properties of approximate solutions $u_\delta$ include the following:
	\begin{prop}[\textbf{\cite[Apendix II]{MB+AF+XR_positivity-and-regularity}}]\label{propiedadesudelta1}
		Let $u$ be a \solution\,of \eqref{CDP} with initial data $0\leq u_0\in \LL^1_{\Phi_1}(\Omega)$ and let $u_\delta$ be a \solution\,of \eqref{AP1} with the same initial data. Then, the following assertions hold true:
		\renewcommand{\labelenumi}{\rm (\roman{enumi})}
		\begin{enumerate}
			\item Positivity. $u_\delta\geq \delta >0$.
			\item Aproximate solutions are ordered. For all $\delta_1\geq \delta_2$ we have $u_{\delta_1}\geq u_{\delta_2}$.
			\item Aproximate solutions are bigger than solutions of \eqref{CDP}. $u_\delta\geq u$.
		\end{enumerate}
	\end{prop}
	Since the Benilan-Crandall estimates hold true for $u_\delta$, the quantity $t^{\frac{1}{m_0-1}} u_\delta(t,x)$ is monotonically non-decreasing for $t>0$. Similar to previous sections, we can prove the positivity (by comparison) and boundedness of solutions:
	$$\delta\leq u_\delta(t,x)\leq \frac{k_2}{t^{\frac{1}{m_i-1}}} +2\delta.$$
	Thus, we can employ the arguments from \cite{MB+AF+XR_positivity-and-regularity, MB+AF+JLV_sharp-global-estimates} to establish the regularity of approximate solutions, i.e.
	\begin{itemize}
		\item $u_\delta$ is globally Holder continuous in $x$ and $t$.
		\item $u_\delta$ is classical in the interior, $u_\delta\in C_x^\infty(\Omega)$ and $u_\delta\in C_t^{1,\alpha}([t_0,T])$ for an $\alpha>0$.
	\end{itemize}

This regularity of $\ud$ justifies the calculations that we will perform the rest of this section.
	
	\begin{lem}\label{propiedadesudelta2}
		Let $\A$ satisfy \eqref{A1} and \eqref{A2}, and let $F$ satisfy \eqref{N1}. Assume $u$ and $\ud$ are \solutions\,of \eqref{CDP} and \eqref{AP1} respectively with the same initial datum $0\leq u_0\in \LL^1_{\Phi_1}(\Omega)$. Then, for all $t>0$ we have
		$$\| \ud (t)-u(t)\|_{\LL^1_{\Phi_1}(\Omega)}\leq \|u_\delta (0)-u_0\|_{\LL^1_{\Phi_1}(\Omega)}= \delta \|\Phi_1\|_{\LL^1(\Omega)}.$$
	\end{lem}
	\begin{proof}By a direct computation we obtain:
		\begin{align*}
			\int_{\Omega} \left( \ud (t,x) - u(t,x) \right) \Phi_1(x) \dx &- \int_{\Omega} \left( \ud (0,x) - u_0(x) \right) \ph(x) \dx\\
			&= \int_0^t \int_{\Omega} \partial_t (\ud -u) (\tau,x) \ph(x) \dx \dtau\\
			&= \int_0^t \int_{\Omega} (\A [F(\ud)] - \A [F(u)]) (\tau,x) \ph(x) \dx \dtau\\
			&=- \lambda_1 \int_0^t \int_{\Omega} \left( F( \ud ) - F(u) \right) (\tau,x) \ph(x) \dx \dtau \leq 0.
		\end{align*}
	\end{proof}
	
	Note that, as $\delta$ decreases to zero, the sequence $\left\{ \ud \right\}$ is monotonically decreasing and bounded below by $u$. This, combined with the convergence in $\LL^1_{\Phi_1}(\Omega)$ allows us to conclude the following pointwise convergence. For all $t>0$, we have
	\begin{equation}\label{pointconvudelta}
		\lim\limits_{\delta \to 0^+} \ud (t,x) = u(t,x) \hspace{6mm}\mbox{for a.e. }x\in \Omega.
	\end{equation}
	Before proceeding with the lower bounds, we divide the proofs into two cases: when the initial datum is small or large relative to the norm  $\|u_0\|_{\LL^1_{\ph(\Omega)}}$, i.e. when the critical time $t_*$, defined by \eqref{waitingtime} has the power $m_1-1$ or $m_0-1$.
	
	%%%%%%%%%%%%%%%%%%%%%%%%%%%%%%%%%%%%%%%%%%%%%%%%%%%%%%%%%%%%%%%%%%%%
	%%%%%%%%%%%%%%%%%%%%%%%%%%%%%%%%%%%%%%%%%%%%%%%%%%%%%%%%%%%%%%%%%%%%
	
	\subsection{Positivity for small initial data}\label{sec5.3}
	Throughout this subsection the waiting time will have the expression
\begin{equation}\label{tstar.111}
t_*=\frac{c_*}{\|u_0\|_{\LL^1_\ph (\Omega)}^{m_1-1}},
\end{equation}
given that our initial assumption is $\|u_0\|_{\LL^1_\ph(\Omega)}\leq 1$. Let us begin with an essential property of approximate solutions that compares the $\LL^1_{\Phi_1}(\Omega)$ norm of $u_0$ and $\ud$.
	\begin{lem}\label{propiedadesudelta3}
		Let $\A$ satisfy \eqref{A1}, \eqref{A2} and \eqref{K2}, and let $F$ satisfy \eqref{N1}. Let $u$ be the \solution\,of \eqref{CDP} with initial data $0\leq u_0 \in \LL^1_\ph(\Omega)$ and $\ud$ be the approximate solution of \eqref{AP1} with the same initial data, and let $t_*$ as in \eqref{tstar.111}. Then,  there exists a constant $\underline{c}$ such that for all $t\in [0,t_*]$ we have
		\begin{align}\label{cotaudelta}
			\underline{c} \|u_0\|_{\LL^1_\ph(\Omega)}^{m_1} \leq \int_{\Omega} F(u(t,x)) \ph(x)\dx
			 \leq \int_{\Omega} F(\ud (t,x)) \ph(x)\dx.
		\end{align}
	\end{lem}
	\begin{proof}
		We split the proof in two steps.\\
		\textbf{Step 1:} We claim that there exists a constant $C$ such that for $t_0=C\left( (2k_7)^{\frac{1}{2s\theta_{1,\gamma}}} \|u_0\|_{\LL^1_{\Phi_1}(\Omega)}^{m_1-1} \right)^{-1}$, the hyphoteses of Corollary \ref{corWE} are satisfied.
		
		Let's define the constant $M:= \left( 2 (2k_7)^{\frac{1}{2s\theta_{1,\gamma}}} \right)^{-(N+\gamma)\theta_{1,\gamma}}$. Then, if we have $\|u_0\|_{\LL^1_\ph(\Omega)}\leq M$, a direct calculation shows us that we can use Corollary \ref{corWE} because
		$$\|u_0\|_{\LL^1_{\Phi_1}(\Omega)}^{\frac{2s}{N+\gamma}}\leq\frac{1}{2(2k_7)^{\frac{1}{2s\theta_{1,\gamma}}} \|u_0\|_{\LL^1_{\Phi_1}(\Omega)}^{m_1-1}}:=t_0\leq \frac{1}{(2k_7)^{\frac{1}{2s\theta_{1,\gamma}}} \|u_0\|_{\LL^1_{\Phi_1}(\Omega)}^{m_1-1}}.$$
		In the other case, when the norm is bounded from above and from below, $M\leq \|u_0\|_{\LL^1_{\Phi_1}(\Omega)}\leq 1$, we consider the quantity $\min\limits_{i=0,1}\left( (2k_7)^{\frac{1}{2s\theta_{i,\gamma}}} ||u_0||_{\LL^1_{\Phi_1}(\Omega)}^{m_i-1} \right)^{-1}$. On the one hand, if the minimum occurs at $i=1$, we are in the same scenario as above. On the other hand, if we have $i=0$ in the minimum we can proceed as in Remark \ref{rem2} to choose
		$$C=\frac{1}{2}\left( 1 \land (1/2)^{m_1-m_0} (2k_7)^{\frac{1-(m_1-m_0)(N+\gamma)\theta_{1,\gamma}}{2s\theta_{1,\gamma}}-\frac{1}{2s\theta_{0,\gamma}}} \right).$$
		
		\textbf{Step 2:} Proof of the bound \eqref{cotaudelta}.
		
		Choose $t_0$ as in Step 1. Thus, thanks to Corolary \ref{corWE} the next inequality holds, for all $0<t\leq t_0$
		$$\frac{1}{2}\int_{\Omega} u_0(x)\ph(x)\dx\leq \int_{\Omega} u(t,x)\ph(x)\dx.$$
		Using the fact $F$ is convex and increasing, and applying Lemma \ref{Nonlin} we get
		\begin{align*}
			\underline{k} F(1/2) \|u_0\|_{\LL^1_\ph(\Omega)}^{m_1}&\leq F \left( \frac{1}{2} \int_{\Omega} u_0(x) \ph(x) \dx \right)
			 \leq \int_{\Omega} F(u(t,x)) \ph(x)\dx
			 \leq \int_{\Omega} F(\ud(t,x))\ph(x)\dx.
		\end{align*}
		Therefore, inequality \eqref{cotaudelta} is proved for any $0<t\leq t_0$. Now, for all $t\in [t_0,t_*]$ observe that $t_0= A_1 t_*$ by construction of $t_0$. Then, we use the Bénilan-Crandall estimates, Lemma \ref{benilancrandall}, and Corollary \ref{corWE} for  $t_0$ as follows:
		\begin{align*}
			\frac{1}{2} \int_{\Omega} u_0(x)\ph(x)\dx &\leq \int_{\Omega} u(t_0,x)\ph(x)\dx
			 \leq \left( \frac{t}{t_0} \right)^{\frac{1}{m_0-1}} \int_{\Omega} u(t,x)\ph(x)\dx
			 \leq \frac{1}{A_1^{\frac{1}{m_0-1}}} \int_{\Omega} u(t,x)\ph(x)\dx,
		\end{align*}
		and using F is convex and increasing again, remember $u\leq \ud$, we get
		\begin{align*}
			\underline{k}F\left( \frac{A_1^{\frac{1}{m_0-1}}}{2} \right) \|u_0\|_{\LL^1_\ph(\Omega)}^{m_1}&\leq F \left( \frac{A_1^{\frac{1}{m_0-1}}}{2} \|u_0\|_{\LL^1_\ph(\Omega)} \right)
			\leq F\left( \int_{\Omega} u(t,x) \ph(x)\dx \right)\\
			&\leq \int_{\Omega} F(u(t,x)) \ph(x)\dx
			\leq \int_{\Omega} F(\ud (t,x)) \ph(x)\dx.
		\end{align*}
	\end{proof}
	To prove the quantitative lower bounds, we need to be more precise about the operator $\A$. This is why we require either \eqref{L1}, for example, for RFL or CFL, or \eqref{L2}, for example, for SFL. The non-local nature of the operator plays a fundamental role in the proof.
	\begin{thm}\label{positivity1}
		Let $\A$ satisfy \eqref{A1}, \eqref{A2} and \eqref{L1}, and let F satisfy \eqref{N1}. Suppose $u$ is the \solution\,of \eqref{CDP} with initial datum $0\leq u_0\in \LL^1_\ph(\Omega)$. Given the waiting time as in \eqref{tstar.111}, we establish:
		
		Assume either $\sigma_1=1$ or $\sigma_1<1$, $K(x,y)<c_1|x-y|^{-(N+2s)}$ and $\ph\in C^\gamma(\overline{\Omega})$. Then, there exists a constant $k_{9}>0$ which depends only on $N$, $s$, $m_i$, $\gamma$, $F$, $\Omega$ and $\lambda_1$ such that for all $t>0$ and a.e. $x\in \Omega$ it follows that
		\begin{equation}\label{boundpositivity1}
			F(u(t,x))\geq k_{9} \left( 1 \land \frac{t}{t_*} \right)^{\frac{m_1^2}{m_1-1}} \frac{\ph^{\sigma_1}(x)}{t^{\frac{m_i}{m_i-1}}}.
		\end{equation}
		Where $i=1$ if $t\leq t_*$ and $i=0$ if $t\geq t_*$.
	\end{thm}
	\begin{proof}
		We split the proof into three steps. Firstly, we estimate a lower bound for $F(u(t,x))$ which is crucial for the proof. Afterward, we derive inequality \eqref{boundpositivity1} for small times $t<t_*$, and finally for large times $t\geq t_*$.
		
		\textbf{Step 1: Claim.} There exists a constant $k>0$ small enought such that $\forall t\in [0,t_*]$
		\begin{equation}\label{posIec1}
			F(u(t,x))\geq \left( k \|u_0\|_{\LL^1_{\Phi_1}(\Omega)}^{m_1} t \right)^{m_1} \phi_1(x)^{\sigma_1}.
		\end{equation}
		\underline{Proof of the claim:} Step 1.1. \textit{Aproximate Solutions.}
		
		We fix $\delta>0$ until the end of the proof of the claim and let $u_\delta$  the \solution\,of \eqref{AP1} with initial data $u_0$. Define $K_0=k\|u_0\|_{\LL^1_{\Phi_1}(\Omega)}^{m_1}$ where k is a constant to be determined. Now, we consider the lower barrier $$\Psi(t,x)=F^{-1}(K_0^{m_1} t^{m_1} \Phi_1(x)^{\sigma_1}).$$
		\begin{rem}\label{obser1}\rm
				Note the following observations, whose estimates we will use throughout the proof.
					\renewcommand{\labelenumi}{\rm (\alph{enumi})}
			\begin{enumerate}
				\item \label{puntoa} We have $K_0\leq 1$ because we have $\|u_0\|_{\LL^1_{\ph(\Omega)}}\leq 1$ and we will choose $k<1$.
				\item \label{puntob} Here, with the goal of simplifying the calculation, $\Phi_1$ is the first eigenfunction of $\A$ normalized, i.e., $\|\Phi_1\|_{\LL^\infty(\Omega)}=1$.
				\item \label{puntoc} We have $K_0^{m_1-1} t_*^{m_1}\leq \overline{k_0}$ by definition of $K_0$ and  $t_*$.
				\item \label{puntod} For any $t\in [0,t_*]$ we have $K_0 t\leq K_0 t_*=k\|u_0\|_{\LL^1_{\Phi_1}(\Omega)}c_*\leq k c_*:=r_0$. Therefore, changing the value of k if it is needed, we get by Lemma \ref{Nonlin} $$F^{-1}\left( \frac{F(kc_*)}{(kc_*)^{m_1}} (K_0t)^{m_1} \right)\leq K_0t.$$
			\end{enumerate}
		\end{rem}
		
		Step 1.2.\textit{ $\Psi$ is a lower barrier for $\ud$, }i.e. $F(\Psi(t,x))< F(u_{\delta}(t,x))$ for any $(t,x)\in [0,t_*] \times \overline{\Omega}$.\\
		We prove this by contradiction. Assume that the inequality is false. By the regularity of $F$, $\ud$ and $\Psi$, there exists a first touching point $(t_c,x_c)\in [0,t_*] \times \overline{\Omega}$ such that $F(\Psi(t_c,x_c))=F(\ud(t_c,x_c))$. Note that $(t_c,x_c)\in (0,t_*)\times \Omega$ because if $t_c=0$ or $x_c\in \partial \Omega$ we have the following: $$F(\Psi(t_c,x_c))=0<F(\delta)\leq F(\ud(t_c,x_c)).$$
		
		Now, let's bound $-\A[F(\ud(t_c,x_c))-F(\Psi(t_c,x_c))]$ from above and from below. Firstly, we use the equation of \eqref{AP1} to establish the upper bound
		$$-\A [F(\ud(t_c,x_c))-F(\Psi(t_c,x_c))]= \partial_t \ud(t_c,x_c) + \A F(\Psi(t_c,x_c)),$$
		and we bound every single term of the equation above. On the one hand, as $(t_c,x_c)$ is the first touching point, what means $F(\Psi(t_c,x_c))=F(\ud(t_c,x_c))$ and $F(\Psi(t,x_c))<F(\ud(t,x_c))$ for any time $0<t<t_c$, we have
		\begin{align}\label{posIec2}
			\partial_t \ud(t_c,x_c)\leq \partial_t \Psi(t_c,x_c) &= \frac{m_1 K_0^{m_1} t_c^{m_1-1} \Phi_1(x)^{\sigma_1}}{F'(F^{-1}(K_0^{m_1}t_c^{m_1}\Phi_1(x)^{\sigma_1}))}\nonumber\\
			&\leq \frac{m_1 K_0^{m_1} t_c^{m_1-1} \Phi_1(x)^{\sigma_1}}{m_0 K_0^{m_1} t_c^{m_1} \Phi_1(x)^{\sigma_1}} F^{-1}(K_0^{m_1} t_c^{m_1} \Phi_1(x)^{\sigma_1})\nonumber\\
			&\leq \frac{m_1}{m_0} K_0.
		\end{align}
		Where we have used \eqref{N1}, more precisely, $m_0 F(r)/r\leq F'(r)\leq m_1 F(r)/r$ and Remark \ref{obser1}\,$\rm{(d)}$. On the other hand, if $\sigma_1=1$ we have
		\begin{equation}\label{posIec3}
			\A F(\Psi(t_c,x_c))=K_0^{m_1} t_c^{m_1}\lambda_1 \Phi_1(x_c)\leq K_0 \overline{k_0}\lambda_1.
		\end{equation}
		Remember Remark \ref{obser1}\,$\rm{(c)}$, which establishes $K_0^{m_1-1} t_c^{m_1}\leq K_0^{m_1-1} t_*^{m_1}\leq \overline{k_0}$. Now, if $\sigma_1<1$, using the facts that $\Phi_1\in C^{\gamma}(\overline{\Omega})$ and $g(x)=x^{\sigma_1}$ is concave, we have
		$$|\Phi_1(x)^{\sigma_1} - \Phi_1(y)^{\sigma_1}|\leq |\Phi_1(x)-\Phi_1(y)|^{\sigma_1}\leq |x-y|^{\gamma \sigma_1}.$$
		Observe that, here $\gamma\sigma_1>2s$ and $K(x,y)\leq c_1|x-y|^{-N-2s}$ by assumption. Thus, we obtain
		\begin{align}\label{posIec4}
			\A F(\Psi(t_c,x_c)) = K_0^{m_1} t_c^{m_1} \A \Phi_1^{\sigma_1}(x_c) &= K_0^{m_1} t_c^{m_1} P.V \int_{\RN} |\Phi_1(x_c)^{\sigma_1} - \Phi_1(y)^{\sigma_1}| K(x_c,y) \dy\nonumber \\
			&\leq K_0^{m_1} t_c^{m_1} \left(\int_{\Omega} c |x_c-y|^{\gamma \sigma_1} \frac{c_1}{|x_c-y|^{N+2s}}\dy + \int_{\RN \setminus \Omega} \frac{2 c_1}{|x_c-y|^{N+2s}}\right)\dy\nonumber\\
			&\leq \overline{C}K_0 \overline{k_0}.
		\end{align}
		Therefore, putting together \eqref{posIec2}, \eqref{posIec3} and \eqref{posIec4} we achieve the following upper bound
		\begin{equation}\label{posIec5}
			-\A[F(\ud(t_c,x_c))-F(\Psi(t_c,x_c))]\leq K_0 \left( \frac{m_1}{m_0} + \overline{k_0} (\overline{C}+\lambda_1) \right).
		\end{equation}
		
		Secondly, let's establish the lower bound. Here we use that the operator is nonlocal, specifically property \eqref{L1}, which states $\inf_{x,y\in \Omega} K(x,y)\geq k_\Omega>0$. Remember that $F(\ud(t_c,x_c))=F(\Psi(t_c,x_c))$.
		\begin{align}\label{posIec6}
			-\A [F(\ud(t_c,x_c)) &- F(\Psi(t_c,x_c))]\nonumber\\ &= - P.V \int_{\RN} [F(\ud(t_c,x_c) - F(\Psi(t_c,x_c))) - (F(\ud(t_c,y)) -F(\Psi(t_c,y)))] K(x_c,y) \dy\nonumber \\
			&= \int_{\RN} F(\ud(t_c,y)) K(x_c,y) \dy - \int_{\RN} F(\Psi(t_c,y)) K(x_c,y)\dy\nonumber\\
			&\geq k_{\Omega} \int_{\RN} F(\ud(t_c,y)) \dy - k_{\Omega} \int_{\RN} (K_0 t_c)^{m_1}\Phi_1(y)^{\sigma_1}\dy\nonumber\\
			&\geq k_{\Omega} \int_{\RN} F(\ud(t_c,y)) \dy - k_{\Omega} K_0 \overline{k_0} |\Omega|.
		\end{align}
		Where we have used Remark \ref{obser1} again. Summing up, thank to \eqref{posIec5} and \eqref{posIec6} we know that
		\begin{equation}\label{posIec7}
			k_{\Omega} \int_{\Omega} F(\ud(t_c,x)) \dx \leq K_0\left( k_{\Omega} \overline{k_0} |\Omega| + \frac{m_1}{m_0} + \overline{k_0} (\overline{C}+\lambda_1) \right).
		\end{equation}
		Finally, this is the moment when we realize we have a contradiction. By combining inequality \eqref{posIec7} and Lemma \ref{propiedadesudelta3}, we obtain
		\begin{equation*}
			k_{\Omega} \underline{c} \|u_0\|_{\LL^1_{\Phi_1}(\Omega)}^{m_1}\leq k \|u_0\|_{\LL^1_{\Phi_1}(\Omega)}^{m_1} \left( k_{\Omega} \overline{k_0} |\Omega| + \frac{m_1}{m_0} + \overline{k_0} (\overline{C}+\lambda_1) \right),
		\end{equation*}
		where the norm of the initial data is not zero and all constants are fixed except $k$. Thus, we can choose $k$ small enough to create a contradiction and conclude $F(\Psi)<F(\ud)$ in $[0,t_*]\times \overline{\Omega}$.
		
		Step 1.3. \textit{Taking the limit as $\delta\to 0^+$.}\\
		Previously, we have seen the convergence \eqref{pointconvudelta}, $\lim_{\delta \to 0^+} \ud(t,x)=u(t,x)$ for all $t>0$ and a.e $x\in \Omega$. Thus, by the continuity of F and the last step, we conclude that for any $t\in [0,t_*)$ and a.e. $x\in \Omega$ we have
		$$F(u(t,x)) = \lim_{\delta \to 0} F(\ud(t,x))\geq F(\Psi(t,x))= (k\|u_0\|_{\LL^1_{\Phi_1}(\Omega)}^{m_1} t)^{m_1} \Phi_1(x)^{\sigma_1}.$$
		
		\textbf{Step 2:} The following inequality holds for all $t\in (0,t_*]$
		\begin{equation}\label{posIec8}
			F(u(t,x))\geq \underline{K} \left( \frac{t}{t_*} \right)^{\frac{m_1^2}{m_1-1}} \frac{\Phi_1(x)^{\sigma_1}}{t^{\frac{m_1}{m_1-1}}}.
		\end{equation}
		We just have to rewrite \eqref{posIec1} as follows
		\begin{align*}
			F(u(t,x))&\geq (k \|u_0\|_{\LL^1_{\Phi_1}(\Omega)}^{m_1} t)^{m_1} \Phi_1(x)^{\sigma_1}\\
			&= \left( k \frac{c_*^{\frac{m_1}{m_1-1}}}{t_*^{\frac{m_1}{m_1-1}}}  \frac{t^{\frac{m_1}{m_1-1}}}{t^{\frac{1}{m_1-1}}} \right)^{m_1} \Phi_1(x)^{\sigma_1}
			\geq \underline{K} \left( \frac{t}{t_*} \right)^{\frac{m_1^2}{m_1-1}} \frac{\Phi_1(x)^{\sigma_1}}{t^{\frac{m_1}{m_1-1}}}.
		\end{align*}
		
		\textbf{Step 3:} For all $t> t_*$ we have
		\begin{equation}\label{posIec9}
			F(u(t,x))\geq \underline{K} \frac{\Phi_1(x)^{\sigma_1}}{t^{\frac{m_0}{m_0-1}}}.
		\end{equation}
		Using Benilan-Crandall estimates \ref{BC2} and step 2 for $t=t_*$, we obtain the desired lower bound
		\begin{align*}
			t^{\frac{m_0}{m_0-1}}F(u(t,x)) &\geq t_*^{\frac{m_0}{m_0-1}} F(u(t_*,x))
			\geq t_*^{\frac{m_0}{m_0-1}} \underline{K} \frac{\Phi_1(x)^{\sigma_1}}{t_*^{\frac{m_1}{m_1-1}}}\\
			&= \underline{K} t_*^{\frac{m_0}{m_0-1}-\frac{m_1}{m_1-1}} \Phi_1(x)^{\sigma_1}
			\geq \underline{K} c_*^{\frac{m_0}{m_0-1}-\frac{m_1}{m_1-1}} \Phi_1(x)^{\sigma_1}.
		\end{align*}
		 Putting together inequalities \eqref{posIec8} and \eqref{posIec9}, we conclude the proof of the theorem.
	\end{proof}
	
	\begin{thm}\label{positivity2}
		Let $\A$ satisfy \eqref{A1}, \eqref{A2} and \eqref{L2}, and let F satisfy \eqref{N1}. Suppose $u$ is the \solution\,of \eqref{CDP} with initial datum $0\leq u_0\in \LL^1_{\ph(\Omega)}$. Let the waiting time  $t_*$ be as in \eqref{tstar.111}, then there exists a positive constant $k_{10}>0$ which depends only on $N$, $s$, $m_i$, $F$, $\gamma$, $\Omega$ and $\lambda_1$ such that $\forall t>0$ and a.e. $x\in \Omega$ the following inequality holds
		\begin{equation*}
			F(u(t,x))\geq k_{10} \left( 1 \land \frac{t}{t_*} \right)^{\frac{m_1^2}{m_1-1}} \frac{\ph(x)^{m_1}}{t^{\frac{m_i}{m_i-1}}}.
		\end{equation*}
		Where $i=1$ if $t\leq t_*$ and $i=0$ if $t\geq t_*$.
	\end{thm}
	\begin{proof}
		The proof follows from the same argument as Theorem \ref{positivity1} with some slight modifications in the calculations. In fact, we split the proof in the same steps.\\
		\textbf{Step 1:} For all $t\in [0,t_*]$ there exists a constant $k>0$ such that
		\begin{equation}\label{posIIec1}
			F(u(t,x))\geq \left( k \|u_0\|_{\LL^1_{\Phi_1}(\Omega)}^{m_1} t \right)^{m_1} \Phi_1(x)^{m_1}.
		\end{equation}
	The idea to complete the proof of inequality \eqref{posIIec1} is similar to that of step 1 in the previous proof. Therefore, we will remake the proof without dwelling too much on the details. As before, we fix $\delta>0$, let $\ud$ be the \solution\,of \eqref{AP1} and define the lower barrier $$\Psi(t,x)=F^{-1}((K_0 t \Phi_1(x))^{m_1}).$$
		Where $K_0= k \|u_0\|_{\LL^1_{\Phi_1}(\Omega)}^{m_1}$ and $k$ is a positive constant to be determined. Notice that we have the same properties as Remark \ref{obser1}.
		
		Now, we claim $F(\Psi)<F(\ud)$ in $[0,t_*]\times \overline{\Omega}$. Let's prove it by contradiction, assume the inequality is false; thus, there exists a point $(t_c,x_c)\in (0,t_*]\times \Omega$ such that $F(\Psi(t_c,x_c))=F(\ud(t_c,x_c))$ for the first time. The contradiction follows if we bound the next term from above and from below
		$$-\A [F(\ud(t_c,x_c)) - F(\Psi(t_c,x_c))]=\partial_t \ud(t_c,x_c) + \A F(\Psi(t_c,x_c)).$$
		
		Firstly, we find an upper bound for every term on right-hand side of the equation above.
		\begin{align}\label{posIIec2}
			\partial_t \ud(t_c,x_c)\leq \partial_t \Psi(t_c,x_c) &= \frac{m_1 K_0^{m_1}t_c^{m_1-1}\Phi_1(x_c)^{m_1}}{F'(F^{-1}(K_0^{m_1}t_c^{m_1}\Phi_1(x_c)^{m_1}))}\nonumber\\
			&\leq \frac{m_1 K_0^{m_1}t_c^{m_1-1}\Phi_1(x_c)^{m_1}}{m_0 K_0^{m_1}t_c^{m_1} \Phi_1(x_c)^{m_1}} F^{-1}(K_0^{m_1} t_c^{m_1} \Phi_1(x_c)^{m_1})\nonumber\\
			&\leq \frac{m_1}{m_0}\frac{K_0 t_c \Phi_1(x_c)}{t_c}\nonumber\\
			&\leq \frac{m_1}{m_0} K_0 \Phi_1(x_c).
		\end{align}
	Here, we have used the condition \eqref{N1} of the nonlinearity $F$. To find an upper bound for $\A F(\Psi)$ we will use Kato type inequality, $\A(f^m)\leq m f^{m-1}\A(f)$ for any non negative $f$ and $m>1$. You can see the proof in \cite{MB+AF+JLV_sharp-global-estimates}. Note that $F(\Psi(t,x))= (K_0 t \Phi_1(x))^{m_1}$, so we have
		\begin{align}\label{posIIec3}
			\A F(\Psi(t_c,x_c))&\leq m_1 (K_0 t_c \Phi_1(x_c))^{m_1-1} K_0 t_c \A \Phi_1(x_c)\nonumber\\
			&= m_1 K_0^{m_1} t_c^{m_1} \lambda_1 \Phi_1(x_c)^{m_1}\nonumber\\
			&\leq K_0 (m_1 \lambda_1 \overline{k_0} \|\Phi_1\|_{\LL^\infty(\Omega)}) \Phi_1(x_c).
		\end{align}
		
		Secondly, for the lower bound, we are going to use the definition of the operator
		$$P.V \int_{\RN} (F(\ud(t_c,y))-F(\Psi(t_c,y))) K(x_c,y)\dy=-\A [F(\ud(t_c,x_c)) - F(\Psi(t_c,x_c))].$$
		 Thus, using the left-hand side of the equation above  and the property \eqref{L2} of the kernels, we obtain
		\begin{align}\label{posIIec4}
			P.V \int_{\RN} (F(\ud(t_c,y))&-F(\Psi(t_c,y))) K(x_c,y)\dy\geq c_0 \Phi_1(x_c) \int_{\Omega} (F(\ud(t_c,y)) - F(\Psi(t_c,y)))\Phi_1(y)\dy\nonumber\\
			&= c_0 \Phi_1(x_c) \int_{\Omega} F(\ud(t_c,y))\Phi_1(y)\dy - c_0 \Phi_1(x_c) \int_{\Omega} (K_0 t_c \Phi_1(y))^{m_1} \Phi_1(y)\dy\nonumber\\
			&\geq c_0 \Phi_1(x_c)\int_{\Omega} F(\ud(t_c,y))\Phi_1(y)\dy - c_0\Phi_1(x_c) K_0 \overline{k_0} ||\Phi_1||_{\LL^\infty(\Omega)}^{m_1+1}|\Omega|.
		\end{align}
		
		Moreover, by Lemma \ref{propiedadesudelta3} we also know $\underline{c}\|u_0\|_{\LL^1_{\Phi_1}(\Omega)}^{m_1}\leq \int_{\Omega} F(\ud(t,x))\Phi_1(x)\dx$. Putting all these inequalities together, \eqref{posIIec2}, \eqref{posIIec3}, \eqref{posIIec4}, and this last one, we have
		\begin{align*}
			c_0\underline{C} \|u_0\|_{\LL^1_{\Phi_1}(\Omega)}^{m_1}&\leq c_0 \int_{\Omega} F(\ud(t_c,x))\Phi_1(x)\dx\\
			&\leq K_0 \left( \frac{m_1}{m_0} + m_1\lambda_1 \overline{k_0} \|\Phi_1\|_{\LL^\infty(\Omega)} + c_0 \overline{k_0} \|\Phi_1\|_{\LL^1(\Omega)}^{m_1+1} |\Omega| \right)\\
			&\leq k \|u_0\|_{\LL^1_{\Phi_1}(\Omega)}^{m_1} \left( \frac{m_1}{m_0} + m_1\lambda_1 \overline{k_0} \|\Phi_1\|_{\LL^\infty(\Omega)} + c_0 \overline{k_0} \|\Phi_1\|_{\LL^1(\Omega)}^{m_1+1} |\Omega| \right).
		\end{align*}
	Then, all constants are fixed except $k$, and the initial data is not identically zero. Thus, we can choose $k$ small enough to lead us to the contradiction. Finally, we take the limit as $\delta \to 0^+$ and conclude \eqref{posIIec1}.
		
		\textbf{Step 2:} For all $t\in (0,t_*]$ the following inequality holds
		\begin{equation}\label{posIIec5}
			F(u(t,x))\geq \underline{K} \left( \frac{t}{t_*} \right)^{\frac{m_1^2}{m_1-1}} \frac{\Phi_1(x)^{m_1}}{t^{\frac{m_1}{m_1-1}}}.
		\end{equation}
		The proof follows by the same argument as in Step 2 of the proof of Theorem \ref{positivity1}.
		
		\textbf{Step 3:} For all $t> t_*$ we obtain the bound
		\begin{equation}\label{posIIec6}
			F(u(t,x))\geq \underline{K} \frac{\Phi_1(x)^{m_1}}{t^{\frac{m_0}{m_0-1}}}.
		\end{equation}
		The proof follows by the same argument as in Step 3 of the proof of Theorem \ref{positivity1}.
		
		We can complete the proof by joining inequalities \eqref{posIIec5} and \eqref{posIIec6}.
	\end{proof}
	
	%%%%%%%%%%%%%%%%%%%%%%%%%%%%%%%%%%%%%%%%%%%%%%%%%%%%%%%%%%%%%%%%%%%%
	%%%%%%%%%%%%%%%%%%%%%%%%%%%%%%%%%%%%%%%%%%%%%%%%%%%%%%%%%%%%%%%%%%%%
	
	\subsection{Positivity for large initial data}\label{sec5.4}
	
	In this subsection, we will proceed similarly to the previous one. However, since the norm of the initial datum is large, $\|u_0\|_{\LL^1_{\Phi_1}(\Omega)}>1$, we have
\begin{equation}\label{tstar.222}
t_*=\frac{c_*}{\|u_0\|_{\LL^1_{\Phi_1}(\Omega)}^{m_0-1}}.
\end{equation}
The proof follows by similar ideas, but there are several relevant technical modifications that need to be addressed.
	
	\begin{lem}\label{propiedadesudelta4}
		Let $\A$ satisfy \eqref{A1}, \eqref{A2} and \eqref{K2}, let $F$ satisfy \eqref{N1}. If $u$ is a \solution\,of \eqref{CDP} with initial data $0\leq u_0 \in \LL^1_\ph(\Omega)$, let $\ud$ be the approximate solution of \eqref{AP1} with the same initial data and let $t_*$ be as in \eqref{tstar.222}. Then, there exists a constant $\underline{c}>0$ such that for all $t\in [0,t_*]$ we have
		\begin{align}\label{cotaudelta2}
			\underline{c} \|u_0\|_{\LL^1_\ph(\Omega)}^{m_0}
            \leq \int_{\Omega} F(u(t,x)) \ph(x)\dx
			\leq \int_{\Omega} F(\ud (t,x)) \ph(x)\dx.
		\end{align}
	\end{lem}
	
	\begin{proof}
		We split the proof in two steps.\\
		\textbf{Step 1:} We assert that there exists a constant $C$ such that for $t_0=C\left( (2k_7)^{\frac{1}{2s\theta_{0,\gamma}}} ||u_0||_{\LL^1_{\Phi_1}(\Omega)}^{m_0-1} \right)^{-1}$, the conditions of Corollary \ref{corWE} are satisfied.
		
		Let us define $M:= (2k_7)^{-\frac{(n+\gamma)}{2s}}$. A straightforward calculation shows that if $\|u_0\|_{\LL^1_\ph(\Omega)}\geq M$, then $\|u_0\|_{\LL^1_{\Phi_1}(\Omega)}^{\frac{2s}{N+\gamma}}\geq\left( (2k_7)^{\frac{1}{2s\theta_{0,\gamma}}} \|u_0\|_{\LL^1_{\Phi_1}(\Omega)}^{m_0-1} \right)^{-1}$. Therefore, we can apply Corollary \ref{corWE} since, if we choose $t_0=\left( 2(2k_7)^{\frac{1}{2s\theta_{0,\gamma}}} \|u_0\|_{\ph(\Omega)}^{m_0-1} \right)^{-1}$, at the same time we get $t_0\leq \|u_0\|_{\LL^1_{\Phi_1}(\Omega)}^{\frac{2s}{N+\gamma}}$ and $t_0\leq \left( (2k_7)^{\frac{1}{2s\theta_{0,\gamma}}} \|u_0\|_{\ph(\Omega)}^{m_0-1} \right)^{-1}$.
		
		In the case when the norm is bounded from above and from below, $1\leq \|u_0\|_{\LL^1_{\Phi_1}(\Omega)}\leq M$, we consider the quantity $\min_{i=0,1}\left( \left( (2k_7)^{\frac{1}{2s\theta_{i,\gamma}}} \|u_0\|_{\LL^1_{\Phi_1}(\Omega)}^{m_i-1} \right)^{-1} \right)$. If the minimum occurs at $i=0$ we are in the same scenario as described above and,  if it occurs at $i=1$, we can repeat the computation of Remark \ref{rem2} to obtain that the constant for which the hyphoteses of Corollary \ref{corWE} are always satisfied is
		$$C=\frac{1}{2}\left( 1 \land (2k_7)^{(m_1-m_0)\frac{N+\gamma}{2s}-\frac{1}{2s\theta_{1,\gamma}}+\frac{1}{2s\theta_{0,\gamma}}} \right).$$
		
		\textbf{Step 2:} Proof of the bound \eqref{cotaudelta2}.
		
		Select $t_0$ as   in Step 1, then we can repeat the same proof as in Lemma \ref{propiedadesudelta3}, just using the properties of Lemma \ref{Nonlin} for $F$ in the case where $\|u_0\|_{\LL^1_\ph(\Omega)}\geq 1$ together with the fact that $t_0= A_2 t_*$.
	\end{proof}
	
	We prove now the quantitative lower bounds, which require hypotheses \eqref{L1} and \eqref{L2} for the operator $\A$. The main idea is similar to the previous subsection, but the there are dramatic changes in the technical details, in particular the barriers are different.
	\begin{thm}\label{positivity12}
		Let $\A$ satisfy \eqref{A1}, \eqref{A2} and \eqref{L1}, and let F satisfy \eqref{N1}. Suppose $u$ is the \solution\,of \eqref{CDP} with initial datum $0\leq u_0\in \LL^1_\ph(\Omega)$. Given the waiting time $t_*$ as in \eqref{tstar.222}, we establish the following result:
		
		Assume either $\sigma_1=1$ or $\sigma_1<1$, $K(x,y)<c_1|x-y|^{-(N+2s)}$ and $\ph\in C^\gamma(\overline{\Omega})$, then there exists a constant $k_{11}>0$ which depends only on $N$, $s$, $m_i$, $\gamma$, $F$, $\Omega$ and $\lambda_1$ such that for all $t>0$ and a.e. $x\in \Omega$ the following holds
		\begin{equation}\label{boundpositivity12}
			F(u(t,x))\geq k_{11} \left\{\begin{array}{ll} \left( \frac{t}{t_*} \right)^{m_1} \Phi_1^{\sigma_1}, & \mbox{ if }t\leq t_*.\\ \left( \frac{t_*}{t} \right)^{\frac{m_0}{m_0-1}} \Phi_1(x)^{\sigma_1}, & \mbox{ if }t\geq t_*. \end{array}\right.
		\end{equation}
	\end{thm}
	\begin{proof}
		We split the proof in three steps. Firstly, we estimate a lower bound for $F(u(t,x))$ which is the key of the proof, after that, we obtain inequality \eqref{boundpositivity12} for small time $t\leq t_*$ and finally for large time $t\geq t_*$.
		
		\textbf{Step 1: Claim.}  There exists a constant $k>0$ small enought such that  for all $ t\in [0,t_*]$ we have
		\begin{equation}\label{posIec12}
			F(u(t,x))\geq \left( k \|u_0\|_{\LL^1_{\Phi_1}(\Omega)}^{m_0-1} t \right)^{m_1} \Phi_1(x)^{\sigma_1}.
		\end{equation}
		Proof of the claim. Step 1.1.\textit{ Aproximate Solutions.}\\
		We fix $\delta>0$ for the duration of the proof of the claim and let $u_\delta$ be the approximate solution of \eqref{AP1} with initial data $u_0$. Define $K_0=k\|u_0\|_{\LL^1_{\Phi_1}(\Omega)}^{m_0-1}$, where $k$ is a positive constant to be determined. Now let the lower barrier be given by $$\Psi(t,x)=F^{-1}(K_0^{m_1} t^{m_1} \Phi_1(x)^{\sigma_1}).$$
		Note that here, in contrast to the theorem \ref{positivity1} and the results on a priori estimates of \cite{MB+AF+JLV_sharp-global-estimates}, the power of the norm of the initial data in the constant $K_0$ is $m_0-1$, not $m_0$. This is because the norm can be arbitrarily large, always bigger than 1. Some fundamental observations are explained in the following remark.

		\begin{rem}\label{obser2}\rm
			\renewcommand{\labelenumi}{\rm (\alph{enumi})}
			\begin{enumerate}
				\item \label{puntoa2} Here, we have $t_*=c_*\|u_0\|_{\LL^1_{\ph}(\Omega)}^{1-m_0}\leq c_*$ because  $\|u_0\|_{\LL^1_{\Phi_1}(\Omega)}>1$.
				\item \label{puntob2} In order to simplify the calculations, let $\Phi_1$ be the first eigenfunction of $\A$ normalized i.e., $\|\Phi_1\|_{\LL^\infty(\Omega)}=1$.
				\item \label{puntoc2} By the construction of $K_0$ and the definition of $t_*$, for any $t\in [0,t_*]$ we have $$K_0^{m_1-1} t^{m_1}\leq K_0^{m_1-1} t_*^{m_1}\leq (K_0 t_*)^{m_1-1} c_*\leq \overline{k_0} .$$
				\item \label{puntod2} For any $t\in [0,t_*]$ we have $K_0 t\leq K_0 t_*=kc_*:=r_0$. Therefore, by adjusting the value of $k$ if necessary, we obtain from Lemma \ref{Nonlin} that $$F^{-1}\left( \frac{F(kc_*)}{(kc_*)^{m_1}} (K_0t)^{m_1} \right)\leq K_0t.$$
			\end{enumerate}
		\end{rem}
		
		Step 1.2. \textit{$\Psi$ is a lower barrier for $\ud$: }We prove $F(\Psi(t,x))< F(u_{\delta}(t,x))$ for any $(t,x)\in [0,t_*] \times \overline{\Omega}$ by contradiction.
		 Assume that the inequality is false. By the regularity of $F$, $\ud$ and $\Psi$, there exists a first touching point $(t_c,x_c)\in [0,t_*] \times \overline{\Omega}$ where $F(\Psi(t_c,x_c))=F(\ud(t_c,x_c))$ for the first time. Note that $(t_c,x_c)\in (0,t_*)\times \Omega$ because is $t_c=0$ or $x_c\in \partial \Omega$ we have $$F(\Psi(t_c,x_c))=0<F(\delta)\leq F(\ud(t_c,x_c)).$$ Now, let's bound $-\A[F(\ud(t_c,x_c))-F(\Psi(t_c,x_c))]$ from above and below. First, we find the upper bound similarly to Theorem \ref{positivity1}:
		$$-\A [F(\ud(t_c,x_c))-F(\Psi(t_c,x_c))]= \partial_t \ud(t_c,x_c) + \A F(\Psi(t_c,x_c)),$$
		and we bound every single term of above. On the one hand, as $(t_c,x_c)$ is the first touching point, we have
		\begin{align}\label{posIec22}
			\partial_t \ud(t_c,x_c)\leq \partial_t \Psi(t_c,x_c) &= \frac{m_1 K_0^{m_1} t_c^{m_1-1} \Phi_1(x)^{\sigma_1}}{F'(F^{-1}(K_0^{m_1}t_c^{m_1}\Phi_1(x)^{\sigma_1}))}\nonumber\\
			&\leq \frac{m_1 K_0^{m_1} t_c^{m_1-1} \Phi_1(x)^{\sigma_1}}{m_0 K_0^{m_1} t_c^{m_1} \Phi_1(x)^{\sigma_1}} F^{-1}(K_0^{m_1} t_c^{m_1} \Phi_1(x)^{\sigma_1})\nonumber\\
			&\leq \frac{m_1}{m_0} K_0.
		\end{align}
		Where we have used $m_0 F(r)/r\leq F'(r)\leq m_1 F(r)/r$ as a consequence of \eqref{N1}, and Remark \ref{obser2}\,$\rm{(d)}$. On the other hand, if $\sigma_1=1$ we have
		\begin{equation}\label{posIec32}
			\A F(\Psi(t_c,x_c))=K_0^{m_1} t_c^{m_1}\lambda_1 \Phi_1(x_c)\leq K_0 \overline{k_0}\lambda_1.
		\end{equation}
		Where we have used Remark \ref{obser2}\,$\rm{(c)}$. If $\sigma_1<1$, then using  that $\Phi_1\in C^{\gamma}(\overline{\Omega})$ and $g(x)=x^{\sigma_1}$ is concave we get
		$$|\Phi_1(x)^{\sigma_1} - \Phi_1(y)^{\sigma_1}|\leq |\Phi_1(x)-\Phi_1(y)|^{\sigma_1}\leq |x-y|^{\gamma \sigma_1}.$$
		Thus, as we know $\gamma\sigma_1>2s$ and $K(x,y)\leq \frac{c_1}{|x-y|^{N+2s}}$, we obtain
		\begin{align}\label{posIec42}
			\A F(\Psi(t_c,x_c)) = K_0^{m_1} t_c^{m_1} \A &\Phi_1^{\sigma_1}(x_c) = K_0^{m_1} t_c^{m_1} P.V \int_{\RN} |\Phi_1(x_c)^{\sigma_1} - \Phi_1(y)^{\sigma_1}| K(x_c,y) \dy\nonumber \\
			&\leq K_0^{m_1} t_c^{m_1} \left(\int_{\Omega} c |x_c-y|^{\gamma \sigma_1} \frac{c_1}{|x_c-y|^{N+2s}}\dy + \int_{\RN \setminus \Omega} \frac{2 c_1}{|x_c-y|^{N+2s}}\dy\right)\nonumber\\
			&\leq \overline{C}K_0 \overline{k_0}.
		\end{align}
		Therefore, by combining inequalities \eqref{posIec22}, \eqref{posIec32} and \eqref{posIec42} we obtain the following upper bound:
		\begin{equation}\label{posIec52}
			-\A[F(\ud(t_c,x_c))-F(\Psi(t_c,x_c))]\leq K_0 \left( \frac{m_1}{m_0} + \overline{k_0} (\overline{C}+\lambda_1) \right).
		\end{equation}
		
		Secondly, we show how to obtain the lower bound, employing a method similar to that used in Theorem \ref{positivity1}. Here, we utilize the nonlocal nature of the operator, as specified by condition \eqref{L1}.
		\begin{align}\label{posIec62}
			-\A [F(\ud(t_c,x_c)) &- F(\Psi(t_c,x_c))]\nonumber \\
            &= - P.V \int_{\RN} [F(\ud(t_c,x_c) - F(\Psi(t_c,x_c))) - (F(\ud(t_c,y)) -F(\Psi(t_c,y)))] K(x_c,y) \dy\nonumber \\
			&= \int_{\RN} F(\ud(t_c,y)) K(x_c,y) \dy - \int_{\RN} F(\Psi(t_c,y)) K(x_c,y)\dy\nonumber\\
			&\geq k_{\Omega} \int_{\RN} F(\ud(t_c,y)) \dy - k_{\Omega} \int_{\RN} (K_0 t_c)^{m_1}\Phi_1(y)^{\sigma_1}\dy\nonumber\\
			&\geq k_{\Omega} \int_{\RN} F(\ud(t_c,y)) \dy - k_{\Omega} K_0 \overline{k_0} |\Omega|.
		\end{align}
		Therefore, thank to \eqref{posIec52} and \eqref{posIec62} we have
		\begin{equation*}
			k_{\Omega} \int_{\Omega} F(\ud(t_c,x)) \dx \leq K_0\left( k_{\Omega} \overline{k_0} |\Omega| + \frac{m_1}{m_0} + \overline{k_0} (\overline{C}+\lambda_1) \right).
		\end{equation*}
		Remembering Lemma \ref{propiedadesudelta4} we have the inequality which allow us to obtain the contradiction,
		\begin{equation*}\label{posIec72}
			k_{\Omega} \underline{c} \|u_0\|_{\LL^1_{\Phi_1}(\Omega)}^{m_0-1} \leq k_{\Omega} \underline{c} \|u_0\|_{\LL^1_{\Phi_1}(\Omega)}^{m_0}\leq k \|u_0\|_{\LL^1_{\Phi_1}(\Omega)}^{m_0-1} \left( k_{\Omega} \overline{k_0} |\Omega| + \frac{m_1}{m_0} + \overline{k_0} (\overline{C}+\lambda_1) \right).
		\end{equation*}
		Where the norm of the initial data is not zero and all constants are fixed except $k$. Thus, we can choose $k$ small enough to create a contradiction and conclude $F(\Psi)<F(\ud)$ in $[0,t_*]\times \overline{\Omega}$.
		
		Step 1.3.\textit{Taking the limit as $\delta\to 0^+$.}\\
		By the continuity of F, the pointwise limit of $\ud(t,x)$ as $\delta\to 0$ and the lower barrier from step 1.2, we conclude that for any $t\in [0,t_*]$ and a.e. $x\in \Omega$
		$$F(u(t,x))\geq F(\Psi(t,x))= (k||u_0||_{\LL^1_{\Phi_1}(\Omega)}^{m_0-1} t)^{m_1} \Phi_1(x)^{\sigma_1}.$$
		
		\textbf{Step 2:} The following inequality holds for all $t\in (0,t_*]$
		\begin{equation}\label{posIec82}
			F(u(t,x))\geq \underline{K} \left( \frac{t}{t_*} \right)^{m_1} \Phi_1(x)^{\sigma_1}.
		\end{equation}
		We just have to rewrite \eqref{posIec12} as follows
		\begin{align*}
			F(u(t,x))&\geq \left(k \|u_0\|_{\LL^1_{\Phi_1}(\Omega)}^{m_0-1} t\right)^{m_1} \Phi_1(x)^{\sigma_1}
			= \left( k \frac{c_* t}{t_*} \right)^{m_1} \Phi_1(x)^{\sigma_1}
			\geq \underline{K} \left( \frac{t}{t_*} \right)^{m_1} \Phi_1(x)^{\sigma_1}.
		\end{align*}
		
		\textbf{Step 3:} For large times, i.e. $t> t_*$ we have
		\begin{equation}\label{posIec92}
			F(u(t,x))\geq \underline{K} t_*^{\frac{m_0}{m_0-1}}\frac{\Phi_1(x)^{\sigma_1}}{t^{\frac{m_0}{m_0-1}}}.
		\end{equation}
		Using Benilan-Crandall estimates, more precisely inequality \ref{BC2}, and \eqref{posIec82} for $t=t_*$, we obtain
		\begin{align*}
			F(u(t,x)) &\geq \frac{t_*^{\frac{m_0}{m_0-1}}}{t^{\frac{m_0}{m_0-1}}} F(u(t_*,x))
			\geq t_*^{\frac{m_0}{m_0-1}} \underline{K} \frac{\Phi_1(x)^{\sigma_1}}{t^{\frac{m_0}{m_0-1}}}.
		\end{align*}
		Putting together inequalities \eqref{posIec82} and \eqref{posIec92}, we conclude the proof of the theorem.
	\end{proof}
	
	\begin{thm}\label{positivity22}
		Let $\A$ satisfy \eqref{A1}, \eqref{A2} and \eqref{L2}, and let F satisfy \eqref{N1}. Suppose $u$ is the \solution\,of \eqref{CDP} with initial datum $0\leq u_0\in \LL^1_{\ph(\Omega)}$ and let the waiting time $t_*$ be as in \eqref{tstar.222}. Then, there exists a constant $k_{12}>0$ which depends only on $N$, $s$, $m_i$, $F$, $\gamma$, $\Omega$ and $\lambda_1$ such that for all $t>0$ and a.e. $x\in \Omega$ we have
		\begin{equation}\label{boundpositivity22}
			F(u(t,x))\geq k_{12} \left\{\begin{array}{ll} \left( \frac{t}{t_*} \right)^{m_1} \Phi_1^{m_1}, & \mbox{ if }t\leq t_*.\\ \left( \frac{t_*}{t} \right)^{\frac{m_0}{m_0-1}} \Phi_1(x)^{m_1}, & \mbox{ if }t\geq t_*. \end{array}\right.
		\end{equation}
	\end{thm}
	\begin{proof}
		The proof follows a similar argument to that of Theorem \ref{positivity2} with minor modifications.
		
		\textbf{Step 1:} For all $t\in [0,t_*]$ there exists a positive constant $k>0$ such that
		\begin{equation}\label{posIIec12}
			F(u(t,x))\geq \left( k \|u_0\|_{\LL^1_{\Phi_1}(\Omega)}^{m_0-1} t \right)^{m_1} \Phi_1(x)^{m_1}.
		\end{equation}
		To proceed similarly to our previous approach, we fix $\delta>0$, and denote $\ud$ the approximate solution of \eqref{AP1}. We define the lower barrier $$\Psi(t,x)=F^{-1}((K_0 t \Phi_1(x))^{m_1}),$$
		where $K_0= k \|u_0\|_{\LL^1_{\Phi_1}(\Omega)}^{m_0-1}$. Here, $k$ is a positive constant that will be determined, and it possesses similar properties as in the previous proof, (cf. Remark \ref{obser2}).
		
		 We now observe that $F(\Psi)<F(\ud)$ on $[0,t_*]\times \overline{\Omega}$ and we proceed to establish it by contradiction. Suppose the inequality does not hold; then there exists a point $(t_c,x_c)\in (0,t_*]\times \Omega$ such that $F(\Psi(t_c,x_c))=F(\ud(t_c,x_c))$ for the first time. The contradiction follows if we properly bound the following quantity from both above and below
\[\begin{split}
P.V \int_{\RN} (F(\ud(t_c,y))-F(\Psi(t_c,y))) K(x_c,y)\dy&=-\A [F(\ud(t_c,x_c)) - F(\Psi(t_c,x_c))]\\
&=\partial_t \ud(t_c,x_c) + \A F(\Psi(t_c,x_c)).
\end{split}\]
		We first establish an upper bound for each term on the right-hand side of the equation above.
		\begin{align}\label{posIIec22}
			\partial_t \ud(t_c,x_c)\leq \partial_t \Psi(t_c,x_c) &= \frac{m_1 K_0^{m_1}t_c^{m_1-1}\Phi_1(x_c)^{m_1}}{F'(F^{-1}(K_0^{m_1}t_c^{m_1}\Phi_1(x_c)^{m_1}))}\nonumber\\
			&\leq \frac{m_1 K_0^{m_1}t_c^{m_1-1}\Phi_1(x_c)^{m_1}}{m_0 K_0^{m_1}t_c^{m_1} \Phi_1(x_c)^{m_1}} F^{-1}(K_0^{m_1} t_c^{m_1} \Phi_1(x_c)^{m_1})\nonumber\\
			&\leq \frac{m_1}{m_0}\frac{K_0 t_c \Phi_1(x_c)}{t_c}\nonumber\\
			&\leq \frac{m_1}{m_0} K_0 \Phi_1(x_c).
		\end{align}
		Using Kato type inequality $\A(f^m)\leq m f^{m-1}\A(f)$, or any non-negative function $f$ and $m>1$, and noting that $F(\Psi(t,x))= (K_0 t \Phi_1(x))^{m_1}$, we can write:
		\begin{align}\label{posIIec32}
			\A F(\Psi(t_c,x_c))&\leq m_1 (K_0 t_c \Phi_1(x_c))^{m_1} K_0 t_c \A \Phi_1(x_c)\nonumber\\
			&= m_1 K_0^{m_1} t_c^{m_1} \lambda_1 \Phi_1(x_c)^{m_1}
			\leq K_0 (m_1 \lambda_1 \overline{k_0} \|\Phi_1\|_{\LL^\infty(\Omega)}) \Phi_1(x_c).
		\end{align}
		
		Since the operator is non-local (cf. \eqref{L2}) and considering the identity $F(\Psi(t_c,x_c))=F(\ud(t_c,x_c))$ we proceed to derive the lower bound of the left-hand side of the previous equation.
		\begin{align}\label{posIIec42}
			P.V \int_{\RN} (F(\ud(t_c,y))&-F(\Psi(t_c,y))) K(x_c,y)\dy\geq c_0 \Phi_1(x_c) \int_{\Omega} (F(\ud(t_c,y)) - F(\Psi(t_c,y)))\Phi_1(y)\dy\nonumber\\
			&= c_0 \Phi_1(x_c) \int_{\Omega} F(\ud(t_c,y))\Phi_1(y)\dy - c_0 \Phi_1(x_c) \int_{\Omega} (K_0 t_c \Phi_1(y))^{m_1} \Phi_1(y)\dy\nonumber\\
			&\geq c_0 \Phi_1(x_c)\int_{\Omega} F(\ud(t_c,y))\Phi_1(y)\dy - c_0\Phi_1(x_c) K_0 \overline{k_0} \|\Phi_1\|_{\LL^\infty(\Omega)}^{m_1+1}|\Omega|.
		\end{align}
		
		Moreover, thanks to Lemma \ref{propiedadesudelta4} we know $\underline{c}\|u_0\|_{\LL^1_{\Phi_1}(\Omega)}^{m_0}\leq \int_{\Omega} F(\ud(t,x))\Phi_1(x)\dx$. Thus, combining this inequality together with inequalities \eqref{posIIec22}, \eqref{posIIec32} and \eqref{posIIec42}, we obtain
		\begin{align*}
			c_0\underline{c} \|u_0\|_{\LL^1_{\Phi_1}(\Omega)}^{m_0-1}&\leq
			c_0\underline{c} \|u_0\|_{\LL^1_{\Phi_1}(\Omega)}^{m_0}
			\leq c_0 \int_{\Omega} F(\ud(t_c,x))\Phi_1(x)\dx\\
			&\leq K_0 \left( \frac{m_1}{m_0} + m_1\lambda_1 \overline{k_0} \|\Phi_1\|_{\LL^\infty(\Omega)} + c_0 \overline{k_0} \|\Phi_1\|_{\LL^1(\Omega)}^{m_1+1} |\Omega| \right)\\
			&\leq k \|u_0\|_{\LL^1_{\Phi_1}(\Omega)}^{m_0-1} \left( \frac{m_1}{m_0} + m_1\lambda_1 \overline{k_0} \|\Phi_1\|_{\LL^\infty(\Omega)} + c_0 \overline{k_0} \|\Phi_1\|_{\LL^1(\Omega)}^{m_1+1} |\Omega| \right).
		\end{align*}
		Then, all constants are fixed except $k$, and the initial data is non-zero. Thus, we can choose $k$ small enough to reach a contradiction.
		
		Finally, we take the limit as $\delta \to 0^+$ which lead us to \eqref{posIIec12}, namely
		$$F(u(t,x))\geq F(\Psi(t,x))= (k ||u_0||_{\LL^1_{\Phi_1}(\Omega)}^{m_0-1}t \Phi_1(x))^{m_1}.$$
		
		\textbf{Step 2:}  The following inequality holds for all $t\in (0,t_*]$
		\begin{equation}\label{posIIec52}
			F(u(t,x))\geq \underline{K} \left( \frac{t}{t_*} \right)^{m_1} \Phi_1(x)^{m_1}.
		\end{equation}
        The proof follows by the same argument as in Step 2 of the proof of Theorem \ref{positivity12}.
		
		\textbf{Step 3:} For large times, i.e. $t> t_*$ we have
		\begin{equation}\label{posIIec62}
			F(u(t,x))\geq \underline{K} t_*^{\frac{m_0}{m_0-1}} \frac{\Phi_1(x)^{m_1}}{t^{\frac{m_0}{m_0-1}}}.
		\end{equation}
		The proof follows by the same argument as in Step 3 of the proof of Theorem \ref{positivity12}.
		
		We complete the proof by joining estimates \eqref{posIIec52} and \eqref{posIIec62}.
	\end{proof}

%%%%%%%%%%%%%%%%%%%%%%%%%%%%%%%%%%%%%%%%%%%%%%%%%%%%%%%%
%%%%%%%%%%%%%%%%%%%%%%%%%%%%%%%%%%%%%%%%%%%%%%%%%%%%%%%%

\subsection{The anomalous lower bounds with small data}

	Since $\sigma_1\leq 1 < m_1$ we have $\Phi_1(x)^{m_1}< \Phi_1(x)^{\sigma_1}$ for any $x$ close enough to the boundary of the domain. As we show in Theorem \ref{GHPIII}, the lower bound $F(u(t))\gtrsim \Phi_1^{m_1}$ is always valid, but the power does not match the one of the upper bound. Now, we want to discuss the possibility of improving this lower bound.
	
	We first show that we cannot hope to prove that $F(u(t))$ is larger than $\Phi_1^{\sigma_1}$ in general, hence we cannot match the powers of the lower and upper bounds to get $F(u(t))\asymp \Phi_1^{\sigma_1}$ when $\sigma_1<1$.
	
	\begin{thm}\label{smalldata3}
		Let $\A$ be an operator that satisfies \eqref{A1}, \eqref{A2} and \eqref{K2}, and $u$ be a nonnegative weak dual solution to the problem \eqref{CDP} corresponding with the initial datum $0\leq u_0\in \LL^1_{\Phi_1}(\Omega)$. Assume that the initial datum is small, $u_0(x)\leq C_0 \Phi_1(x)$, for some positive constant $C_0$. Then, there exists a constant $k_{13}>0$ depending only on $N$, $s$, $\gamma$, $F$ and $\Omega$ such that
		
		\begin{equation*}
			F(u(t,x))\leq C_0 k_{13} \frac{\Phi_1(x)}{t} \hspace{6mm}\mbox{for all }t>0 \mbox{ and a.e. }x\in \Omega.
		\end{equation*}
	In particular, if $\sigma_1<1$, then we have
	\begin{equation*}
		\lim\limits_{x\to \partial \Omega}\frac{F(u(t,x))}{\Phi_1(x)^{\sigma_1}}=0 \hspace{6mm}\mbox{for any }t>0.
	\end{equation*}
	\end{thm}
	\begin{proof}
		Starting from the Fundamental Upper Bound \eqref{FUB} and since the function $t \to \int_{\Omega} u(t,y) \K(x,y) \dy$ is decreasing, remember Proposition \ref{pointwiseestimates}. We obtain the following:
		\begin{align*}
			F(u(t,x_0))&\leq \frac{2^{\frac{m_0}{m_0-1}}}{t} \int_{\Omega} u(t,x)\K(x,x_0) \dx \leq \frac{2^{\frac{m_0}{m_0-1}}}{t} \int_{\Omega} u_0(x) \K (x,x_0) \dx\\
			&\leq \frac{2^{\frac{m_0}{m_0-1}}}{t} C_0 \int_{\Omega} \Phi_1(x) \K (x,x_0) \dx = \frac{2^{\frac{m_0}{m_0-1}}}{t} C_0 \AI \Phi_1(x_0) = \frac{2^{\frac{m_0}{m_0-1}}}{t} \frac{C_0}{\lambda_1} \Phi_1(x_0).
		\end{align*}
	Where in the last steps we have used the upper bound for the initial data $u_0\leq C_0 \Phi_1$ and the fact $\A \Phi_1=\lambda_1 \Phi_1$.
	\end{proof}
A careful inspection of the above proof reveals that when $u_0(x)\leq C_0 \Phi_1(x)$ for a.e. $x\in\Omega$, we get
	\begin{equation}\label{ecabove}
		\int_{\Omega} u(t,x) \K(x,x_0) \dx \leq \frac{C_0}{\lambda_1} \Phi_1(x_0),\qquad\mbox{for a.e. }x_0\in \Omega\,.
	\end{equation}
	
	A natural question that appears after the above result is: \textit{Is it possible to obtain a sharp general lower bound in terms of $\ph$\,?} The answer is negative, as we shall see. Indeed, under assumption \eqref{K4}, the bound $F(u(t))\gtrsim \Phi_1$ is false for $\sigma_1<1$ and for small initial data, as the following Theorem shows.
	
	\begin{thm}\label{smalldata4}
		Let $\A$ be an operator that satisfies \eqref{A1}, \eqref{A2} and \eqref{K4}, and let $u$ be a nonnegative weak dual solution to the problem \eqref{CDP} corresponding to the initial datum $0\leq u_0 \in \LL^1_{\Phi_1}(\Omega)$. Assume that $u_0\leq C_0 \Phi_1$ for some positive constant $C_0$ and
		$$u(T,x)\geq k_{14}\Phi_1(x)^\alpha \hspace{6mm}\mbox{for a.e. }x\in \Omega,$$
		for some constants $k_{14}$, $T$, $\alpha>0$.
		Then, $\alpha\geq 1-\frac{2s}{\gamma}$. Moreover, if $\sigma_1<1$ we have $\alpha>\frac{1}{m_1}$.
	\end{thm}

\begin{proof}
	Given $x_0\in \Omega$, let $R_0= \dist(x_0, \partial \Omega)$ (recall $R_0\asymp \Phi_1^{\frac{1}{\gamma}}$). Since $\K(x,x_0)\gtrsim |x-x_0|^{-N+2s}$ inside the ball $B_{R_0/2}(x_0)$, by \eqref{K4}, we can use our assumption on $u$ to get
	$$\int_{\Omega} u(T,x) \K(x,x_0)\dx \gtrsim \int_{B_{R_0/2}(x_0)} \frac{\Phi_1(x)^\alpha}{|x-x_0|^{N-2s}}\dx \gtrsim \Phi_1(x_0)^\alpha R_0^{2s}\gtrsim \Phi_1(x_0)^{\alpha + \frac{2s}{\gamma}}.$$
	Combining the above inequality with \eqref{ecabove} gives
	$$\Phi_1(x_0)^{\alpha+\frac{2s}{\gamma}}\lesssim \Phi_1(x_0)\hspace{6mm}\mbox{for all }x_0\in \Omega.$$
	Which directly implies $\alpha \geq 1-\frac{2s}{\gamma}$.
	
	To conclude the proof, we only have to note that $1-\frac{2s}{\gamma}>\frac{1}{m_1}$ if and only if $\sigma_1<1$.
\end{proof}
	
	%%%%%%%%%%%%%%%%%%%%%%%%%%%%%%%%%%%%%%%%%%%%%%%%%%%%%
	%%%%%%%%%%%%%%%%%%%%%%%%%%%%%%%%%%%%%%%%%%%%%%%%%%%%%
	%%%%%%%%%%%%%%%%%%%%%%%%%%%%%%%%%%%%%%%%%%%%%%%%%%%%%
	%%%%%%%%%%%%%%%%%%%%%%%%%%%%%%%%%%%%%%%%%%%%%%%%%%%%%
	
	\section{Regularity results}\label{sec6}
	
	To obtain the regularity results, we require the Global Harnack Principle to be valid with matching power %in the first eigenfunction $\Phi_1$
 (Theorem \ref{GHPI} and Theorem \ref{GHPII}). For higher regularity results, extra assumptions on the kernels are necessary. Prior to addressing interior regularity, we must establish upper and lower bounds on $u$ of the type $0<\delta\leq u \leq M <\infty$. These bounds can be derived from Harnack inequalities, as we shall see below.

 It is also important to mention that in this work, we do not achieve  regularity up to the boundary: despite of the fact that we have matching power in the Global Harnack Principle, the lack of homogeneity of $F$ causes problems with the rescaling, an essential ingredient of the proof; see \cite[Theorem 8.3]{MB+AF+JLV_sharp-global-estimates} for a proof of sharp boundary regularity in the case $F(u)=u^m$ with $m>1$. Boundary regularity remains an intriguing open problem for nonhomogeneous nonlinearities.  Also, we would like to mention that the optimal boundary regularity for the classical case (i.e. $u_t=\Delta u^m$) has been only recently obtained by Jin, Ros-Oton and Xiong \cite{JOX24}.

	\begin{lem}\label{HI}
		Under the assumptions of Theorem \ref{GHPII}, there exists a constant $H>0$ depending only on $N$, $s$, $\gamma$, and $\Omega$ such that for any ball $B_R(x_0)\subset \Omega$ such that $B_{2R}\subset \Omega$ we have
		\begin{equation*}\label{LHec1}
			\sup_{x\in B_R(x_0)}\ph(x)\leq H \inf_{x\in B_R(x_0)} \ph(x).
		\end{equation*}
	\end{lem}
\medskip
	
	\begin{thm}[\textbf{Local Harnack Inequalities}]\label{LHI}
		Under the assumptions of Theorem \ref{GHPII}, there exists a constant $\widehat{H}>0$ depending only on $N$, $s$, $\gamma$, $m_i$, $F$, $c_1$, $k_\Omega$ and $\Omega$ such that for any ball $B_R(x_0)\subset \Omega$ with $B_{2R}\subset \Omega$, the following assertions hold:
		\renewcommand{\labelenumi}{\rm (\roman{enumi})}
		\begin{enumerate}
			\item If $\|u_0\|_{\LL^1_\ph(\Omega)}\leq 1$ and $0<t\leq t_*=c_*\|u_0\|_{\LL^1_\ph(\Omega)}^{1-m_1}$ the following inequality holds
			\begin{equation*}\label{LHI1}
				\sup_{x\in B_R(x_0)} F(u(t,x))\leq  \widehat{H} \frac{t_*^{\frac{m_1^2}{m_1-1}}}{t^{m_1+\frac{m_0}{m_0-1}}} \inf_{x\in B_R(x_0)} F(u(t,x)).
			\end{equation*}
			\item If $\|u_0\|_{\LL^1_\ph(\Omega)}\leq 1$ and $t\geq t_*=c_*\|u_0\|_{\LL^1_\ph(\Omega)}^{1-m_1}$ the following inequality holds
			\begin{equation*}\label{LHI2}
				\sup_{x\in B_R(x_0)} F(u(t,x))\leq \widehat{H} t^{\frac{m_0}{m_0-1} - \frac{m_1}{m_1-1}} \inf_{x\in B_R(x_0)} F(u(t,x)).
			\end{equation*}
			\item If $\|u_0\|_{\LL^1_\ph(\Omega)}\geq 1$ and $0<t\leq t_*=c_*\|u_0\|_{\LL^1_\ph(\Omega)}^{1-m_0}$ the following inequality holds
			\begin{equation*}\label{LHI3}
				\sup_{x\in B_R(x_0)} F(u(t,x)) \leq \widehat{H} \frac{t_*^{m_1}}{t^{m_1+\frac{m_0}{m_0-1}}} \inf_{x\in B_R(x_0)} F(u(t,x)).
			\end{equation*}
			\item If $\|u_0\|_{\LL^1_\ph(\Omega)}\geq 1$ and $t\geq t_*=c_*\|u_0\|_{\LL^1_\ph(\Omega)}^{1-m_0}$ the following inequality holds
			\begin{equation*}\label{LHI4}
				\sup_{x\in B_R(x_0)} F(u(t,x)) \leq \widehat{H} \frac{t^{\frac{m_0}{m_0-1}-\frac{m_1}{m_1-1}}}{t_*^{\frac{m_0}{m_0-1}}} \inf_{x\in B_R(x_0)} F(u(t,x)).
			\end{equation*}
		\end{enumerate}
	\end{thm}
	\begin{proof}
		We will only prove part (i) with full details, and the other cases follow with trivial changes, due to the different forms of the GHP involved. Using the GHP, i.e., inequalities \eqref{GHPineq2} and \eqref{GHPineq3}, together with the previous lemma we obtain the local, and more classical, forms of Harnack inequalities.
		
		Let $x_0\in \Omega$ and $R>0$ such that the balls $B_R(x_0)\subset \Omega$ and $B_{2R}(x_0)\subset \Omega$. Let $0<t\leq t_*$, thus,
		\begin{align*}
			\sup_{x\in B_R(x_0)} F(u(t,x))&\leq \overline{\kappa_4} t^{-\frac{m_1}{m_1-1}} \sup_{x\in B_R(x_0)} \ph(x)^{\sigma_1}\\
			&\leq \overline{\kappa_4} t^{-\frac{m_0}{m_0-1}} H^{\sigma_1} \inf_{x\in B_R(x_0)} \ph(x)^{\sigma_1}
			\leq H^{\sigma_1} \frac{\overline{\kappa_4}}{\underline{\kappa_3}} \frac{t_*^{\frac{m_1^2}{m_1-1}}}{t^{\frac{m_0}{m_0-1}+m_1}} \inf_{x\in B_R(x_0)} F(u(t,x)).
		\end{align*}
	\end{proof}
\begin{rem}\label{localharnack}\rm
	It is also possible to derive the previous inequalities in the Elliptic/Backward Type. Since $F(u(t,x))\leq (1+\frac{h}{t})^{\frac{m_0}{m_0-1}} F(u(t+h,x))$, by the time monotonicity of Lemma  \ref{benilancrandall}, we can obtain, for instance, for the first inequality
	$$\sup_{x\in B_R(x_0)} F(u(t,x))\leq  \widehat{H} \left(1+\frac{h}{t}\right)^{\frac{m_0}{m_0-1}} \frac{t_*^{\frac{m_1^2}{m_1-1}}}{t^{m_1+\frac{m_0}{m_0-1}}} \inf_{x\in B_R(x_0)} F(u(t+h,x)).$$
	
	These results are valid for $t\geq t_*$ thanks to Theorem \ref{GHPI}. As explained in \cite{MB+AF+JLV_sharp-global-estimates}, these Harnack inequalities (in the local case $s=1$) are stronger than the known two-sided inequalities valid for
	solutions to the Dirichlet problem for the classical porous medium equation, cf. \cite{Aron+Caffa, Daska+Kenig, DIB_intr-harnineq, DIB_deg-para-equa, DIB+GIA+VES_harnineq}, which are of forward type and are often stated in terms of the so-called intrinsic geometry. Note that elliptic and backward Harnack-type inequalities usually occur in the fast diffusion range $m<1$, see for instance \cite{MB+GG+JLV_behaviourFD,MB+JLV_globalposi-harn-FD, MB+JLV_posi-smooth-harn-FD, MB+JLV_quan-locandglo-aprioriestimates} (when we have the equation $u_t+\A u^m=0$), or for linear equations in bounded domains, cf. \cite{FAB+GAR+SAL_harnineq-fatou,Saf+Yuan}. It is also possible to obtain forward in time elliptic-type Harnack inequalities. These inequalities are classical and hold for sufficiently small $h<0$.
\end{rem}

	Now, let's discuss the suitable class of solutions used to prove regularity results. Throughout this section, we consider weak solutions of \eqref{CDP}, as defined in Definition \ref{defweak} below. This class of solutions is contained within the broader class of weak dual solutions. We would like to stress that all the regularity results that we prove for weak solutions, can be generalized to the wider class of weak dual solutions to \eqref{CDP}, by approximation.
	
	In \cite{MB+JLV_harnack-inequality}, the first author and V\'azquez construct weak dual solutions starting from weak solutions: more precisely, weak dual solutions should be interpreted as a class of limit solutions that satisfy a ``quasi'' $\LL^1_{\ph}$-contraction, that implies uniqueness of \solutions\,as $\LL^1_{\ph}$ function. Once we have proved regularity results for weak solutions, we only need to approximate a weak dual solution from below by means of a sequence of weak solutions: in this way we can extend all the following regularity results to weak dual solutions and, in particular, to \solutions. Indeed, all the constants in the regularity estimates are stable under this limit process, since they only depend on the supremum of $u$, which is controlled by the $\LL^1_{\ph}$ norm of the initial datum, as indicated by the smoothing effects.
	
	\begin{defn}[\textbf{Weak Solutions}]\label{defweak}
		 We say that a function $u$ is a weak solution of the problem \eqref{CDP} if it satisfies the following:
		\begin{itemize}
			\item $u\in C([0,\infty) : \LL^1(\Omega))$ and $F(u)\in \LL^2_{\rm loc} ((0,\infty) : H^s(\Omega))$.
			\item $u(0,\cdot)=u_0$ for a.e. point of $\Omega$.
			\item For any test function $\psi\in C_c^1((0,\infty) \times \Omega)$ we have
			\begin{equation}\label{weakec1}
				\int_0^\infty \int_{\Omega} u \partial_t \psi = \int_0^\infty \int_{\Omega} \AM F(u) \AM \Psi.
			\end{equation}
		\end{itemize}
	\end{defn}

	\subsection{Proof of Theorem \ref{interiorregularity}}
	
	Observe that the lower and upper bounds in the hypotheses for weak solutions of Theorem \ref{interiorregularity} are perfectly valid for \solutions, thanks to both the local and global Harnack inequalities for $F(u)$, and the monotonicity of $F$.

Our proof relies on a localization argument introduced in \cite{MB+AF+XR_positivity-and-regularity, MB+AF+JLV_sharp-global-estimates}, that we extend here to the more delicate case of a non homogeous nonlinearity $F$. This allows us to exploit the following linear result for (local) weak solutions, which shows that bounded weak solutions are indeed Hölder continuous.
	\begin{thm}[\textbf{Felsinger--Kassmann \cite{Fel+Kass}}]\label{Kass}
		Let $\A$ be an integro-differential operator of the form
		\begin{equation*}\label{Kassec1}
			\A_a v(t,x)= P.V \int_{\RN} (v(t,x)-v(t,y))\frac{a(t,x,y)}{|x-y|^{N+2s}} \dy.
		\end{equation*}
		Where the coefficient $a$ is symetric, i.e. $a(t,x,y)=a(t,y,x)$ and satisfies $\Lambda^{-1}\leq a(t,x,y)\leq \Lambda$ for some $\Lambda>1$. Let $v\in \LL^\infty((0,1)\times \RN)$ be a weak solution of the equation
		$$v_t + \A_a v=f \hspace{4mm}\mbox{for }(t,x)\in (0,1)\times B_1 \mbox{ and } f\in \LL^\infty((0,1)\times B_1).$$
		Then, there exists a small $\alpha>0$ and $C>0$ which depends only on $N$, $s$ and $\Lambda$ such that
		\begin{equation}\label{Kassec2}
			\|v\|_{C_{x,t}^{\alpha,\alpha/2s}((1/2,1)\times B_{1/2})}\leq C \left( \|f\|_{\LL^\infty((0,1)\times B_1)} + \|v\|_{\LL^\infty((0,1)\times \RN)} \right).
		\end{equation}
	\end{thm}
This Theorem is stated in \cite{Fel+Kass} with $f\equiv 0$, but the very same proof allows to deal with the apparently more general case of bounded $f$.

	\begin{proof}
		We split the proof into three steps.
		
		\textbf{Step 1: \textit{Localization of the problem.}} By rescaling, if needed, we may assume $x_0=0$ $r=2$, $T_0=0$, $T_2=1/2$ and $T_1=1$. Let $\rho\in C_c^\infty(B_4)$ be the first smooth cutoff function such that $\rho\equiv 1$ in $B_3$ and $\rho\equiv 0$ in $\RN\setminus B_4$. We define $v=\rho u$ and observe that
		$$\A F(v) = \A F(u) - \A [F(\rho u) - F(u)],$$
		where $F(\rho u) - F(u)=0$ in $B_3$ by definition of $\rho$. Moreover, for any $x\in B_3$ we have
		$$-\A [F(\rho u) - F(u)](x)= P.V \int_{\RN\setminus B_3} [F(\rho(y)u(y)) - F(u(y))]K(x,y)\dy,$$
		which belongs to $C^\infty(B_2)$. Therefore, the equation that $v$ satisfies is $v_t + \A F(v)=g(t,x)$ where g is a $C^\infty(B_2)$ function in x for fixed $t\in (0,1)$.
		
		Now, let $\eta\in C_c^\infty(B_2)$ be the second smooth cutoff function such that $\eta\equiv 0$ in $\RN \setminus B_2$ and $\eta\equiv 1$ in $B_1$. Then, for any $t\in (0,1)$ and $x\in B_1$, we write
		\begin{equation}\label{loc1}
			F(v(t,x)) - F(v(t,y))= (v(t,x)-v(t,y)) a(t,x,y) + h(t,x,y),
		\end{equation}
		where the coefficients are defined as follows:
		\begin{align*}
			a(t,x,y)&= \frac{F(v(t,x)) - F(v(t,y))}{v(t,x) - v(t,y)} \eta(x-y) +(1-\eta(x-y)).\\
			h(t,x,y)&= (1-\eta(x-y)) [(F(v(t,x)) - F(v(t,y))) - (v(t,x) - v(t,y))].
		\end{align*}
		
		\textbf{Step 2: \textit{H\"older continuity.}} The coefficient $a$ defined above can be expressed as
		$$a(t,x,y)= \eta(x-y) \int_0^1 F'(v(t,x)+ \lambda[v(t,y)-v(t,x)]) {\rm d }\lambda + (1-\eta(x-y)).$$
		For the first term to be non-zero, we require $|x-y|<2$. Given that $x\in B_1$, this implies $y\in B_3$. Consequently, all functions $v$ are essentially $u$, for which we have the bound $0<\delta\leq u\leq M$ in $(0,1)\times B_4$. This allows us to conclude that $\Lambda^{-1}\leq a(t,x,y)\leq \Lambda$, because $F'(r) \asymp F(r)/r$ for $r>0$, which is increasing, and $v(t,x)+ \lambda[v(t,y)-v(t,x)]$ is a convex sum. Multiplying \eqref{loc1} by $K(x,y)$ and integrating in the entire space with respect to $y$, we obtain
		$$\A F(v)= \A_a v +f, \hspace{4mm}\mbox{with}\hspace{2mm}f(t,x)=\int_{\RN} h(t,x,y) K(x,y)\dy + B(x)F(v(t,x)).$$
		Thus, recalling that $v_t +\A F(v)=g$, we also have the following equation for $v$
		\begin{equation*}
			v_t+ \A_a v = g-f.
		\end{equation*}
		In order to apply Theorem \ref{Kass} and conclude that $v$ is Hölder continuous in the interior, we need to show that $f$ is bounded. It is clear that the term $B(x)F(v(t,x))$ is bounded for $(t,x)\in (1/2,1)\times B_1$, this is why we asked $B(x)$ to be bounded. By the definition of $h$, if $|x-y|\leq1$, then $h=0$ because $\eta(x-y)=1$. Moreover, if $y\in \RN\setminus B_4$, then $v=\rho u=0$ and thus $h=0$ as well. Summing up, $f$ is non-zero for $x\in B_1$ and $t\in (0,1)$ fixed, only when $y\in B_4$ with $|x-y|>1$. In this scenario, using the bound for $u$, which are also valid for $v$, we conclude that $|h(t,x,y)|\leq c$ and therefore $f$ is bounded. Consequently, Theorem \ref{Kass} guarantees that there exists $\alpha>0$ such that
		$$\|v\|_{C_{x,t}^{\alpha, \alpha/2s}((1/2,1)\times B_{1/2})}\leq C \left( \|g-f\|_{\LL^\infty((0,1)\times B_1)} + \|v\|_{\LL^\infty((0,1)\times B_1)} \right).$$
		Since $u=v$ in the relevant balls and $0\leq u \leq M$, we obtain
		\begin{equation*}\label{loc2}
			\|u\|_{C_{x,t}^{\alpha,\alpha/2s}((1/2,1)\times B_{1/2})}\leq C.
		\end{equation*}
		
		\textbf{Step 3: \textit{Solutions are Classical in the Interior.}} In Step 2, we established that there exists an $\alpha>0$ such that $u\in C_{x,t}^{\alpha, \alpha/2s}((1/2,1)\times B_{1/2})$. Similar to step 1, we use cutoff smooth functions $\rho$ and $\eta$, supported within $((1/2,1)\times B_{1/2})$, to  ensure that $v=\rho u$ is $\alpha$-Hölder continuous in $(1/2,1)\times \RN$.
		
		Now, let $\beta_1=(\alpha \land \beta)$ (where $\beta$ is the parameter in assertion ii) ) and define $K_a(t,x,y)= a(t,x,y) K(x,y)$. Given the assumptions on the kernel $K$ and the Hölder continuity, we can verify that
		$$|K_a(t,x,y) -K_a(t',x',y)|\leq C (|x-x'|^{\beta_1} + |t-t'|^{\beta_1/2s}) |y|^{-(N+2s)},$$
		inside $((1/2,1)\times B_{1/2})$. Futhermore, since $f,g\in C_{x,t}^{\beta_1, \beta_1/2s}((1/2,1)\times B_{1/2})$, we can apply the Schauder estimates from \cite{Dong_schauder} to obtain
		$$\|v\|_{C_{x,t}^{2s+\beta_1, 1+\beta_1/2s}((3/4,1)\times B_{1/4})}\leq C \left(\|g-f\|_{C_{x,t}^{\beta_1, \beta_1/2s}((1/2,1)\times B_{1/2})} + \|v\|_{C_{x,t}^{\beta_1, \beta_1/2s}((1/2,1)\times \RN)}\right).$$
		In particular, this implies $u\in C_{x,t}^{2s+\beta_1, 1+\beta_1/2s}((3/4,1)\times B_{1/4})$. If $\beta_1=\beta$, we are done. If $\beta_1=\alpha$ we know that $u\in C_{x,t}^{2s+\alpha, 1+\alpha/2s}((3/4,1)\times B_{1/4})$. We then set $\alpha_1=2s+\alpha$, $\beta_2=(\alpha_1 \land \beta)$ and repeat the process with $\beta_2$ instead of $\beta_1$. This allows us to conclude $u\in C_{x,t}^{2s+\beta_2, 1+\beta_2/2s}((1-2^{-4},1)\times B_{2^{-5}})$, and by iterating this process, we eventually obtain
		$$u\in C_{x,t}^{2s+\beta, 1+\beta/2s}((1-2^{-k},1)\times B_{2^{-k-1}}).$$
		Finally, a standard covering argument allows to complete the proof.
	\end{proof}
	
	%%%%%%%%%%%%%%%%%%%%%%%%%%%%%%%%%%%%%%%%%%%%%%%%%%%%%%%
	%%%%%%%%%%%%%%%%%%%%%%%%%%%%%%%%%%%%%%%%%%%%%%%%%%%%%%%
	%%%%%%%%%%%%%%%%%%%%%%%%%%%%%%%%%%%%%%%%%%%%%%%%%%%%%%%
	%%%%%%%%%%%%%%%%%%%%%%%%%%%%%%%%%%%%%%%%%%%%%%%%%%%%%%%	
	
	\section{Some examples of operators.}\label{sec7}
	
	The fractional Laplacian $(-\Delta)^s$ posed on the whole space $\RN$, can be defined in several equivalent ways, at least three . On the other hand, as noticed in \cite{MB+AF+JLV_sharp-global-estimates, MB+JLV+YS_exist-and-unique, MB+JLV_PM-with-F} such definitions happen to generate three different operators on bounded domains, as we recall below. These three definitions are the prototypes of the wide class of operators that we treat in this work.
	\begin{itemize}[leftmargin=*]
		\item \textbf{Restricted Fractional Laplacian (RFL)}: starting from the definition in the whole space of the fractional laplacian
		\begin{equation}\label{RFL}
			(-\Delta)^s f(x)= C_{N,s} P.V \int_{\RN} \frac{f(x)- f(y)}{|x-y|^{N+2s}} \dy,
		\end{equation}
		where $C_{N,s}$ is an explicit positive constant. We define the RFL on Omega $\A=(-\Delta|_\Omega)^s$ as the operator above restricted to the functions which vanish outside $\Omega$. In this case $\A$ satisfies \eqref{A1}, \eqref{A2}, \eqref{L1}, \eqref{K2} and \eqref{K4} with $\gamma=s\in (0,1)$ and $\sigma_1=1$, see \cite{JAK_greenestimates,KULC_properties-green}
		\item \textbf{Censored Fractional Laplacian (CFL)}: It is the infinitesimal operator of the censored stochastic processes introduced in \cite{Bogdan-1} and it has the following expression
		\begin{equation}\label{CFL}
			(-\Delta)^s f(x)= P.V \int_{\Omega} \frac{f(x)- f(y)}{|x-y|^{N+2s}} \dy,\hspace{3mm} \mbox{with }s\in (1/2,1).
		\end{equation}
In this case $\A$ satisfies \eqref{A1}, \eqref{A2}, \eqref{L1}, \eqref{K2} and \eqref{K4} with $\gamma=2s-1$ and $\sigma_1=1$, see \cite{Chen}.
		\item \textbf{Spectral Fractional Laplacian (SFL)}: Starting from the classical Dirichlet Laplacian on domains $(-\Delta_\Omega)$, we denote by $\left\{ \lambda_j \right\}_{j=1}^\infty$ the eingenvalues of this operator in non-decreasing order (repeated according to their multiplicity) and by $\left\{ \phi_j \right\}_{j=1}^\infty$ the sequence of eigenfunctions normalized in $\LL^2(\Omega)$. In this way, the spectral power of $(-\Delta_\Omega)$ can be expressed in terms both in terms of Fourier series and of the associated heat semigroup:
		\begin{align}\label{SFL}
			(-\Delta_\Omega)^s f(x)&= \frac{1}{\Gamma(-s)}\int_0^\infty \frac{e^{t \Delta_\Omega} f(x) - f(x)}{t^{1+s}} \dt
			= \sum_{j=1}^{\infty} \lambda_j^s \widehat{f_j} \phi_j(x),\nonumber
		\end{align}
		where $\widehat{f_j}=\int_{\Omega} f(x)\phi_j(x) \dx$.
		
		The SFL operator can also be written as the convolution with a kernel plus a zero-order term as follows, see for instance \cite{Abat_thesis, Song+Vondracek-2003}

		\begin{equation}\label{SFL2}
			(-\Delta_\Omega)^s f(x)= P.V \int_{\Omega} [f(x)-f(y)] K(x,y) \dy + B(x)f(x).
		\end{equation}
		Here, $K$ is a singular kernel with compact support and $B\asymp d^{2s}$. This operator satisfies \eqref{A1}, \eqref{A2}, \eqref{L2}, \eqref{K2}, and \eqref{K4}, with $\gamma=1$ and $\sigma_1=\left( 1 \land \frac{2s m_1}{m_1-1} \right)$, see \cite{DAV_equiv-heat-green, DAV_explicitconstants, DAV_heatkernel-spectral, DAV_spectral-theory, DAV_ultracontractivity, Zhang2002}.

		\item \textbf{Further examples:}  Let's show several other relevant examples to which this theory applies. Firstly, consider operators of the form
		$$\A f(x)= P.V \int_\Omega (f(x)-f(y)) \frac{a(x,y)}{|x-y|^{N+2s}}\dy,\hspace{4mm}\mbox{with }s\in (1/2,1).$$
		Where $a$ is a symetric $C^1$ function bounded between two positive constants. In this case, $\gamma = 2s-1$ and the Green function $\K(x,y)$ satisfies \eqref{K4}; see \cite{Bogdan-1, Chen}.
		
		Another example includes operators with more general kernels, such as integral operators of the form
		$$\A f(x)= P.V \int_\Omega (f(x)-f(y)) \frac{a(x,y)}{|x-y|^{N+2s}}\dy.$$
		Here, the coefficient $a$ is a measurable symmetric function, bounded between two positive constants, and satisfying $|a(x,y)-a(x,x)| \chi_{|x-y|<1}\leq c |x-y|^\sigma$ for $0<s<\sigma\leq 1$. In this scenario, $s\in (0,1]$, $\gamma=s$ and the Green function satisfies \eqref{K4}; see \cite{Kim}.
		
	Finally, our last example consists of spectral powers of uniformly elliptic operators. Consider a linear operator in divergence form with uniformly elliptic $C^1$ coefficients $\mathcal{A}=-\sum_{i,j=1}^{N} \partial_i(a_{i,j}\partial_j)$. We can construct a self-adjoint operator on $\LL^2$, with discrete spectrum $(\lambda_k, \phi_k)$. Thanks to the ellipticity and the spectral theorem, we can construct the spectral power $s\in (0,1)$ of such an operator
		$$\A f(x)=\mathcal{A}^s f(x)=\sum_{k=1}^{\infty} \widehat{f_k} \lambda_k^s \phi_k(x),$$
		where $\widehat{f_k}=\int_{\Omega} f(x) \phi_k(x)\dx$. In this case, $\gamma=1$ and the Green function satisfies \eqref{K4}, see \cite{KIM+SONG+VOND_twosides-greenestimates}.
	\end{itemize}
	
	More examples can be found in \cite[Section 3.3]{MB+JLV_PM-with-F}, see also \cite{MB+AF+JLV_sharp-global-estimates}.
	
%%%%%%%%%%%%%%%%%%%%%%%%%%%%%%%%%%%%%%%%%%%%%%%%%%%%%%%%
%%%%%%%%%%%%%%%%%%%%%%%%%%%%%%%%%%%%%%%%%%%%%%%%%%%%%%%%
 \newpage

\addcontentsline{toc}{section}{Acknowledgments}
\noindent \textbf{Acknowledgments.}  Both authors were partially supported by the Projects PID2020-113596GB-I00 and PID2023-150166NB-I00 (Spanish Ministry of Science and Innovation) and acknowledge the financial support from the Spanish Ministry of Science and Innovation, through the ``Severo Ochoa Programme for Centres of Excellence in R\&D'' (CEX2019-000904-S and CEX2023-001347-S) and by the European Union’s Horizon 2020 research and innovation programme under the Marie Sk\l odowska-Curie grant agreement no. 777822. The second author was also supported by the Project PID2021-127105NB-100 and acknowledge the financial support from the Spanish Ministry of Science and Innovation.

\noindent\textbf{Conflict of interest statement. }On behalf of all authors, the corresponding author states that
there is no conflict of interest.

\noindent\textbf{Data availability statement. }All data generated or analysed during this study are included in
the published article.
	
	%%%%%%%%%%%%%%%%%%%%%%%%%%%%%%%%%%%%%%%%%%%%%%%%%%%%
	%%%%%%%%%%%%%%%%%%%%%%%%%%%%%%%%%%%%%%%%%%%%%%%%%%%%
	%%%%%%%%%%%%%%%%%%%%%%%%%%%%%%%%%%%%%%%%%%%%%%%%%%%%
	%%%%%%%%%%%%%%%%%%%%%%%%%%%%%%%%%%%%%%%%%%%%%%%%%%%%

\end{document}